\documentclass[11pt]{article}

\usepackage[utf8]{inputenc}
\usepackage[T1]{fontenc}
\usepackage[english]{babel}
\usepackage{bbold}
\usepackage{mathtools}
\usepackage{fullpage}
\usepackage{graphicx}
\usepackage{amsmath}
\usepackage{amssymb}
\usepackage{comment}
\usepackage{tablefootnote} % for footnotes inside a table
\usepackage[title]{appendix}
\usepackage{microtype}
\usepackage[export]{adjustbox}
\usepackage{FleuratSalvy_arxiv_style}

\usepackage{sectsty}
\paragraphfont{\sffamily}
\allowdisplaybreaks %to allow page breaks inside equations

\DeclarePairedDelimiter\ceil{\lceil}{\rceil}
\DeclarePairedDelimiter\floor{\lfloor}{\rfloor}

\renewcommand{\leq}{\leqslant}
\renewcommand{\geq}{\geqslant}
\renewcommand{\(}{\left(}
\renewcommand{\)}{\right)}

\newcommand{\vect}[1]{\overrightarrow{#1}}

\renewcommand{\epsilon}{\varepsilon}

\newcommand{\N}{\mathbb{Z}_{\geq0}}
\newcommand{\Z}{\mathbb{Z}}
\newcommand{\R}{\mathbb{R}}

\newcommand{\m}{\mathfrak{m}}
\renewcommand{\b}{\mathfrak{b}}
\renewcommand{\t}{\mathfrak{t}}
\newcommand{\q}{\mathfrak{q}}

\newcommand{\M}{\mathbf{M}}
\newcommand{\B}{\mathbf B}
\newcommand{\T}{\mathbf{T}}
\newcommand{\Q}{\mathbf{Q}}

\newcommand{\pr}[2][]{\mathbb{P}#1\( #2\)}
\newcommand{\E}[1]{E\left[ #1 \right]}

\newcommand{\prnu}[1]{\mathbb{P}_{n,u}\( #1\)}

\newcommand{\diam}{\mathrm{diam}}
\newcommand{\indic}[1]{\mathbb{1}_{\left\{#1\right\}}}
\newcommand{\indicBis}[1]{\mathbb{1}_{#1}}

\newcommand{\xic}{{\widehat\xi}}

\title{A phase transition in block-weighted random maps}
\author{William Fleurat\footnote{ENS de Lyon, UMPA, CNRS UMR 5669, 46 allée d’Italie, 69364 Lyon Cedex 07, France} \qquad Zéphyr Salvy\footnote{Univ Gustave Eiffel, CNRS, LIGM, F-77454 Marne-la-Vallée, France}}

\begin{document}

\maketitle

% !TEX root = FleuratSalvy.tex

\begin{abstract}
We consider the model of random planar maps of size $n$ biased by a weight $u>0$ per $2$-connected block, and the closely related model of random planar quadrangulations of size $n$ biased by a weight $u>0$ per simple component. We exhibit a phase transition at the critical value $u_C=9/5$. If $u<u_C$, a condensation phenomenon occurs: the largest block is of size $\Theta(n)$. Moreover, for quadrangulations we show that the diameter is of order $n^{1/4}$, and the scaling limit is the Brownian sphere. When $u > u_C$, the largest block is of size $\Theta(\log(n))$, the scaling order for distances is $n^{1/2}$, and the scaling limit is the Brownian tree. Finally, for $u=u_C$, the largest block is of size $\Theta(n^{2/3})$, the scaling order for distances is $n^{1/3}$, and the scaling limit is the stable tree of parameter $3/2$.
\end{abstract}

% !TEX root = FleuratSalvy.tex

\section{Introduction}

%%%%%%%%%%%%%%%%%%%%%%%%%%%% Existant %%%%%%%%%%%%%%%%%%%%%%%%%%%%%%%%

Models of planar maps exhibit a form of \emph{universality}: many ``natural'' classes of random maps exhibit a similar behaviour when the size grows to infinity. This can be made precise by considering \emph{scaling limits}: when taking an object $\M_n$ uniformly among all objects of size $n$ in some class, then, after an appropriate rescaling, the sequence $(\M_n)_{n\geq1}$ converges in distribution towards some random metric space. This was first proved for uniform quandrangulations by Miermont \cite{Miermont2013} and independently for the cases of uniform $2q$-angulations ($q\geq2$) and uniform triangulations by Le Gall \cite{LeGall2013}, following a sequence of results on this subject \cite{marckert_depth_2003,ChassaingSchaeffer2004,LeGall2007,LG10,LGM11}. Since then, these results have been extended to other families of maps: the sequence $(\M_n)$ converges towards the \emph{Brownian sphere} $\mathcal{M}_e$ (also called Brownian map, see \cref{img-brownian-sphere}), always with a rescaling by $c n^{1/4}$ for some model-dependent $c>0$. Gromov-Hausdorff's topology allows to make sense of the convergence of a sequence of maps to a certain limit, considering them as (isometry classes of) compact metric spaces. In particular, uniform planar maps also converge towards the Brownian sphere \cite{brownianmap}, as well as other families such as uniform triangulations and uniform $2q$-angulations ($q\geq2$) \cite{LeGall2013}, uniform simple triangulations and uniform simple quadrangulations \cite{add-alb}, bipartite planar maps with a prescribed face-degree sequence \cite{Marzouk2018}, $(2q+1)$-angulations \cite{AddarioBerryAlbenque2021} and Eulerian triangulations~\cite{Carrance2021}.

\begin{figure}
\begin{center}
\includegraphics[width=16.5cm,center]{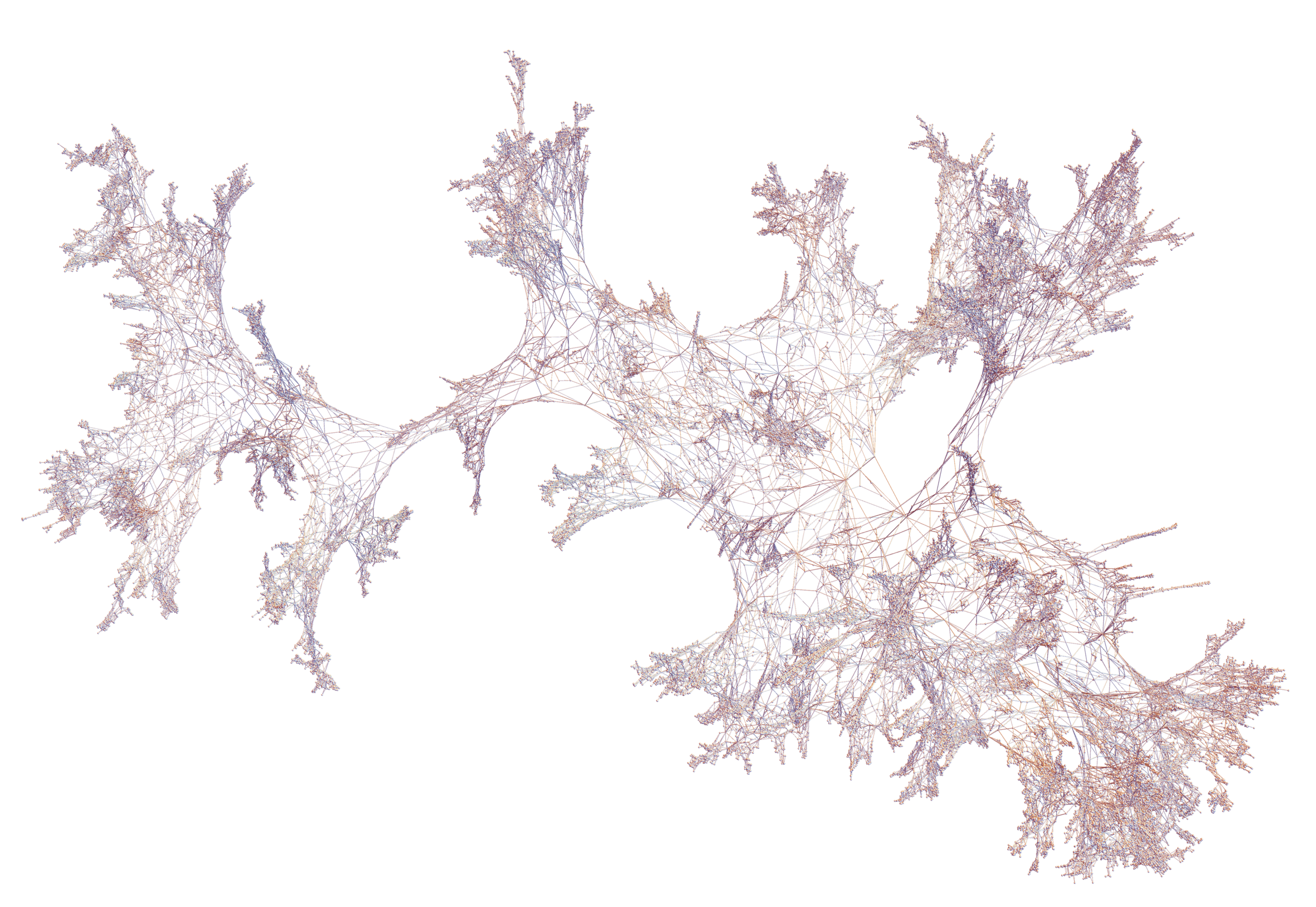}
\end{center}
\caption{Approximation of the Brownian sphere by a simple quadrangulation of size $50\ 000$, using a generator by Éric Fusy.}
\label{img-brownian-sphere}
\end{figure}

On the other hand, ``degenerate'' classes of maps that ``look like'' trees exhibit another universality phenomenon. In particular, upon rescaling by $cn^{1/2}$, there is a convergence to the \emph{Brownian tree} $\mathcal{T}^{(2)}$ (see \cref{img-brownian-tree}), the scaling limit of critical Galton-Watson trees with finite variance \cite{CRT,legallcrt}. This is the case for classes of maps with a tree-decomposition such as stack triangulations \cite{AlbenqueMarckert2008}; classes of maps with some particular boundary conditions, such as quadrangulations of a polygon \cite{Bettinelli2015}, outerplanar maps \cite{Caraceni2016}; or, more generally for ``subcritical'' classes \cite{Stufler-survey-2020} (see \cite{3graphCRT} for the case of graphs).

\begin{figure}
\begin{center}
\includegraphics[width=16.5cm, center]{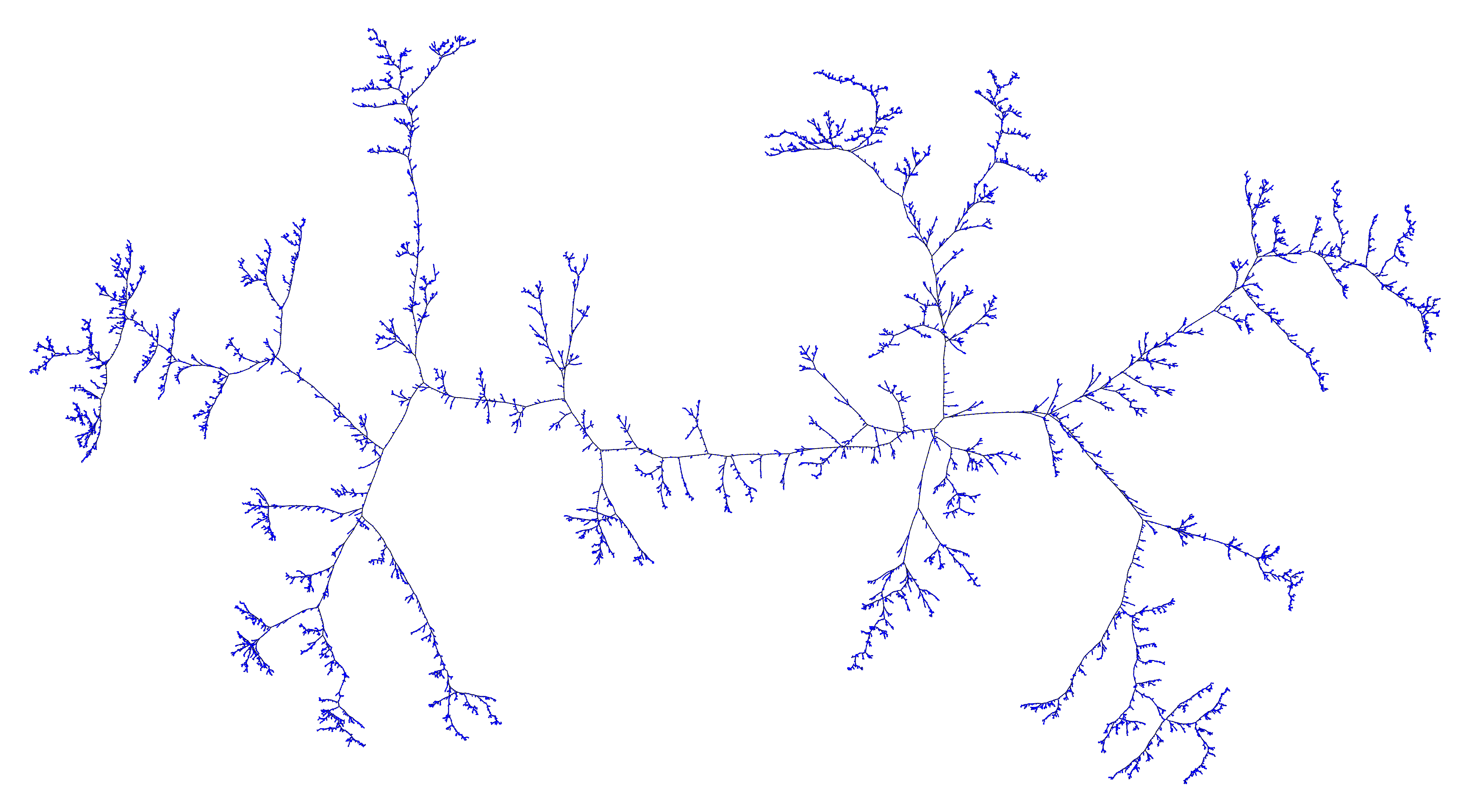}
\end{center}
\caption{Approximation of the Brownian tree by a binary tree of size approximately $70\ 000$.}
\label{img-brownian-tree}
\end{figure}

Models interpolating between the Brownian tree and the circle can be obtained by using \emph{looptrees} \cite{CurienKortchemski2013a}. Curien and Kortchemski considered the boundary of large percolation clusters in the \emph{uniform infinite planar triangulation} (which is the local limit of large triangulations) where each vertex is coloured (independently) white with probability $a \in(0,1)$ and black otherwise. They showed that if $a \in (0,1/2)$, the scaling limit is the Brownian tree, if $a \in (1/2, 1)$ it is the unit circle and if $a=1/2$ it is the \emph{stable looptree of parameter $3/2$} \cite{CurienKortchemski2013b}, which correspond to the \emph{stable tree of parameter $3/2$} (see \cref{img-stable-tree}) where each branching point is replaced by a circle. Richier \cite{Richier2018} also showed that the boundary of critical Boltzmann planar maps with face-degrees in the domain of attraction of a stable distribution with parameter $\alpha \in (1,2]$ exhibit a similar phase transition: if $\alpha \in (1,3/2)$, the scaling limit is the stable looptree of parameter $(\alpha-1/2)^{-1}$, and, with Kortchemski, Richier showed that it is the circle of unit length if $\alpha\in(3/2,2]$ and conjectured that this holds also for $\alpha=3/2$ \cite{KortchemskiRichier2020}. Stef\'ansson and Stufler showed that face-weighted outerplanar maps have a similar phase diagram, with looptrees the $\alpha$-stable looptree being the scaling limit when their $\alpha\in(1,2)$, the Brownian tree when $\alpha=2$ and the deterministic circle of unit length when $\alpha=1$ \cite{SS19}. In all three cases, the parameter of the model allows the number of cut vertices appearing on the boundary to be adjusted, thus changing from a ``round'' to a ``tree'' phase.

\begin{figure}
\begin{center}
\includegraphics[width=16.5cm, center]{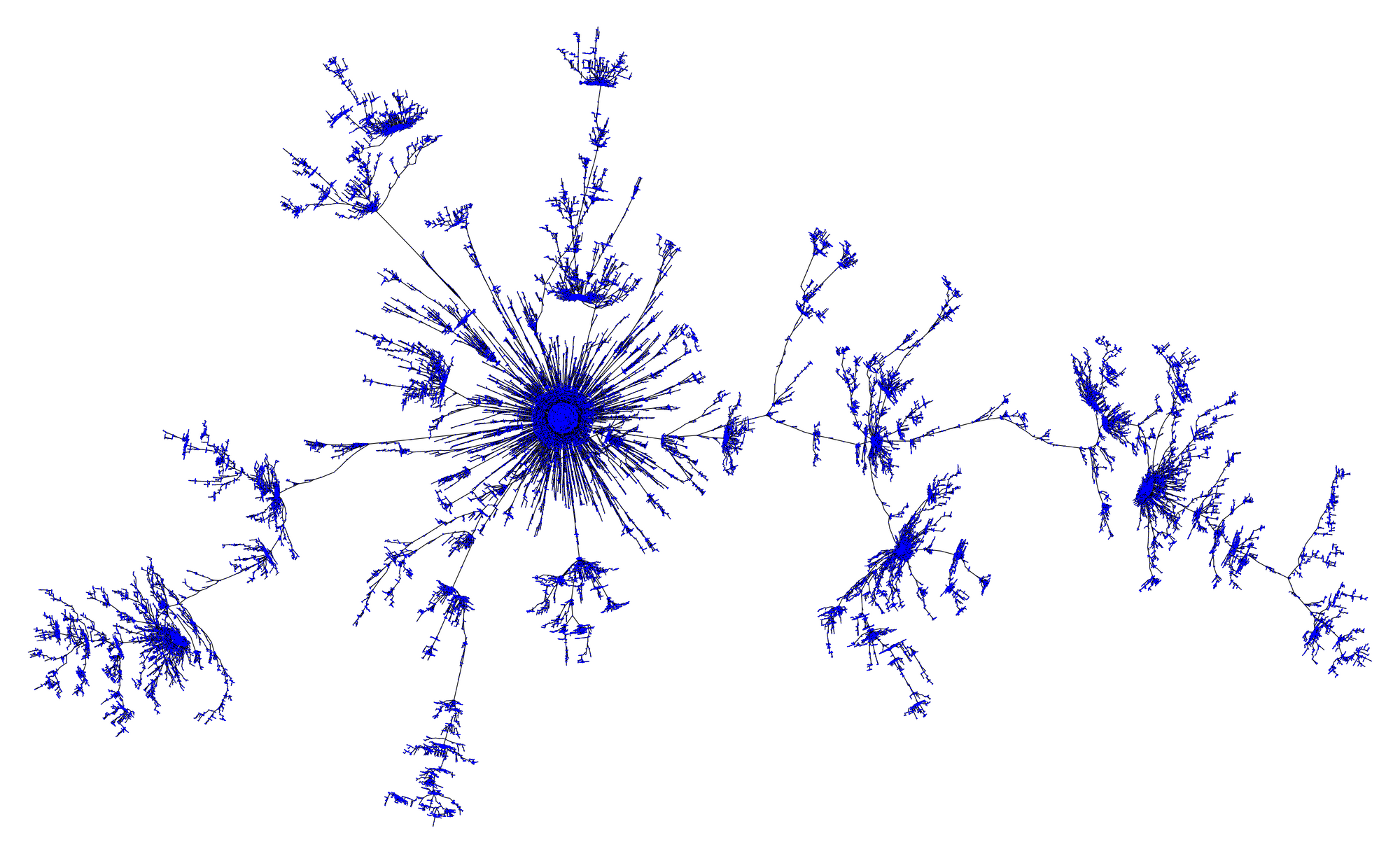}
\end{center}
\caption{Approximation of the stable tree $3/2$ by a tree of size approximately $150\ 000$.}
\label{img-stable-tree}
\end{figure}

Some natural models also interpolate between the Brownian sphere and the Brownian tree. For example, consider random quadrangulations with $n$ faces and a boundary of length $\ell$, where $\ell/\sqrt{n}\to \sigma$. When $\sigma=0$, the scaling limit is the Brownian sphere, when $\sigma=\infty$ it is the Brownian tree, and for all $\sigma \in(0,\infty)$ it is the \emph{Brownian disk} with boundary length $\sigma$ \cite{Bettinelli2015}. Another example is random bipartite planar maps with properly normalized face-weights, which converge towards the Brownian tree when the distribution for the weights has expected value smaller than $1$ \cite{JanSte15}, and towards the
 Brownian map when the expected value is $1$ and the variance is finite \cite{Marzouk2018}. Moreover, when the the expected value is $1$ and the distribution is in the domain attraction of a stable law of parameter $\alpha\in(1,2)$, these maps converge, at least along suitable subsequences, towards a limit which is not the Brownian sphere, and is conjectured to be the \emph{stable map} of parameter $\alpha$ \cite{LGM11}.

%%%%%%%%%%%%%%%%%%%%%%%%%%%% Model %%%%%%%%%%%%%%%%%%%%%%%%%%%%%%%%

\paragraph{Model.} The purpose of this paper is to propose yet another model interpolating between the Brownian sphere and the Brownian tree, but where the transition does not appear through the boundary. It relies on a parameter tuning the density of separating elements. In this model, a map $\m$ is sampled with a probability that depends on its number $b(\m)$ of maximal 2-connected components, or ``blocks'', for which a precise definition will be given later in \cref{sec:tree-decomp}.

\sloppy In fact, we will consider two probability distributions on maps, both indexed by a parameter $u>0$. The first one is a fixed size model: for any $n\in \N$, we define
\begin{equation}\label{eq:defPrnu}
\prnu{\m} = \frac{u^{b(\m)}}{[z^n] M(z,u)}\quad\text{for any }\m \in \mathcal{M}_n,
\end{equation}
where $\mathcal{M}_n$ is the set of maps with $n$ edges, $M(z,u) \allowbreak=\allowbreak \sum_{\m \in\mathcal{M}} u^{b(\m)} z^{|\m|} \allowbreak=\allowbreak \sum_{n\in\N} \([z^n]M(z,u)\) z^{n}$ and $|\m|$ is the number of edges of $\m$. The second one is a Boltzmann-type distribution which samples maps with random sizes. More precisely, write $\rho(u)$ for the radius of convergence of $z\mapsto M(z,u)$. We define\footnote{The finiteness of $M(\rho(u),u)$ is justified in \cref{decomp-blocks}.}:
\begin{equation}\label{eq:defPru}
\pr[_u]{\m}=\frac{u^{b(\m)}}{M(\rho(u),u)}\rho(u)^{|\m|}\quad\text{for any }\m \in \mathcal{M}.
\end{equation}

The qualitative properties of maps sampled according to these measures change drastically when $u$ varies, and we will see that it gives rise to different regimes with a phase transition. Examples of such maps are represented on \cref{fig:simul-sub-1-small,fig:simul-sub-8/5-small,fig:simul-crit-9/5-small,fig:simul-super-5/2-small,fig:simul-super-5-small}. In this paper, blocks will be either maximal 2-connected components of maps, or maximal simple components of quadrangulations. Indeed, both models have the same underlying structure, so one study gives results for both (see \cref{subseq:quad,subsec:conseq-proba}), except for some of the scaling limit results, where some convergence results for 2-connected maps are missing. However, our approach could be generalised to many other models with an underlying tree structure (see \cref{tab:values-models-airy}), such as the ones described in \cite{airy}. In particular, the case $u=1$ corresponds to sampling a uniform map and $u\to0$ to sampling a uniform block.

Block decompositions have already been used in the context of scaling limits, and some joint convergences are known: a quadrangulation, its largest 2-connected block, and its largest simple block jointly converge to the same Brownian sphere \cite{quadrang-coeur-cv}.

The scaling limit of a tree-decomposed model like ours depends on the geometries of the blocks and of the underlying decomposition tree. In our setting, one of the behaviour always ends up dominating, but this is not always the case: Sénizergues, Stef\'ansson and Stufler study situations where both geometries play a role in the scaling limit, and define the \emph{decorated $\alpha$-stable trees} which are the corresponding scaling limits \cite{SSS22}. Our results for the scaling limits in the critical and supercritical cases confirm their conjecture in \cite[Remark 1.1]{SSS22}. They build on a model introduced by Archer, which, contrary to this work, develops the \emph{local limit} point of view \cite[Chapter 6]{archer-random-walks}. In particular, Archer shows that the fractal dimension of the local limit for the critical and supercritical cases are respectively $3$ and $2$\footnote{This uses that the diameter for uniform blocks is $\Theta\(n^{1/4}\)$, which is known for simple quadrangulations but only assumed for $2$-connected maps.}. Both cases correspond to what Archer called the ``tree regime'', where the local geometry of the tree is preponderant in the limit. Both articles consider only \emph{critical} offspring distributions for the trees, which does not hold in our subcritical regime.

The model with a weight per 2-connected blocks was already analysed with a combinatorics point of view by Bonzom \cite[\S8]{bonzomLagrange} with physical applications in mind (see \cite{DuplantierSheffield2009} for a thorough discussion). The so-called \emph{quadric} model studied in his work can be specialized to our model. Bonzom obtains rational parametrisations for the generating series, and exhibits the possible singular behaviours, which suggest the existence of three different regimes: a ``map behaviour'', a ``tree-behaviour'', and in-between a ``proliferation of baby universes''. Since his focus is much broader, he does not go into details to study this particular model from a probabilistic point of view, and this is the main topic of the present article. For $u=1$, which corresponds to sampling maps uniformly, this model has also been studied with the point of view of block decomposition in \cite{airy} and \cite{2Louigi}.

% !TEX root = FleuratSalvy.tex

\begin{figure}
    \centering
    \begin{minipage}{6.5cm}
        \centering
        \includegraphics[height=5cm]{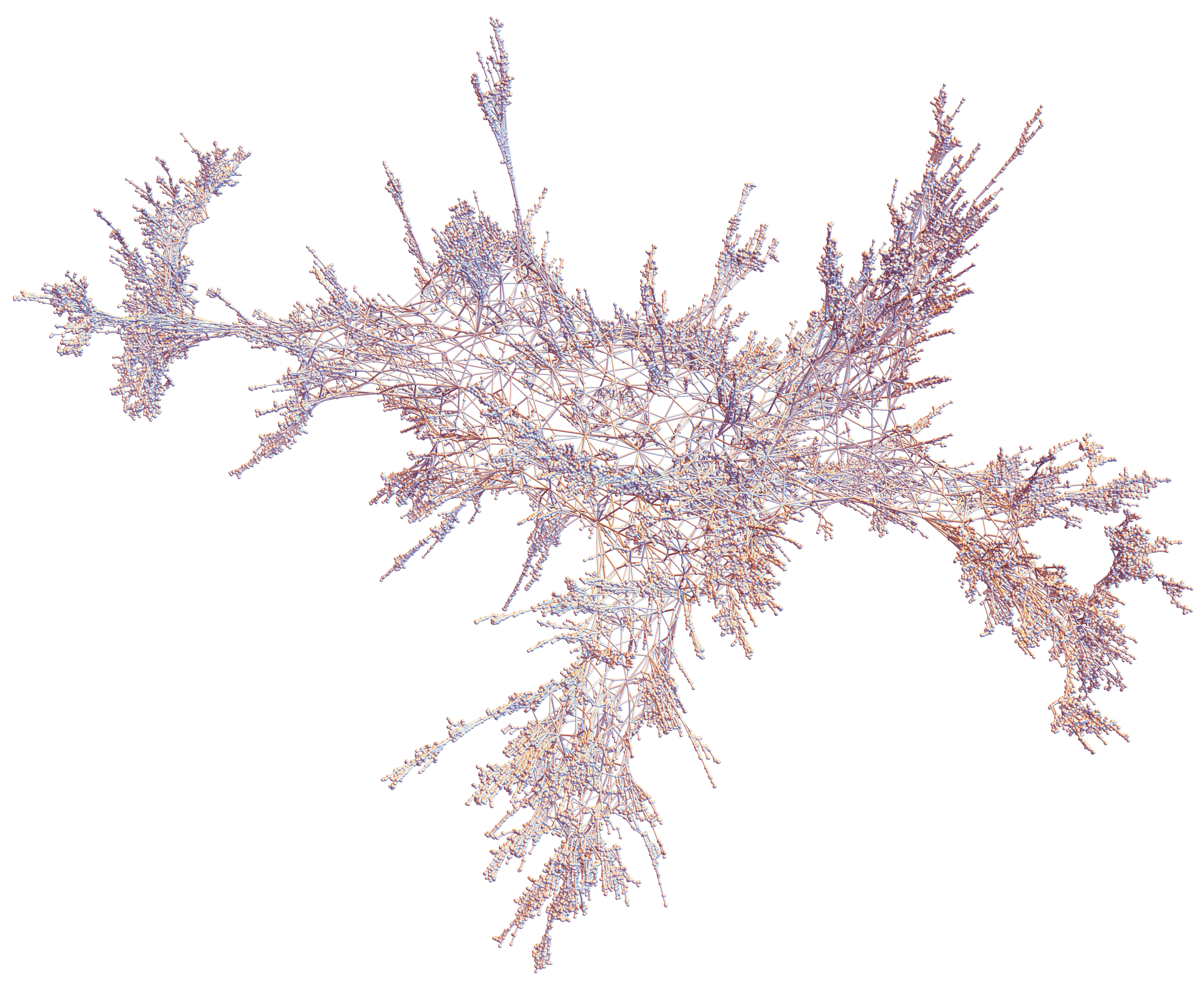}
        \caption{Map drawn according to the subcritical model $\mathbb{P}_{n,1}$ of size around $55\ 000$ (see larger version in \cref{fig:simul-sub-1}).}
        \label{fig:simul-sub-1-small}
    \end{minipage}\hfill
    \begin{minipage}{6.5cm}
        \centering
        \includegraphics[height=5cm]{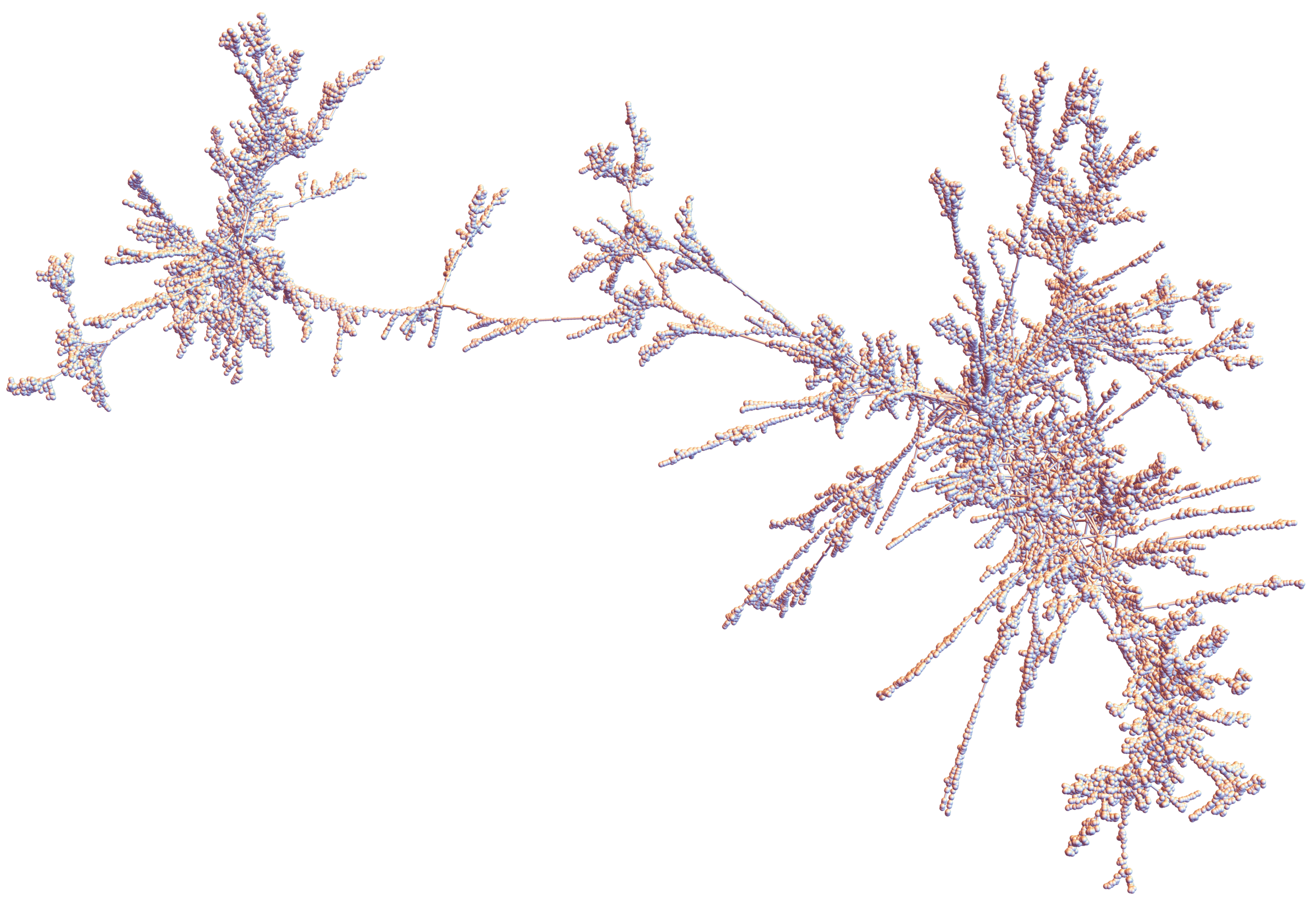}
        \caption{Map drawn according to the subcritical model $\mathbb{P}_{n,8/5}$ of size around $55\ 000$ (see larger version in \cref{fig:simul-sub-8/5}).}
        \label{fig:simul-sub-8/5-small}
    \end{minipage}
\end{figure}

\begin{figure}
\begin{center}
\includegraphics[height=5cm, center]{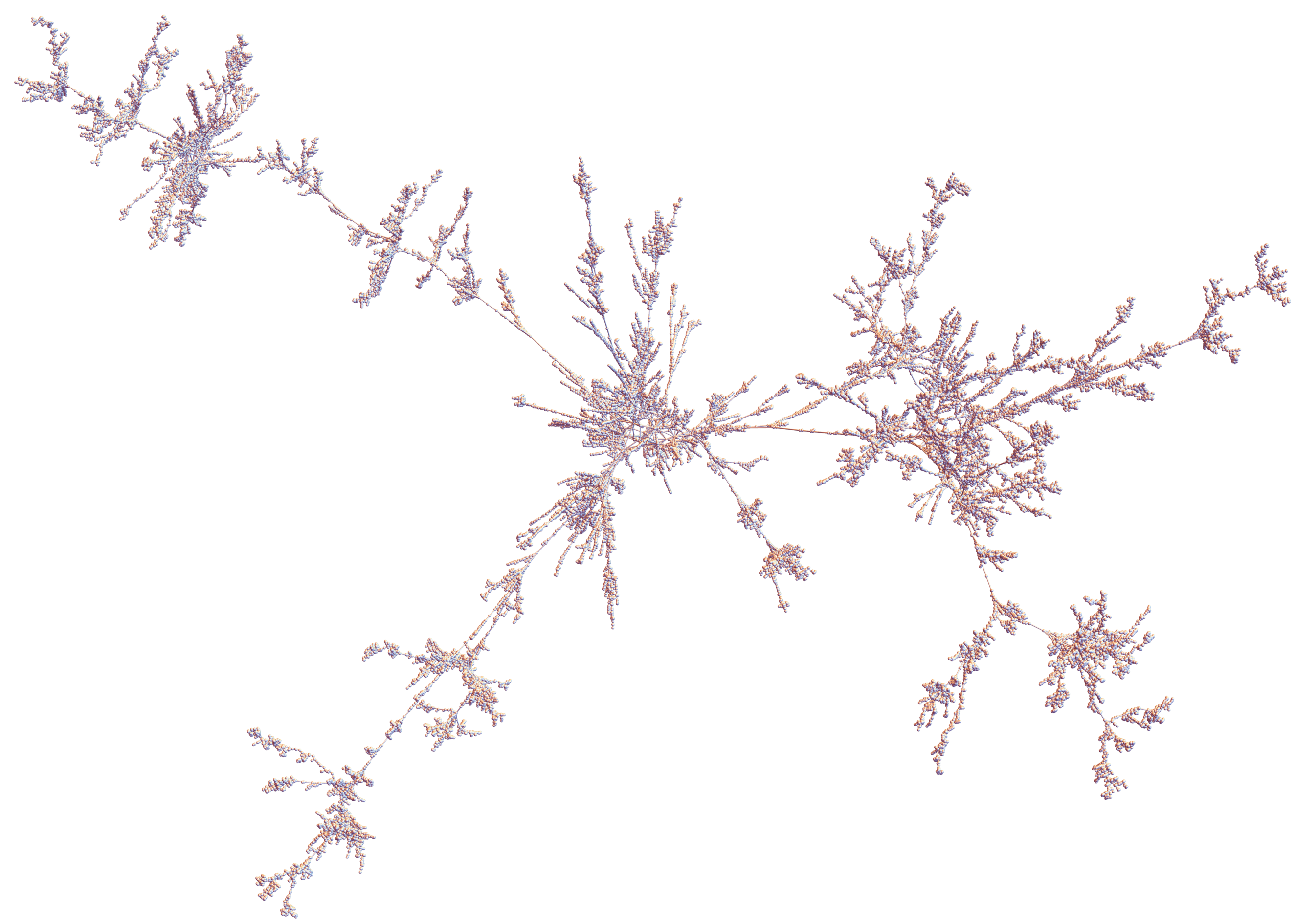}
\end{center}
\caption{Map drawn according to the critical model $\mathbb{P}_{n,9/5}$ of size around $80\ 000$ (see larger version in \cref{fig:simul-crit-9/5}).}
\label{fig:simul-crit-9/5-small}
\end{figure}

\begin{figure}
    \centering
    \begin{minipage}{6.5cm}
        \centering
        \includegraphics[height=5cm]{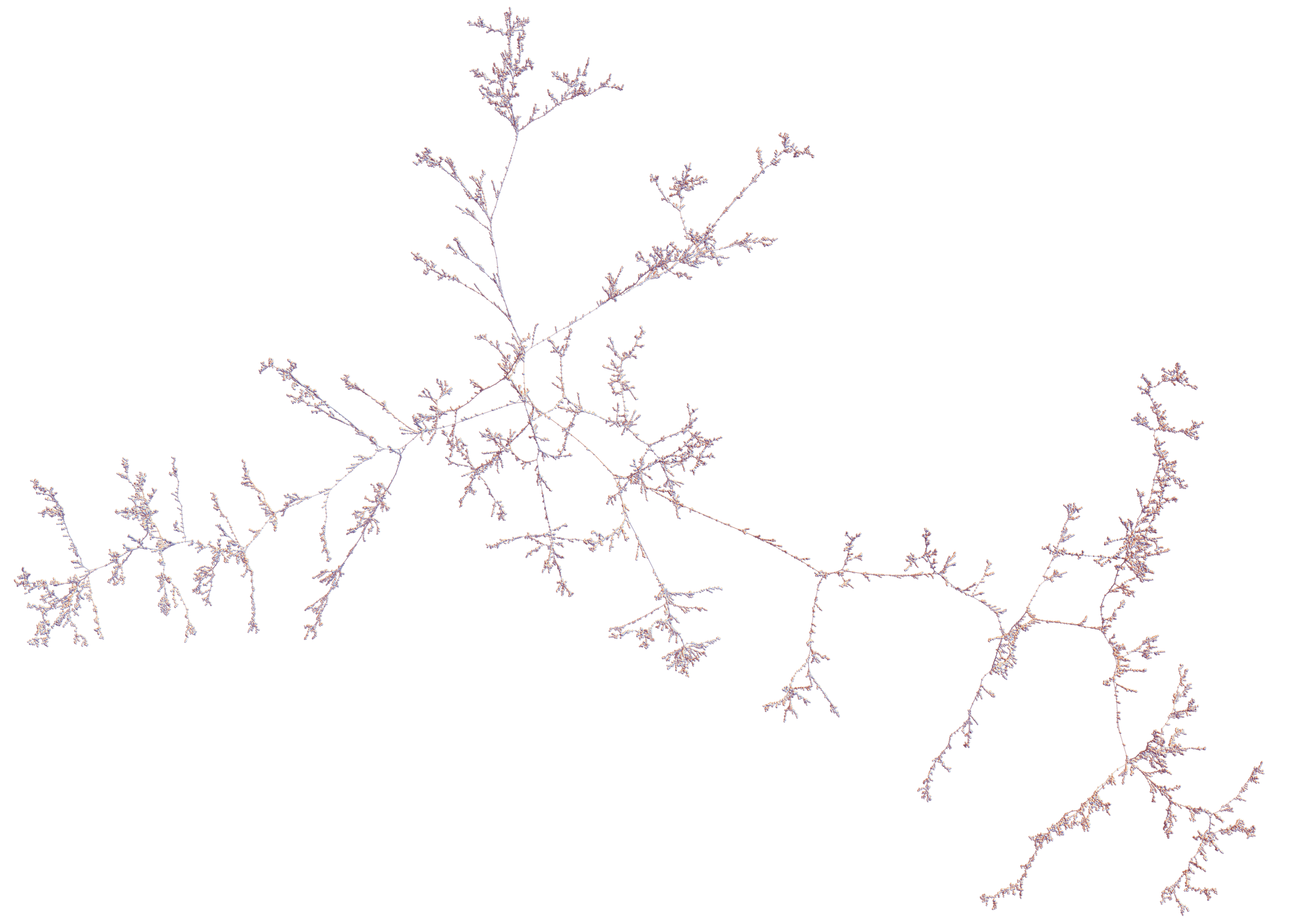}
        \caption{Map drawn according to the supercritical model $\mathbb{P}_{n,5/2}$ of size around $75\ 000$ (see larger version in \cref{fig:simul-super-5/2}).}
        \label{fig:simul-super-5/2-small}
    \end{minipage}\hfill
        \begin{minipage}{6.5cm}
        \centering
        \includegraphics[height=5cm]{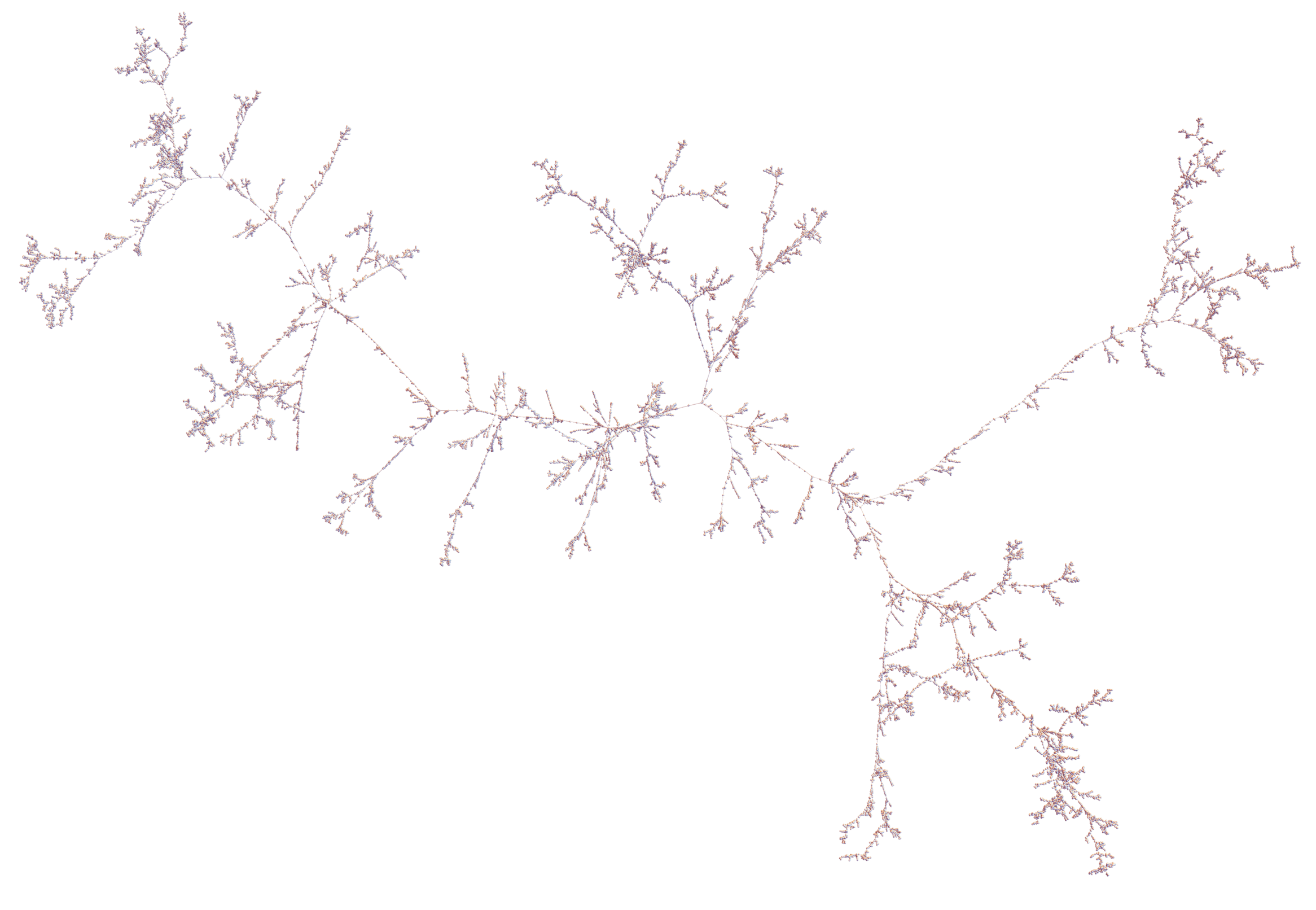}
        \caption{Map drawn according to the supercritical model $\mathbb{P}_{n,5}$ of size around $50\ 000$ (see larger version in \cref{fig:simul-super-5}).}
        \label{fig:simul-super-5-small}
    \end{minipage}
\end{figure}

%%%%%%%%%%%%%%%%%%%%%%%%%%%% Results %%%%%%%%%%%%%%%%%%%%%%%%%%%%%%%%

\paragraph{Results.} Our results are summarized in \cref{results}. In \cref{section:blocks}, we show that, with high probability, when $u<9/5$, there is condensation with one block of size $\Theta(n)$ and all others of size $O(n^{2/3})$, see \cref{souscritique}; when $u>9/5$, the largest block has size $\Theta(\log(n))$, see \cref{grandbloc}; and when $u=9/5$ the largest block is of size $\Theta(n^{2/3})$, see \cref{size-bloc-critical}.

In \cref{section:scaling-limits}, we give a unified proof of the convergence towards $\mathcal{T}^{(2)}$, after renormalising distances by $n^{1/2}$, in the supercritical case $u>9/5$; and towards $\mathcal{T}^{(3/2)}$, after renormalising distances by $n^{2/3}$, in the critical case $u=u_C$ (\cref{scaling-limit-super-critical}). For $u>9/5$, we retrieve a previous result by Stufler for more general weighted models \cite{Stufler-survey-2020}. All these results hold for both maps and their 2-connected cores, and quadrangulations and their simple cores. Finally, when $u<9/5$, we show in \cref{scaling-limit-subcritical} that quadrangulations converge towards the Brownian sphere when renormalising distances by $n^{1/4}$. In the case of quadrangulations, these results are consistent with existing literature for the case $u=1$ \cite{Miermont2013,LeGall2013,brownianmap}, as well as when $u\to 0$ \cite{add-alb}. We rely crucially on the convergence of uniform \textit{simple} quadrangulations with the same normalisation, which is proven in \cite{add-alb}, and recalled in \cref{scaling-lim-quad} below. A similar convergence result for uniform 2-connected maps would be needed in order to prove a version of \cref{scaling-limit-subcritical} for maps, see the discussion after the statement of \cref{scaling-lim-quad}. Such a convergence is expected to hold and hinted at for instance by Lehéricy's results \cite{Lehericy2022}, which show that graph distances on a uniform map of size $n$ and on its quadrangulation via Tutte's bijection behave similarly when $n\to \infty$.

 \Cref{sec:tree-decomp,tree-critical} introduce tools to prove these theorems. We show that maps and quadrangulations can be decomposed into blocks with an underlying tree structure. We show that the law of such trees can be described by a Galton-Watson model (as in several papers cited above). From there, we exhibit in \cref{sec:phase-diagram} a phase diagram going from a condensation phenomenon ($u<9/5$) to a critical ``generic'' regime ($u>9/5$) going through a ``non-generic'' critical point ($u=9/5$).

\begin{table}
\begin{center}
\begin{tabular}{ l | c c c }
   & Largest block & Scaling & Scaling limit \\
   \hline
    $u<9/5$ & $\Theta(n)$ & $n^{1/4}$ & Brownian sphere\tablefootnote{
    	We only prove convergence to the Brownian sphere in the case of quadrangulations and their simple blocks, see the discussion after \cref{add-alb}.} $\mathcal{M}_{e}$\\
%    	When considering maps and their 2-connected blocks, this result is conditional on the convergence of 2-connected maps towards the Brownian sphere.
    $u=9/5$ & $\Theta(n^{2/3})$ & $n^{1/3}$ & Stable tree $\mathcal{T}^{(3/2)}$ \\
    $u>9/5$ & $\Theta(\log(n))$ & $n^{1/2}$ & Brownian tree $\mathcal{T}^{(2)}$ \\
 \end{tabular}
\end{center}
\caption{Behaviour of the model when $u$ varies.}
\label{results}
\end{table}

\section*{Acknowledgments}
The authors would like to thank Marie Albenque, Éric Fusy and Grégory Miermont for their supervision throughout this work, and for all the invaluable comments and discussions. They also wish to thank the anonymous EJP referees for their careful reading and helpful suggestions.

% !TEX root = FleuratSalvy.tex

\section{Tree decomposition of maps}
\label{sec:tree-decomp}

\subsection{Maps and their enumeration}
A \emph{planar map} $\m$ is the proper embedding into the two-dimensional sphere of a connected planar finite multigraph, considered up to homeomorphisms. Let $V(\m)$ be the set of its vertices, $E(\m)$ the set of its edges and $F(\m)$ the set of its faces. The size of a planar map $\m$~---~denoted by $|\m|$~---~is defined as its number of edges.

A \emph{half-edge} $e$ is an oriented edge from $u$ to $v$ (with possibly $v = u$) and is represented as half of an edge starting from $u$. Its \emph{starting vertex} $u$ is denoted by $e^-$ and its \emph{end vertex} $v$ is denoted $e^+$. Let $\vect{E}(\m)$ be the set of half-edges of $\m$.

A \emph{corner} is the angular sector between two consecutive edges in counterclockwise order around a vertex. Each half-edge is canonically associated to the corner that follows it in counterclockwise order around its starting vertex. The \emph{degree} of a face is the number of corners incident to it.

All the maps considered in this paper are \emph{rooted}, meaning that one of their half-edges (or one of their corners) is distinguished. The set of rooted planar maps~---~simply called maps in the following~---~is denoted by~$\mathcal{M}$. For $n$ in $\N$, let $m_n$ be the number of maps of size $n$ and $M(z) = \sum_{n\in\N}m_n z^n$ be the associated generating series. By convention, we set $m_0 = 1$ which corresponds to the \emph{vertex map}: the map reduced to a single vertex. Similarly, define the \emph{edge map} as the map reduced to a single edge between two vertices.

Rooting simplifies the study by avoiding symmetry problems, however we expect our results remain true in the non-rooted setting due to the general results of \cite{RichmondWormald1995}. The enumerative study of rooted planar maps was initiated by Tutte in the 60s. In particular, he obtained the following result:

\begin{proposition}[\cite{tutte_1963}] The number $m_{n}$ of maps of size $n$ is equal to
\begin{equation}
\label{enum-cartes}m_n = \frac{2 (2n)! 3^n}{(n+2)!n!} \sim \frac{2}{\sqrt{\pi}} 12^n n^{-5/2},\quad n\to\infty.
\end{equation}
This implies in particular that $\rho_M=1/12$ and $M(\rho_M)<\infty$, where $\rho_M$ denotes the radius of convergence of $M(z)$.
\end{proposition}

\subsection{2-connected maps and block decomposition}
\label{decomp-blocks}
\begin{defin}
A map $\m \in \mathcal{M}$ is said to be \emph{separable} if it is possible to partition its edge-set $E(\m)$ into two non-empty sets $E$ and $E'$ such that there is exactly one vertex (called \emph{cut vertex}) incident to both a member of $E$ and a member of $E'$. The map $\m$ is said to be \emph{2-connected} otherwise, see \cref{separable}.
\end{defin}

Note that, by definition, the vertex map is $2$-connected. For $n\in\N$, we write $\mathcal{B}_n$ for the set of 2-connected maps of size $n$, and $b_{n} = |\mathcal{B}_{n}|$. From \cref{2_connected}, we see that $b_{0} = 1$, $b_1 = 2$ and $b_2=1$. Contrary to Tutte \cite{tutte_1963}, we choose $m_0 = b_0 = 1$ (and not $m_0 = b_0= 0$) and express the results accordingly. Notice in particular that the only 2-connected map with a loop is the map reduced to a loop-edge.

\begin{figure}
    \centering
    \begin{minipage}{6.5cm}
        \centering
        \includegraphics[height=4cm]{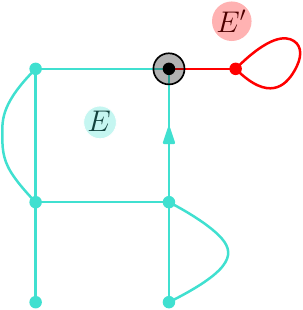} % first figure itself
        \caption{Example of a separable map. The circled black dot is a cut vertex.}
        \label{separable}
    \end{minipage}\hfill
    \begin{minipage}{6.5cm}
        \centering
        \includegraphics[height=4cm]{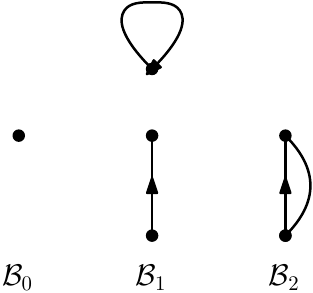} % second figure itself
        \caption{The classes $\mathcal{B}_0$, $\mathcal{B}_1$ and $\mathcal{B}_2$ of 2-connected maps with respectively $0$, $1$ and $2$ edges.}
        \label{2_connected}
    \end{minipage}
\end{figure}

\begin{defin}
\label{def-block}
A \emph{block} of a planar map $\m$ is a maximal 2-connected submap \emph{of positive size}. The number of blocks of $\m$ is denoted by $b(\m)$.
\end{defin}

In the 60's, Tutte introduced the so-called ``block decomposition of maps'' \cite{tutte_1963}, which roughly speaking corresponds to cutting the map at all cut-vertices, and is illustrated on \cref{blocs} (this is known for graphs as well and called \emph{block-cut tree}, see \emph{e.g.} \cite{harary1969}).

\begin{figure}
    \centering
    \begin{minipage}{6.5cm}
        \centering
        \includegraphics[height=4cm]{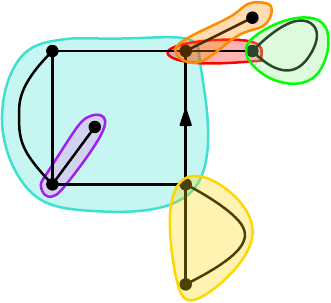}
        \caption{Decomposition of a map into blocks.}
        \label{blocs}
    \end{minipage}\hfill
    \begin{minipage}{6.5cm}
        \centering
        \includegraphics[height=4cm]{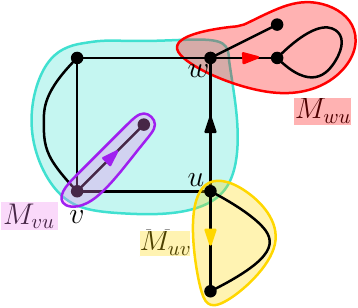} % first figure itself
        \caption{Pendant submap: the block to which the half-edges $uv$, $vu$ and $wu$ belong is in blue.}
        \label{pendant}
    \end{minipage}
\end{figure}

We describe here this decomposition drawing inspiration from Addario-Berry's presentation \cite[\S2]{2Louigi}. Let $\m$ be a map and let $\b$ be the block containing its root. For each half-edge $e$ of $\b$, we define the \emph{pendant submap} $\m_{e}$ of $e$ as the maximal submap of $\m$ disjoint from $\b$ except at $e^-$ and located to the left of $e$ (it is possibly reduced to the vertex map). If $\m_e$ has at least one edge, we root it at the half-edge of $\m$ following $e$ in counterclockwise order around $e^-$ (see \cref{pendant}).

From $\b$ and the $2|E(\b)|$ pendant submaps $\{\m_{e}, e \in \overrightarrow{E}(\b)\}$, it is possible to reconstruct the map $\m$: for each $\m_{e}$ rooted at the half-edge $\rho$, insert $\m_{e}$ in the corner associated to $e$ in such a way that $\rho$ is the first edge after $e$ in counterclockwise order and merge $\rho^-$ and $e^-$. Thus, a map can be encoded as a block where each edge is decorated by two maps. This decomposition induces an identity of generating series, thanks to the symbolic method \cite[Ch1]{flajolet_sedgewick_2009}. Letting $B(y) =\sum_{n\in\N}b_n y^n$, Tutte's block decomposition translates into the following equality of generating series:
\begin{equation}
\label{eqn-cartes}
M(z) = B(zM(z)^2).
\end{equation}
Thanks to \cref{eqn-cartes} and an explicit expansion for $M(z)$ obtained in \cite{tutte_1962}, Tutte obtained the following enumerative results for 2-connected maps.

\begin{proposition}[\cite{tutte_1963}]\label{proposition-B}
The number $b_n$ of 2-connected maps of size $n$ is
\begin{equation}
\label{asymptB}
b_0 = 1,\qquad\text{and for }n\geq1,\ b_n = \frac{2(3n-3)!}{n!(2n-1)!} \sim \sqrt{\frac{3}{\pi}} \frac{2}{27} \(\frac{27}{4}\)^n n^{-5/2},\quad n\to\infty.
\end{equation}

Moreover, writing $\rho_B$ for the radius of convergence of the series $B$, we have
\begin{equation}
\label{tutte-valeurs}
\rho_B = \frac{4}{27},\qquad B(\rho_B) = 4/3\qquad\text{and}\qquad\rho_B\cdot B'(\rho_B) = \sum_{n\in\N}n b_n \rho_B^{\,n} = 4/9.
\end{equation}
\end{proposition}

In the following, we consider maps enumerated by both their number of edges and their numbers of blocks. Namely, we consider the following bivariate series: $M(z,u) = \sum_{\m\in \mathcal{M}}z^{|\m|}u^{b(\m)}$ (recall that $b(\m)$ is the number of blocks of $\m$ and $|\m|$ is its number of edges). Tutte's decomposition of a map into blocks translates in the following refined version of \cref{eqn-cartes}:
\begin{equation}
\label{bivar}
M(z,u) - 1= u\(B(zM(z,u)^2) -1\)\qquad\text{\emph{i.e.}}\qquad M(z,u) = uB(zM(z,u)^2) + 1-u,
\end{equation}
where the term $1-u$ accounts for the fact that the vertex map has no block by \cref{def-block} (even if it is $2$-connected). For $u>0$, denote by $\rho(u)$ the radius of convergence of $z\mapsto M(z,u)$.
Since for $z\geq0$ and $u\geq 1$
\[M(z,u) \leq \sum_{\m\in\mathcal{M}} z^{|\m|}u^{|\m|} = M(uz),\]
if $|uz| < \rho_M = 1/12$, then $M(z,u)$ is a converging sum. Hence, for $u\geq 1$, $\rho(u) \geq \frac{1}{12u} > 0$. On the other hand, since $\rho(u)$ is decreasing, for $u\leq 1$ we have $\rho(u) \geq \rho(1)=\rho_M =1/12$ (and $\rho(u) \leq \rho(0)=\rho_B = 4/27$).

In view of the form of the equation \cref{bivar} and in particular that it is non-linear, it holds that $M(\rho(u), u)<\infty$. Indeed, since $B(y)\geq 1 + 2y$ for all $y\geq 0$, we get $M(z,u) \geq 1 + 2uzM(z,u)^2$. This shows that it is impossible that $M(z,u) \xrightarrow[z\to\rho(u)^-]{} +\infty$.

\subsection{Block tree of a map and its applications}
\label{louigi}
Tutte's block decomposition can also be applied recursively, \emph{i.e.} we consider first the root block and then apply the block decomposition to each of the pendant submaps. By doing so, for any map $\m$ we can obtain a decomposition tree $T_{\m}$, which was first explicity described by Addario-Berry in \cite[\S2]{2Louigi}. More precisely:
\begin{enumerate}
\item Let $\b=(\b,\rho)$ be the maximal 2-connected submap containing the root $\rho$. The root $v_\rho$ of $T_\m$ represents $\b$, and has $2|E(\b)|$ children (in particular, if $\b$ is of size $0$, $v_\rho$ is a leaf);
\item List the half-edges of $\b$ as $a_1, \dots, a_{2|E(\b)|}$ according to an arbitrarily fixed deterministic order on half-edges (\emph{e.g.} the order in a left-to-right depth first search). Let $\m_i$ be the pendant submap in the corner corresponding to the half-edge $a_i$ in $\b$. The $i$-th pendant subtree of $T_\m$ is the subtree encoding $T_{\m_i}$.
\end{enumerate}
An example of such a correspondence is described in \cref{ex_TM}. This decomposition has three essential properties, that follow directly from its definition, and that we summarize in the following proposition.

\begin{figure}
\begin{center}
\includegraphics[width=8.5cm, center]{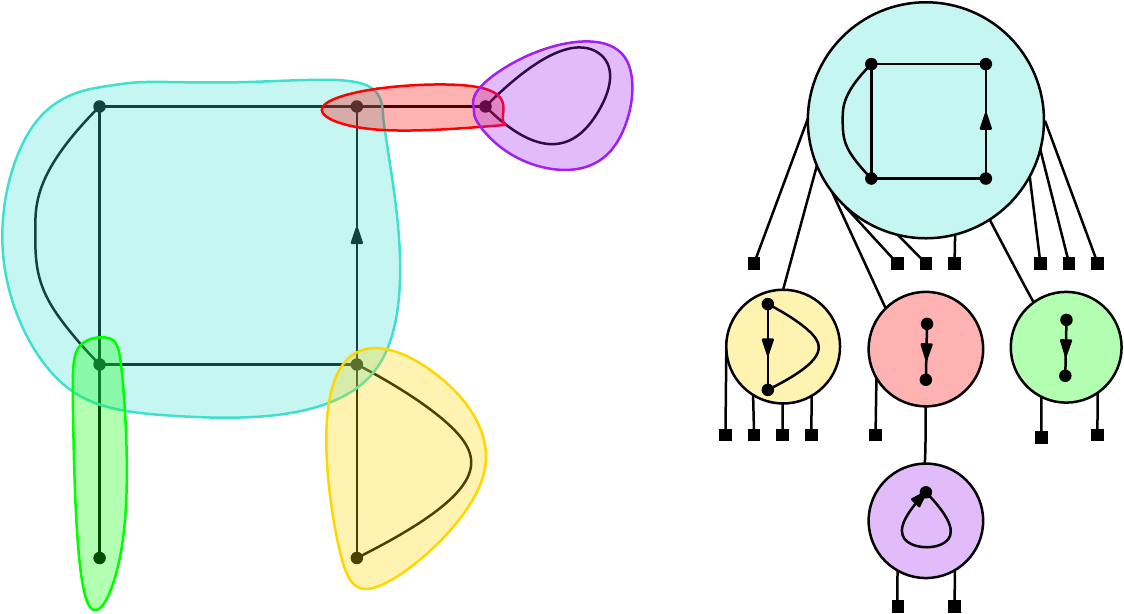}
\end{center}
\caption{Block tree corresponding to a planar map.}
\label{ex_TM}
\end{figure}

\begin{proposition}[\cite{tutte_1963,2Louigi}]
\label{prop-t-m-n}
The block tree $T_\m$ of a map $\m$ satisfies the following properties:
\begin{itemize}
\item The edges of $T_\m$ correspond to the half-edges of $\m$;
\item The internal nodes of $T_\m$ correspond to the blocks of $\m$: if an internal node $v$ of $T_\m$ has $r$ children, then the corresponding block $\b_v$ of $\m$ has size $r/2$;
\item The map $\m$ is entirely determined by $\(T_\m, (\b_v, v\in T_\m)\)$ where $\b_v$ is the block of $\m$ represented by $v$ in $T_\m$ if $v$ is an internal node; else, by convention, $\b_v$ is the vertex map.
\end{itemize}
\end{proposition}
By abuse of language, we might refer to $(\b_v, v\in T_\m)$ as the family of blocks (even if blocks necessarily have positive size). A direct consequence of this proposition is that to study the block sizes of a map $\m$, it is sufficient to study the degree distribution of $T_{\m}$. This is precisely the strategy developed by Addario-Berry in \cite{2Louigi}. This allows him to study the block sizes of a uniform random map $\M_n$ of size $n$, by describing $T_{\M_n}$ as a Galton-Watson tree with an explicit degree distribution conditioned to have $2n$ edges, and one of our contributions is to extend his result to our model.

\subsection{Block tree of a quadrangulation}
\label{subseq:quad}
We describe in this section how a quadrangulation can be decomposed into maximum simple quadrangular components, in the same way that a map can be decomposed into maximum 2-connected components.

\begin{defin}
A \emph{quadrangulation} is a map with all faces of degree $4$.
\end{defin}

Planar quadrangulations are \emph{bipartite}, \emph{i.e.} their vertices can be properly bicolored in black and white. In the following, we always assume that they are endowed with the unique such coloring having a black root vertex. Although quadrangulations are maps, when an object is explicitly defined as a quadrangulation, its size will be its number of faces. Thus, a quadrangulation of size $n$ has $2n$ edges.

\begin{defin}
A \emph{quadrangulation of the $2$-gon} is a map where the \emph{root face}~---~the face containing the corner associated to the root~---~has degree $2$ and all other faces have degree $4$.
\end{defin}
A quadrangulation of the $2$-gon with at least two faces can be identified with a quadrangulation of the sphere by simply gluing together both edges of the root face.

%%%%%%%%%%%%% Décomposition des quadrangulations
\begin{defin}
A quadrangulation is called \emph{simple} if it has neither loops nor multiple edges.
\end{defin}

\begin{defin}
Let $e_1 e_2$ be a $2$-cycle of a quadrangulation $\q$, its \emph{interior} is the submap of $\q$ between $e_1$ and $e_2$ (both included) which does not contain the root corner of $\q$. A $2$-cycle is \emph{maximal} when it does not belong to the interior of another $2$-cycle.
\end{defin}

\begin{defin}
Let $e_1 e_2$ be a maximal $2$-cycle of a quadrangulation $\q$, its \emph{pendant subquadrangulation} is defined as its interior, which is turned into a quadrangulation of the $2$-gon by rooting it at the corner incident to the unique black vertex of $e_1e_2$.

Let $e$ be a half-edge of a quadrangulation $\q$. If $e$ is oriented from black to white and there exists a half-edge $f$ such that $ef$ is a maximal $2$-cycle of $\q$, then the \emph{pendant subquadrangulation} of $e$ is the pendant subquadrangulation of $ef$. Else, it is the edge map (which is also a quadrangulation of the $2$-gon).
\end{defin}

\begin{figure}
\begin{center}
\includegraphics[width=12cm, center]{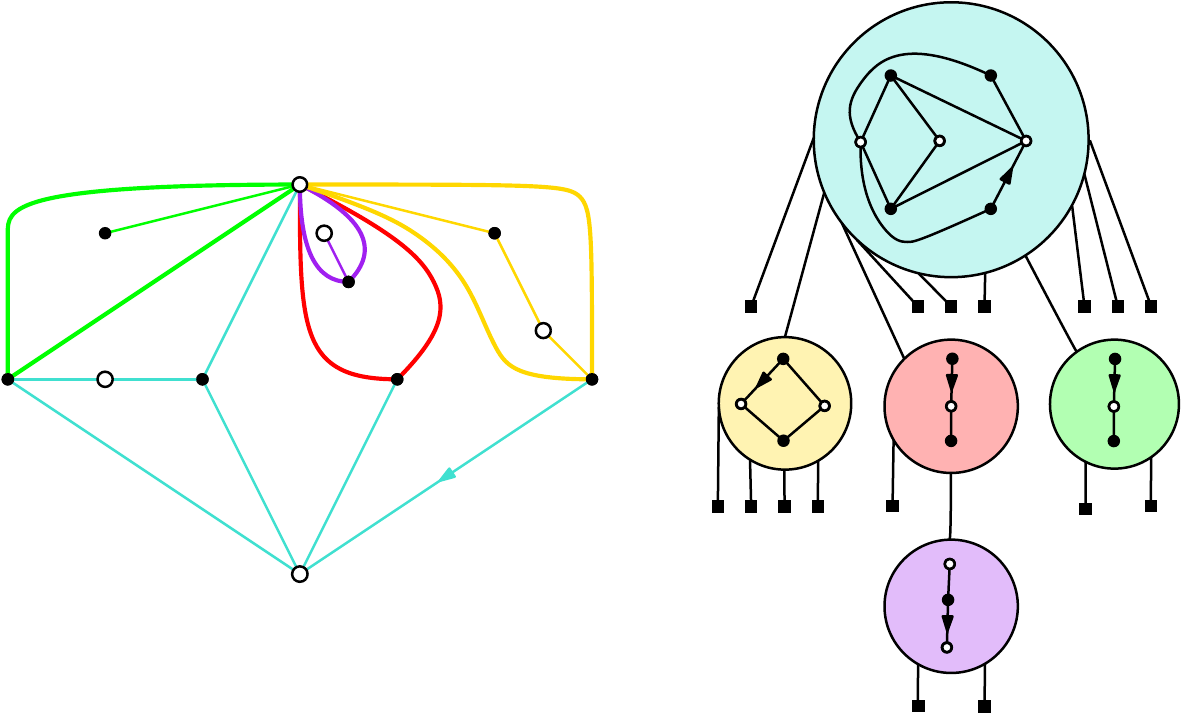}
\end{center}
\caption{The image of the map of \cref{ex_TM} via Tutte's bijection, and its block tree.}
\label{ex_TMq}
\end{figure}

For $\q$ a quadrangulation, its \emph{simple core} $\q_{s}$~---~the simple block containing the root~---~is obtained by collapsing the interior of every maximal $2$-cycle of $\q$. Similarly as for maps, a decomposition tree $T^{(q)}_{\q}$ can be associated to a quadrangulation $\q$, by recursively decomposing the pendant subquadrangulations at the simple core, see \cref{ex_TMq}. \emph{Simple blocks} are recursively defined as the simple cores appearing in the underlying arborescent decomposition. We then have an exact parallel with the situation of maps and their 2-connected components.

\begin{figure}
\begin{center}
\includegraphics[width=8.5cm, center]{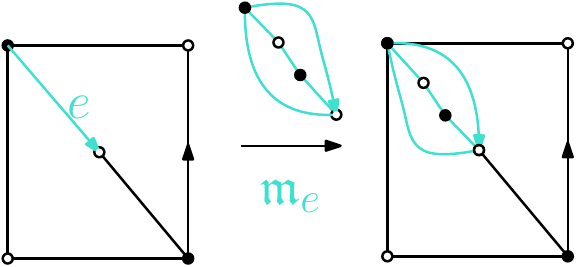}
\end{center}
\caption{Reconstructing a quadrangulation from its simple core and the pendant subquadrangulations}
\label{pendant_subquad}
\end{figure}

Given a simple quadrangulation $\q_{s}$ and a collection of $|E(\q_s)| = 2|q_{s}|$ quadrangulations of the $2$-gon $\{\m_{e}, e \in E(\q_s)\}$, it is possible to construct a quadrangulation: for each $\m_{e}$ of root $\rho_e$ replace $e$ by $\m_e$ such that $\rho_e$ has the orientation $e$. See \cref{pendant_subquad} for an illustration. This transformation is invertible. Thus, a quadrangulation can be encoded as a simple quadrangulation where each edge is decorated by one quadrangulation of the $2$-gon, \emph{i.e.} each face is decorated by two quadrangulations of the $2$-gon:
\begin{equation}
\label{bivar-quad}
Q(z,u) + 1 = uS(z(Q(z,u)+1)^2)+1-u,
\end{equation}
where $Q$ is the generating series for quadrangulations (with a weight $z$ for faces, and $u$ for simple blocks) and $S$ is the generating series for simple quadrangulations (with a weight $z$ for faces). Note that this equation is isomorphic to \cref{bivar}.

%%%%%%%%%%%%%% Cette décomposition est équivalente à la décomposition présentée plus haut
This decomposition and the former one presented for general maps are in fact two sides of the same coin. Indeed, they can be related via Tutte's bijection as we now present: there exists an explicit bijective construction between quadrangulations of size $n$ and (general) maps of size $n$. More precisely, for a map $\m$ (rooted in $\rho$), its image by $\varphi$, called its \emph{angular map}, can be constructed as follows, see \cref{construction-incidence-map}.
\begin{enumerate}
    \item Add a (white) vertex inside each face of $\m$ and draw an edge from this new (white) vertex to each corner around the face (respecting the order of the corners);
    \item The half-edge $e$ created in the corner of $\rho$ is now the root, oriented from black to white;
    \item Remove the original edges.
\end{enumerate}

\begin{figure}
\begin{center}
\includegraphics[width=12cm, center]{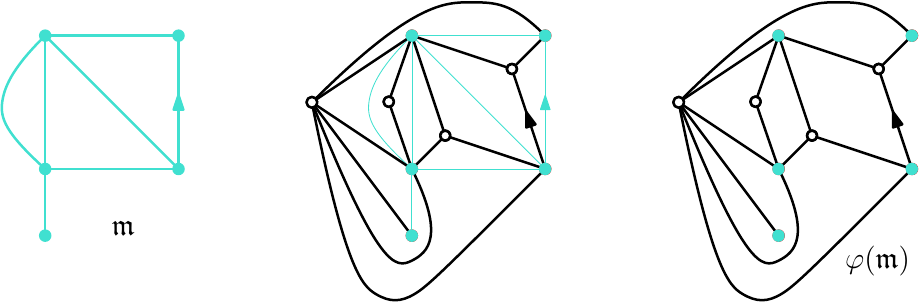}
\end{center}
\caption{The quadrangulation corresponding to a map via Tutte's bijection.}
\label{construction-incidence-map}
\end{figure}

\begin{proposition}
\label{bij-quad-maps}
For $n\in\Z_{>0}$, the function $\varphi$ is a bijection between maps of size $n$ and quadrangulations of size $n$. Moreover, it maps bijectively 2-connected maps of size $n \geq 1$ to simple quadrangulations of size $n$.
\end{proposition}

The construction $\varphi$ is due to Tutte \cite[\S5]{tutte_1963} (he defines the notion of derived map, from which the angular map is extracted by deleting one of the 3 classes of vertices, as explained in \cite[\S7]{Brown1965}). The specialization to 2-connected maps is explained \emph{e.g.} in \cite{Brown1965}. In particular, it implies that $S(y) = B(y)$. Moreover, given \cref{bivar,bivar-quad}, this gives $M(z,u) = Q(z,u) + 1$.

%%%%% Conclusion
Finally, when constructing the decomposition tree $T^{(q)}_{\q}$, if the deterministic orders used for the half-edges of 2-connected maps and for the edges of simple quadrangulations are consistent via Tutte's bijection, then the decomposition trees of $\m$ and of $\varphi(\m)$ are the same, and for each node $v$ of the tree, the 2-connected map (resp.~simple quadrangulation) at $v$ are in correspondence by Tutte's bijection, \emph{e.g.} the example of \cref{ex_TM} is consistent with the example of \cref{ex_TMq} via Tutte's bijection. This can be rephrased as the following result.

\begin{proposition}
\label{arbre-bloc-quad-carte}
For all $\m \in \mathcal{M}$,
\[T^{(q)}_{\varphi(\m)} = T_{\m}\]
and, for all $v \in T^{(q)}_{\varphi(\m)}$,
\[\b_v^{(q)} = \varphi(\b_v).\]
\end{proposition}

\subsection{Probabilistic consequences}\label{subsec:conseq-proba}
Recall the model defined in \cref{eq:defPrnu,eq:defPru} for general maps. As promised, we now define its analogue on quadrangulations, and show their equivalence. To that end, we set for all $\m \in \mathcal{M}_n$, and for all $\q \in \mathcal{Q}_n$,
\[\pr[_{n,u}^{\mathrm{map}}]{\m} = \frac{u^{b(\m)}}{[z^n] M(z,u)} \propto u^{b(\m)}\qquad\text{and}\qquad
\pr[_{n,u}^{\mathrm{quad}}]{\q} = \frac{u^{b(\q)}}{[z^n] Q(z,u)} \propto u^{b(\q)},
\]
and consider for all $\m \in \mathcal{M}$ and $\q \in \mathcal{Q}$ the singular Boltzmann laws (remember that, as explained in \cref{decomp-blocks}, $M(\rho(u),u) = Q(\rho(u),u) <\infty$)
\[\pr[_{u}^{\mathrm{map}}]{\m} = \frac{u^{b(\m)} \rho(u)^{|\m|}}{M(\rho(u),u)}\qquad\text{and}\qquad\pr[_{u}^{\mathrm{quad}}]{\q} = \frac{u^{b(\q)} \rho(u)^{|\q|}}{Q(\rho(u),u)},
\]
then
\[\mathbb{P}_{n,u}^{\mathrm{map}} = \pr[_u^{\mathrm{map}}]{\cdot\mid\mathcal{M}_n}\qquad\text{and}\qquad\mathbb{P}_{n,u}^{\mathrm{quad}} = \pr[_u^{\mathrm{quad}}]{\cdot\mid\mathcal{Q}_n}.\]

By \cref{arbre-bloc-quad-carte}, one has:
\begin{proposition}\label{prop-Tutte}
For all $\q \in \mathcal{Q}$ and $n\in\N$,
\[\pr[_{n,u}^{\mathrm{quad}}]{\q} = \pr[_{n,u}^{\mathrm{map}}]{\varphi^{-1}(\q)}\qquad\text{and}\qquad\pr[_{u}^{\mathrm{quad}}]{\q} = \pr[_{u}^{\mathrm{map}}]{\varphi^{-1}(\q)},\]
so, denoting by $*$ the pushforward, for all $n\in\N$,
\[\mathbb{P}_{n,u}^{\mathrm{quad}} = \varphi _* \mathbb{P}_{n,u}^{\mathrm{map}}\qquad\text{and}\qquad\mathbb{P}_{u}^{\mathrm{quad}} = \varphi _* \mathbb{P}_{u}^{\mathrm{map}}.\]
\end{proposition}

\subsection{A word on the probabilistic setting}\label{subsec:proba-setting}

We denote by $\M : \mathcal{M} \to \mathcal{M}$ the canonical random variable on the space of maps, and let $\Q = \varphi(\M)$. We denote by $\T$ the block tree associated to $\M$ (and also to $\Q$ by \cref{arbre-bloc-quad-carte}). In this way, under $\mathbb P_u$ (resp.~$\mathbb{P}_{n,u}$), $\M$ has law $\mathbb{P}_u^{\mathrm{map}}$ (resp.~$\mathbb{P}_{n,u}^{\mathrm{map}}$), and, by \cref{prop-Tutte}, $\Q$ has law $\mathbb{P}_u^{\mathrm{quad}}$, (resp.~$\mathbb{P}_{n,u}^{\mathrm{quad}}$). Therefore, we will simply use $\mathbb{P}_{n,u}$ and $\mathbb{P}_{u}$ as a shorthand notation for $\mathbb{P}_{n,u}^{\mathrm{map}}$ and $\mathbb{P}_{u}^{\mathrm{map}}$.

Maximal simple components of quadrangulations will also be called ``blocks'' because everything that has been said about blocks (in the sense of maximum 2-connected components) can also be said about the maximum simple quadrangular components of quadrangulations; and likewise in everything that follows. As a consequence, every result about the size of the blocks of a map of size $n$ is valid for blocks of quadrangulations of size $n$ as well.

For $v$ a vertex of $\T$, we denote by $\b_v^\M$ (resp.~$\b_v^\Q$) the 2-connected block of $\M$ (resp.~simple block of $\Q$) represented by $v$ in $\T$.
By \cref{arbre-bloc-quad-carte}, it holds that $\b_v^\Q=\varphi(\b_v^\M)$ for all $v\in \T$, where $\varphi$ is Tutte's bijection.

These random variables will be studied under probability measures $\mathbb P_u$ and $(\mathbb P_{n,u})_{n\geq 1}$, which were introduced in \cref{subsec:conseq-proba}.
We write accordingly $\mathbb E_u[\dots]$ and $\mathbb E_{n,u}[\dots]$ the expectations with respect to these probability measures.
Unless mentioned otherwise or if it is clear from context, other random variables shall be viewed as defined on some probability space $(\Omega,P)$, and the according expectations will be written as $\E{\dots}$.
In particular we will use the following random variables defined on $(\Omega,P)$:
\begin{itemize}
	\item For each $u\geq 0$, the triplet $(\T_{n,u},\M_{n,u},\Q_{n,u})$ is $(\T,\M,\Q)$ under the law $\mathbb P_{n,u}$.
	\item For each $k\geq 1$, the pair $(B_k^{\mathrm{map}},B_k^{\mathrm{quad}})$ consists of a 2-connected map $B_k^{\mathrm{map}}$ with $k$ edges sampled uniformly, together with $B_k^{\mathrm{quad}}=\varphi(B_k^{\mathrm{map}})$ its image by Tutte's bijection. By \cref{bij-quad-maps}, the latter is a simple quadrangulation with $k$ faces sampled uniformly.
\end{itemize}

% !TEX root = FleuratSalvy.tex

\section{Phase diagram}
\label{sec:phase-diagram}

For $\mu$ a probability distribution on $\N$ and $n\in\N$, we denote by $GW(\mu, n)$ the law of a Galton-Watson tree with offspring distribution $\mu$ and conditioned to have $n$ edges. Following \cite{2Louigi}, for $u>0$ we aim at finding a measure $\mu^{u}$ such that $\T$ under $\mathbb{P}_{n,u}$ has law $GW(\mu^u,2n)$. To that end, for any $y\in[0, \rho_B]$ we introduce the following probability distribution
\begin{equation}
\label{mu}
\mu^{y,u}(2j) := \frac{b_j y^j u^{\mathbb{1}_{j\ne 0}}}{1+u(B(y)-1)}\quad\text{and}\quad\mu^{y,u}(2j+1):=0\qquad\text{for all }j\in \N
\end{equation}
where $b_j$ and $B$ are defined in \cref{proposition-B}. Moreover (see \cref{rem-choice-y-u} for a discussion), we set
\begin{equation}
\label{def-y-u}
y(u) := \rho(u) M^2(\rho(u),u)\quad\text{and}\quad\mu^{u} := \mu^{y(u),u}\qquad\text{for any }u>0,
\end{equation}
where we recall that $\rho(u)$ is the radius of convergence of $z\mapsto M(z,u)$. On \cref{z_function_u}, the value of $y(u)$ is represented, using an explicit expression (see \cref{rem-y-fonction-u}). Notice that in view of \cref{bivar}, $y(u) \leq \rho_{B}$ for all $u>0$ and
\begin{equation}
\label{Mrhouu}
1 + u(B(y(u)) -1) = M(\rho(u),u).
\end{equation}
Then, by \cref{asymptB}, for all $u>0$, we have:
\[\mu^{u}(\{2j\}) \sim \sqrt{\frac{3}{\pi}} \frac{2}{27} \frac{u}{M(\rho(u),u)} \(\frac{27}{4} y(u)\)^j j^{-5/2},\quad\text{as }j\to\infty,\]
so that by setting
\begin{equation}
\label{cu}
c(u) = \sqrt{\frac{3}{\pi}} \frac{2}{27} \frac{u}{M(\rho(u),u)},
\end{equation}
it holds that
\begin{equation}
\label{asympt-mu}
\mu^{u}(\{2j\}) \sim c(u) \(\frac{27}{4}y(u)\)^j j^{-5/2},\quad\text{as }j\to\infty.
\end{equation}
The following proposition extends \cite[Proposition 3.1]{2Louigi} to our setting.

\begin{figure}
\begin{center}
\includegraphics[width=8cm, center]{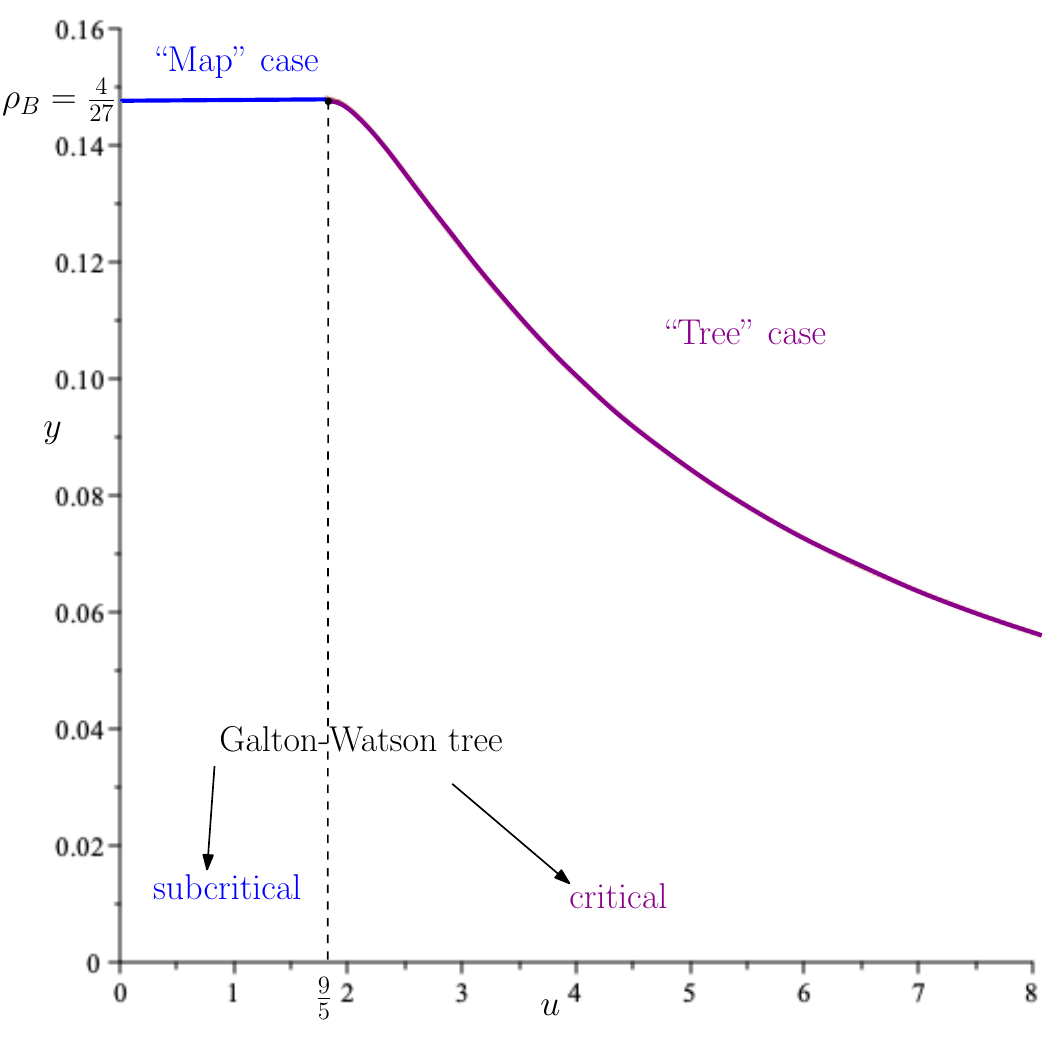}
\end{center}
\caption{Plot of $y$ as a function of $u$.}
\label{z_function_u}
\end{figure}

\begin{proposition}
\label{th-trees}
Let $(\mathbf{B}_v,v\in \mathbf{T})$ be either the family $(\b_v^\M)_{v\in\T}$ of blocks of $\M$, or $(\b_v^\Q)_{v\in\T}$ of blocks of $\Q$.
For every $u>0$, under $\mathbb{P}_u$, the law of tree of blocks
$(\mathbf{T},(\mathbf{B}_v,v\in \mathbf{T}))$ can be described as
follows.\begin{itemize}
\item $\mathbf{T}$ follows the law $GW(\mu^u)$;
\item Conditionally given $\mathbf{T}=\mathfrak{t}$,
the blocks $(\mathbf{B}_{v},v\in \mathfrak{t})$ are independent random variables, and, for $v\in\t$, $\B_v$ follows a uniform distribution on the set of blocks of size $k_v(\t)/2$, where $k_v(\mathfrak{t})$ is the number of children of $v$ in $\mathfrak{t}$.
\end{itemize}
For every $n\geq 1$, the same statements hold under $\mathbb{P}_{n,u}$, only replacing $GW(\mu^u)$ with $GW(\mu^u,2n)$.
\end{proposition}

\begin{proof}
It suffices to prove the statement for $\mathbb{P}_u$ as, by \cref{prop-t-m-n}, the block-tree of a map of size $n$ has size $2n$.
  
  Let $\mathfrak{t}$ be a tree where each vertex has
  an even number of children, and let $\(\mathfrak{b}_{v},v\in
  \mathfrak{t}\)$ be a family of ($2$-connected, or simple) blocks, with
  $2|\mathfrak{b}_{v}|=k_v(\mathfrak{t})$ for any $v\in\t$. Let $\mathfrak{m}$ be the map (or quadrangulation) with block decomposition given by
   $(\mathfrak{t},(\mathfrak{b}_v,v\in \mathfrak{t}))$.

   Then, we have
   \begin{align*}
&\mathbb{P}_u\(\mathbf{T}=\mathfrak{t},\mathbf{B}_v=\mathfrak{b}_v\ \forall v \in \t\)
=\mathbb{P}_u(\mathfrak{m})\\
&=\frac{\rho(u)^{|\mathfrak{m}|}u^{b(\mathfrak{m})}}{M(\rho(u),u)} = \frac{\rho(u)^{\sum_{v\in\t} k_v(\t)/2} u^{\sum_{v\in\t} \mathbb{1}_{k_v(\t) \ne 0}} }{M(\rho(u),u)} \prod_{v\in \t} \frac{b_{\frac{k_v(\t)}{2}}}{b_{\frac{k_v(\t)}{2}}}\\
&=\frac{1}{M(\rho(u),u)} \(\frac{y(u)}{M^2(\rho(u),u)}\)^{\sum_{v\in\t} k_v(\t)/2}\prod_{v\in
          \mathfrak{t}}{b_{\frac{k_v(\t)}{2}}u^{\mathbb{1}_{k_v(\mathfrak{t})\ne
              0}}}\times \prod_{v\in
          \mathfrak{t}}\frac{1}{b_{\frac{k_v(\t)}{2}}}
        \\
        &        =\frac{\prod_{v\in
          \mathfrak{t}}b_{\frac{k_v(\t)}{2}}y(u)^{k_v(\mathfrak{t})/2}u^{\mathbb{1}_{k_v(\mathfrak{t})\ne
              0}} }{M(\rho(u),u)^{1+\sum_{v\in\t} k_v(\t)}}\times \prod_{v\in
          \mathfrak{t}}\frac{1}{b_{\frac{k_v(\t)}{2}}}
        \\
        &=\prod_{v\in
          \mathfrak{t}}\frac{b_{\frac{k_v(\t)}{2}}y(u)^{k_v(\mathfrak{t})/2}u^{\mathbb{1}_{k_v(\mathfrak{t})\ne
              0}}}{M(\rho(u),u)}\times \prod_{v\in
          \mathfrak{t}}\frac{1}{b_{\frac{k_v(\t)}{2}}}\\
        &=GW(\mu^u)(\mathfrak{t}) \times \prod_{v\in
          \mathfrak{t}}\frac{1}{b_{\frac{k_v(\t)}{2}}}\, .
\end{align*}
This concludes the proof.
    \end{proof}

\begin{theorem}
\label{tree-critical} Recall the definition of $c(u)$ given in \cref{cu}. Then, depending on the value of $u$, the model $\mathbb{P}_u$ undergoes the following phase transition, driven by the properties of $\mu^u$:
\begin{description}
\item[Subcritical case.] For $u < u_C := 9/5$,
\begin{equation}\label{Eu}
E(u) := \E{\mu^{u}} = \frac{8u}{3(3+u)} <1\qquad\text{and}\qquad\mu^{u}(\{2j\}) \sim c(u) j^{-5/2}
\end{equation}
where $c(u)=\sqrt{\frac{3}{\pi}} \frac{2u}{9(3+u)}$;
\item[Critical case.] For $u = u_C:= 9/5$,
\[\E{\mu^{u_C}} =1\qquad\text{and}\qquad\mu^{u_C}(\{2j\}) \sim \frac{1}{4\sqrt{3\pi}}j^{-5/2};\]%,\]
%so $\mu^{u_C}$ is in the domain of attraction of a $Stable(3/2)$ law;
\item[Supercritical case.] For $u > 9/5$,
\[\E{\mu^u} =1\qquad\text{and}\qquad\mu^{u}(\{2j\}) \sim c(u) \(\frac{27}{4} y(u)\)^j j^{-5/2},\]
where $y(u)<4/27$ so that $\mu^u$ has exponential moments.
\end{description}
\end{theorem}
Notice that the case $u=1$, which corresponds to uniform planar maps, as studied by Addario-Berry \cite{2Louigi}, falls in the subcritical regime.

\begin{proof}[Proof of \cref{tree-critical}]
Let us first explain how the value $u_C:=9/5$ appears. Let $u > 0$ and $y \in \left (0,4/27\right]$. By \cref{mu},
\begin{equation}
\label{esperance}
\E{\mu^{y,u}} = \sum_{j\in \N}\frac{2jb_jy^ju^{\mathbb{1}_{j\ne0}}}{1+u(B(y)-1)} = \frac{2u yB'(y)}{1+u(B(y)-1)}.
\end{equation}

It follows that
\begin{equation}
\label{bij}
\E{\mu^{y,u}} = 1 \Leftrightarrow u = \frac{1}{2y{B}'(y)-B(y)+1}.
\end{equation}

The mapping $y \in \left(0,4/27\right] \mapsto d(y) := 2y{B}'(y)-B(y)+1$ is increasing. Indeed, for all $y \in \left(0,4/27\right]$,
\[d(y) = \sum_{n\geq 1} 2 n b_{n} y^{n} - \sum_{n\geq 0} b_{n} y^{n} + 1 = \sum_{n\geq 1} (2n-1) b_{n} y^n.\]

Moreover, if follows from \cref{tutte-valeurs} that $d(0) = 0$ and $d(4/27) = 5/9$. So $1/d(y)$ maps bijectively $\left (0,4/27\right]$ to $[9/5,+\infty)$. Therefore, there exists $y \in \left(0,4/27\right]$ such that the law $\mu^{y,u}$ is critical if and only if $u \in [9/5, +\infty)$, and this $y$ is unique.

We now conclude the proof of the theorem. For the sake of completeness, we recall an argument from \cite[\S8.2.2]{bonzomLagrange}. Recall \cref{bivar}:
\begin{equation*}
M(z,u) = uB(zM(z,u)^2) + 1-u.
\end{equation*}

For a fixed $u$, there are two possible sources of singularity:
\begin{enumerate}
  \item The pair $(z_0=\rho(u), m_0=M(\rho(u),u))$ satisfies $\frac{\partial H}{\partial m}(z_0,m_0) = 0$ for $H :(z,m) \mapsto m - uB(zm^2) - 1 +u$, thus being a singularity by the contraposition of the implicit function theorem. In this case,
\begin{equation*}
1 - 2 \rho(u) M(\rho(u), u) u B'(\rho(u) M^2(\rho(u),u)) = 0,\quad\text{so}\quad 2 \rho(u) M(\rho(u), u) u B'(y(u)) = 1.
\end{equation*}
Then, by \cref{Mrhouu},
\begin{equation}\label{eq-y-fonctions-implicites}
2 y(u) B'(y(u)) - B(y(u)) + 1 = \frac{2 \rho(u) M^2(\rho(u),u)}{2 u \rho(u) M(\rho(u),u)} - \frac{M(\rho(u),u)+u-1}{u} + 1 = \frac{1}{u},
\end{equation}
which is to say that $y(u) = \rho(u)M^2(\rho(u),u)$ satifies \cref{bij}. This is possible if and only if $u\geq 9/5$. Then, it follows that $\E{\mu^u} = 1$, and \cref{asympt-mu} gives the asymptotic behaviour of $\mu^u(2j)$.
\item A singularity of $B$ is reached so $\rho(u)M^2(\rho(u),u) = \rho_B = 4/27$ \emph{i.e.} $y(u) = 4/27$. Then, the value of $E(u)$ is obtained as an immediate consequence of \cref{esperance,tutte-valeurs}, and the asymptotic behaviour of $\mu^u(2j)$ comes from \cref{cu,asympt-mu}. This happens iff $u\leq 9/5$.
\end{enumerate}
Notice that at $u=u_C$, both types of singularity are reached.
\end{proof}

\begin{rem}
\label{rem-choice-y-u}
The proof of \cref{tree-critical} highlights the reasons behind our choice of $y(u)$ in \cref{def-y-u}. When $u\geq 9/5$, we choose $y(u)$ such that $E(u) = 1$. When $u<9/5$, this is not possible, and we choose the value of $y(u)$ maximising $E(u)$ so that, when conditioning the trees to be of size $2n$, the conditioning is as little degenerated as possible. See \cite[\S7]{survey-trees} for further details.
\end{rem}

\begin{rem}
\label{rem-y-fonction-u}
Using \cref{bij}, we obtain an explicit expression for $y$ in terms of $u$ for $u\geq u_C$. By \cite{tutte_1963}, the series $B$ is algebraic and for all $y \in [0,4/27]$,
\begin{equation}
\label{B-alg}
B(y)^3 - B(y)^2 - 18yB(y) + 27y^2 + 16y = 0.
\end{equation}
This gives an expression of $B'$ in terms of $B$, and taking the resultant between this new equation and \cref{bij} allows to eliminate $B$. Initial conditions then give
\begin{equation}
\label{y-en-fonction-u}
u = \frac{1}{2y{B}'(y)-B(y)+1} \Leftrightarrow y=\(1-\sqrt{1 - \frac1u}\)\(1 - \frac1u\).
\end{equation}
\end{rem}

% !TEX root = FleuratSalvy.tex

\section{Study of the size of the largest blocks}
\label{section:blocks}

\subsection{Subcritical case}
To investigate the distribution of the size of the largest blocks, in the subcritical case, we follow the approach developped in \cite{2Louigi}, which consists in studying the degrees in the block tree of a map. To that end, we rely on results of \emph{condensation} in Galton-Watson trees: exactly one of the nodes has a degree linear in the size. To that end, we rely on Janson's survey \cite{survey-trees}, in which there is a refinement of the study of the largest degree of a subcritical Galton-Watson tree with condensation by Jonsson and Stefánsson \cite{tree-deg}. The condensation phenomenon is visible in the following result where, denoting by $d_{TV}$ the total variation distance, we write $X_n \overset{(d)}{\approx} Y_{n}$ if $d_{TV}(X_n,Y_n) \to 0$ as $n\to\infty$:

\begin{proposition}[\protect{\cite[Theorem 19.34]{survey-trees}}]
\label{survey-tree-subcritical} Let $\mu$ be a probability distribution on $\N$ such that $\mu(0) > 0$, $\E\mu < 1$ and there exists $c$ satisfying $\mu(k) \sim_{k\to\infty} c k^{-5/2}$.
 Let $D_{n,1} \geq D_{n,2} \geq \dots \geq D_{n,n}$ be the ranked list of the number of children of a $\mu$-Galton-Watson tree conditioned to have $n$ edges. Then, letting $\xi_1, \dots, \xi_{n-1}$ be a family of $n-1$ independent random variables of law $\mu$ and $\(\xi_{1}^{(n)}, \dots, \xi_{n-1}^{(n)}\)$ their decreasing reordering, it holds that:
\begin{equation}
\label{lemma3.4Louigi}
\(D_{n,1},\dots,D_{n,n}\) \overset{(d)}{\approx}\(n-\sum_{i=1}^{n-1} \xi_i, \xi_{1}^{(n)},\dots,\xi_{n-1}^{(n)}\).
\end{equation}
\end{proposition}

We combine this proposition with the fact that $\T$ is a Galton-Watson tree under $\mathbb P_u$ to get the following generalization of \cite[Theorem 3.3]{2Louigi}\footnote{One may notice that our $c(1)$ differs from Addario-Berry's $c$ which is because there is a small miscalculation for $c$ in \cite{2Louigi} due to the fact that $T_{n,u}$ does not have $n$ edges but $2n$ edges.} to every value of $u\in(0,9/5)$. This is a rephrasing of the results for trees of \cite{survey-trees}, to which we add the proof of the joint convergence. For $\m$ a map of size $n$, denote by $\mathrm{LB}_{1}(\m) \geq \dots \geq \mathrm{LB}_{b(\m)}(\m)$ the sizes of its blocks in decreasing order. By convention, we set $\mathrm{LB}_{k}(\m) = 0$ if $k>b(\m)$.

\begin{theorem}
\label{souscritique}
Let $u \in \(0,9/5\)$. Recall that $E(u)$ and $c(u)$ are defined in \cref{Eu,cu}. Then,
\[\mathrm{LB}_{1}(\M_{n,u}) = (1-E(u))n + O_{\mathbb{P}}(n^{2/3})\quad\text{and}\quad \mathrm{LB}_{2}(\M_{n,u})=O_{\mathbb{P}}(n^{2/3}).\]
Moreover, the following joint convergence holds:
\begin{equation}
\label{th-taille-blocs-sous-crit}
\(\frac{1}{2n c(u)}\)^{2/3}\((1-E(u))n-\mathrm{LB}_{1}(\M_{n,u}), \(\mathrm{LB}_{j}(\M_{n,u}), j\geq 2\)\) \xrightarrow[n\to\infty]{(d)} \(L_1,\(\Delta L_{(j-1)},j\geq 2\)\)
\end{equation}
where $(L_t)_{t\in[0,1]}$ is a Stable process of parameter $3/2$ such that
\[\E{e^{-sL_1}} = e^{\Gamma(-3/2)s^{3/2}}\]
and $\Delta L_{(1)} \geq \Delta L_{(2)} \geq \dots$ is the ranked sequence of its jumps.
\end{theorem}

When $u\to 0$, we have $1-E(u)\to1$: as expected, if the map has only one block, its size is $n$.

\begin{rem}
If $(L_t)_{t\in[0,1]}$ is a Stable process of parameter $3/2$ satisfying $\E{e^{-sL_1}} = e^{\Gamma(-3/2)s^{3/2}}$ for $s$ such that $Re(s)\geq 0$; then, it is known that (see \cite[Theorem 1]{bertoin1996levy} and its proof):
\[L_1 \overset{(d)}{=} \lim_{\varepsilon\to 0}\sum_{j:\Delta L_{(j)}\geq \varepsilon} \Delta L_{(j)} - \frac{2}{\sqrt{\varepsilon}}.\]
\end{rem}

\begin{proof}[Proof of \cref{souscritique}]
Recall that the subcritical case corresponds to $u \in \left(0,9/5\right)$, for which we have \[\rho(u)M^2(\rho(u),u)=4/27.\]

We follow essentially the same lines of proof as in \cite{2Louigi}, but refining the arguments so as to establish the joint convergence stated in \cref{th-taille-blocs-sous-crit}. \Cref{tree-critical} shows that the hypotheses of \cref{survey-tree-subcritical} are satisfied in the subcritical case.

Let $\(\xi_i\)_{i\geq1}$ be a family of iid random variables of law $\mu^u$ and let $\(\xi^{(n)}_{1}, \dots, \xi^{(n)}_{n}\)$ be the decreasing reordering of its first $n$ variables (take the convention $\xi^{(n)}_{i} = 0$ if $i> n$). Let us consider the following cumulative process:
\[L^{(n)}_t = \frac{\sum_{i=1}^{\ceil{2nt}} \xi_i - 2ntE(u)}{C(u) (2n)^{2/3}}\quad\text{for}\ t\in[0,1],\quad\text{where}\quad C(u) = 2 c(u)^{2/3}.\]

It is standard \cite[Theorem XVII.5.2]{feller} \cite[Chapter VII, Corollary 3.6]{JS87} that there exists a Lévy process $(L_t)_{t\in[0,1]}$ with Lévy measure $\pi(dx)= x^{-5/2}dx \mathbb{1}_{\{x>0\}}$ so that for $s$ such that $Re(s)\geq0$,
\[\E{e^{-sL_1}} = e^{\Gamma(-3/2)s^{3/2}},\]
and such that the following convergence holds in the Skorokhod topology
\begin{equation}
\label{convergence-process}
\(L^{(n)}_t\)_{t\in[0,1]} \xrightarrow[n\to\infty]{(d)} \(L_t\)_{t\in [0,1]}.
\end{equation}

By definition of the process $L^{(n)}_t$, $\frac{\xi_{i}}{{C(u) (2n)^{2/3}}}$ is its $i$-th jump. In particular, denoting by $\Delta P_t$ the jump of the process $(P_t)$ at time $t$ (which may equal $0$),
\[\frac{\xi^{(2n)}_{1}}{{C(u) (2n)^{2/3}}} = \sup_{0\leq t\leq1} \Delta L^{(n)}_t.\]
By \cite[Chapter VI, Proposition 2.4]{JS87}, \cref{convergence-process} gives
\[\frac{\xi^{(2n)}_{1}}{C(u) (2n)^{2/3}} \xrightarrow[n\to\infty]{(d)} \sup_{0\leq t\leq1} \Delta L_t:=\Delta L_{(1)}.\]
By construction of a Lévy process, $(\Delta L_{(j)})_{j\geq 1}$ has same law as the decreasing rearrangement of the atoms of a Poisson random measure with intensity $\pi$ on $\mathbb{R}^+$ (see e.g. \cite[Theorem 1]{bertoin1996levy}). By denoting $t^{(n)}_1$ the time at which the jump $\xi^{(2n)}_{1}$ of the process $L^{(n)}_t$ is realised, one has:
\[\frac{\xi^{(2n)}_{2}}{{C(u)(2n)^{2/3}}} = \sup_{0\leq t\leq1} \Delta \(L^{(n)}_t - \frac{\xi^{(2n)}_{1}}{{C(u)(2n)^{2/3}}} \mathbb{1}_{t\geq t^{(n)}_1}\)_t.\]
So, applying again \cite[Chapter VI, Proposition 2.4]{JS87}, one gets, denoting by $t_1$ the time of the largest jump of $(L_1)$:
\[\frac{\xi^{(2n)}_{2}}{{C(u)(2n)^{2/3}}} \xrightarrow[n\to\infty]{(d)} \sup_{0\leq t\leq1} \Delta \(L_t-\Delta L_{(1)}\mathbb{1}_{t\geq t_1}\)_t=\Delta L_{(2)}.\]
It is again possible to iterate by subtracting the largest jump: for all $k\geq 1$,
\begin{equation}
\label{cv-sauts}
\frac{1}{{C(u)(2n)^{2/3}}}\(\xi^{(2n)}_{1}, \dots, \xi^{(2n)}_{k}\) \xrightarrow[n\to\infty]{(d)} \(\Delta L_{(1)}, \dots, \Delta L_{(k)}\).
\end{equation}
However, by \cref{survey-tree-subcritical,lemma3.4Louigi}, one has (recall that a map of size $n$ has $2n+1$ components, some of which might be empty):
\begin{equation*}
2\(\mathrm{LB}_{1}(\M_{n,u}), \dots, \mathrm{LB}_{2n+1}(\M_{n,u})\) \overset{(d)}{\approx}\(2n-\sum_{i=1}^{2n} \xi_i, \xi^{(2n)}_{1},\dots,\xi_{2n}^{(2n)}\).
\end{equation*}

Therefore, for all $k \geq 2$ fixed
\begin{align*}
\label{app_lemme_janson}
&\(\frac{(1-E(u))n-\mathrm{LB}_{1}(\M_{n,u})}{\frac12C(u)(2n)^{2/3}}, \frac{\mathrm{LB}_{2}(\M_{n,u})}{\frac12C(u)(2n)^{2/3}}, \dots,\frac{\mathrm{LB}_{k}(\M_{n,u})}{\frac12C(u)(2n)^{2/3}}\)\\
&\overset{(d)}{\approx}
\(\frac{\sum_{i=1}^{2n} \xi_i - 2E(u)n}{C(u)(2n)^{2/3}}, \frac{\xi^{(2n)}_{1}}{C(u)(2n)^{2/3}}, \dots,\frac{\xi^{(2n)}_{k}}{C(u)(2n)^{2/3}}\)\\
&\xrightarrow[n\to\infty]{(d)}
\(L_1, \Delta L_{(1)}, \dots, \Delta L_{(k)}\).
\end{align*}
This allows to conclude since $k$ is arbitrary.

\end{proof}

\subsection{Supercritical case}
\label{large-block-supercritical}
The supercritical case corresponds to $u \in \left(9/5,+\infty\right)$ and $y(u) = \rho(u)M^2(\rho(u),u)\in \left(0,4/27\right)$. Recall that in this case $T$ is distributed under $\mathbb P_u$ as a critical Galton-Watson tree with finite exponential moments by \cref{th-trees} and \cref{tree-critical}.

Properties of the maximum degree of critical Galton-Watson trees have been extensively studied by Janson \cite{survey-trees}, building on work by Meir and Moon \cite{MeirMoon}. For the case where the offspring distribution admits finite exponential moments, Janson shows the following result.

\begin{proposition}[\protect{\cite[Theorem 19.16]{survey-trees}}]
Let $\mu$ be a probability distribution on $\N$ such that $\mu(0) > 0$, and $\mu(k+1)/\mu(k)$ converges to a finite limit as $k \to \infty$. Let $D_{n,i}$ be the $i$-th maximal number of children of nodes in a $\mu$-Galton-Watson tree conditioned to have $n$ edges. Denote by $\rho$ the radius of convergence of $\Phi:t\mapsto \sum_{k\in\N} \mu(k)t^k$, and $\nu= \lim_{x\to\rho^-} x \frac{\Phi'(x)}{\Phi(x)}$. Suppose $\nu > 1$. Then, denoting $k(n) = \max\{k \in \N \mid \mu(k) \geq 1/n\}$, for all $j \geq 1$,
\[D_{n,j} = k(n) + O_{\mathbb{P}}(1).\]
\end{proposition}

In our case, the asymptotic of $k(n)$ can be computed thanks to results about the Lambert $W$ function, which is the compositional inverse of $x \in \R \mapsto xe^x \in [-e^{-1}, +\infty)$. This gives the following theorem.

\begin{theorem}
\label{grandbloc}
Let $u>u_C$. For all fixed $j \geq 1$, it holds as $n\rightarrow\infty$ that
\[\mathrm{LB}_{j}(\M_{n,u}) = \frac{\ln(n)}{2\ln\(\frac{4}{27y(u)}\)} - \frac{5\ln(\ln(n))}{4\ln\(\frac{4}{27y(u)}\)}+O_{\mathbb{P}}(1).\]
\end{theorem}

\begin{proof}
The probability $\mu^u(\{2k\})$ is decreasing with $k$. So, by \cref{asympt-mu}, for $n$ large enough, to study $k(n)$ it is sufficient to study for which $k$ one has
\[c(u) \(\frac{27}{4} \rho(u) M^2(\rho(u),u)\)^k k^{-5/2} \(1+o(1)\) \geq \frac1n.\]
For sake of compactness, set $w(u) = \(\frac{27}{4} \rho(u) M^2(\rho(u),u)\)^{{-1}} =\(\frac{27}{4} y(u)\)^{{-1}}$. Note that $w(u)>1$ since $u>u_C$.
Consequently, the previous inequality is equivalent to
\[w(u)^{k} k^{5/2} \leq c(u)n \(1+o(1)\).\]
Notice that this is equivalent to
\[\frac25 \ln(w(u)) k \cdot e^{\frac25 \ln(w(u)) k}\leq \frac25 \ln(w(u)) (n c(u))^{2/5} \(1+o(1)\).\]
Therefore, $k(n)$ is the largest integer such that:
\[\frac25 \ln(w(u)) k(n) \leq W\(\frac25 \ln(w(u)) (n c(u))^{2/5} \(1+o(1)\)\)\]
where $W$ denotes the Lambert $W$ function. It is known that $W$ satisfies, for $x\to\infty$,
\[W(x) = \ln(x) - \ln(\ln(x)) + o(1),\]
which concludes the proof.
\end{proof}

\subsection{Critical case}
The critical case corresponds to $u=9/5$ and $\rho(u)M^2(\rho(u),u) = 4/27$. As shown in \cref{tree-critical}, the offspring distribution has a power law tail in $c j^{-\alpha-1}$, where $\alpha=3/2\in (1,2)$. In this case, the variance is infinite, so that the method of \cref{large-block-supercritical} cannot be used. However, this case is directly treated in Janson's survey \cite[Example 19.27 and Remark 19.28]{survey-trees}.

\begin{theorem}
\label{size-bloc-critical}
The following convergence holds:
\[\(\frac{\mathrm{LB}_{j}(\M_{n,u_c})}{n^{2/3}}, j \geq 1\) \xrightarrow[n\to\infty]{(d)} \(E_{(j)}, j \geq 1\),\]
where the $\(E_{(j)}\)$ are the ordered atoms of a Point Process $E$ on $[0,\infty]$, satisfying that the random variable $E_{a,b} = \# \(E \cap [a, b]\)$ has a probability generating function convergent for all $z\in\mathbb{C}$ with
\[\E{z^{E_{a,b}}} = \frac{1}{2\pi g(0)} \int_{-\infty}^{\infty} \exp\(c\Gamma(-3/2) (-it)^{3/2} + (z-1)c \int_{a}^b x^{-5/2}e^{itx}dx\)dt,\]
where
\[g:x\mapsto \frac{1}{2\pi} \int_{-\infty}^{\infty}e^{-ixt + c \Gamma(-3/2) (-it)^{3/2}} dt.\]
The intensity measure $\pi$ of $E$ satisfies, for $x>0$,
\[\pi(dx) = c x^{-5/2} \frac{g(-x)}{g(0)} dx,\]
and, for all $j\geq 1$, \[E_{(j)} > 0\qquad\text{almost surely}.\]
\end{theorem}

\begin{rem}
By \cref{tree-critical} and \cite[Proposition 4.3]{arbrestable}\footnote{The result is stated under an aperiodicity hypothesis for the reproduction law, which can be omitted; see the discussion in the proof of \cref{limite-echelle-periodique}.}, one has the convergence of the (appropriately) rescaled Łukasiewicz path of $\T_{n,u_C}$ towards a $3/2$-stable excursion.
Therefore, using \cite[Chapter VI, Proposition 2.4]{JS87} and following the same line of arguments as in the proof of \cref{souscritique}, one gets that the $E_{(j)}$ are distributed like the reordered jumps of a $3/2$-stable excursion (multiplied by a constant factor).
\end{rem}

% !TEX root = FleuratSalvy.tex

\section{Scaling limits}
\label{section:scaling-limits}

The preceding sections exhibited, \textit{via} a study of the block-tree, a phase transition of a combinatorial nature, in terms of the size of the largest blocks, when the parameter $u$ reaches $u_C=9/5$, both for the model on general maps and the one on quadrangulations. The goal of the present section is to expand on this phase transition by considering metric properties of the models in each phase, in the sense of taking \textit{scaling limits}, see \cref{sec:prelim} for definitions.

Because Tutte's bijection commutes with the block decomposition of both models under consideration, as stated in \cref{arbre-bloc-quad-carte}, the combinatorial picture of \cref{section:blocks} is the same for both models. However, obtaining global metric properties under either model requires a good understanding of the metric behaviour of the underlying blocks. As of now, the required results exist only for simple quadrangulations.
Consequently, our scaling limit results are complete only for the quadrangulation model.

In \cref{sec:prelim}, we introduce the relevant formalism to state our scaling limit results, as well as a deviation estimate for the diameters of blocks, which will be useful for all values of $u$.

In \cref{sec:surcritique}, we prove \cref{scaling-limit-super-critical}, which identifies scaling limits simultaneously when $u>u_C$ and $u=u_C$. For both models, there is convergence after suitable rescaling to a random continuous tree, namely a \textit{Brownian tree }when $u>u_C$ and a $3/2$ \textit{stable tree} when $u=u_C$. This convergence holds in the Gromov-Hausdorff-Prokhorov ($\mathrm{GHP}$) sense -- between measured metric spaces -- when maps and quadrangulations are
equipped with the uniform measure on their vertices.

Finally in \cref{sec:sous-critique}, we prove \cref{scaling-limit-subcritical} which deals with the $\mathrm{GHP}$ scaling limit when $u<u_C$. In this phase, the one-big-block identified in \cref{souscritique} converges after rescaling to a scalar multiple of the \textit{Brownian sphere}, and the contribution of all other blocks is negligible. This result is proved only for the quadrangulation model since it relies crucially on the scaling limit result for uniform simple quadrangulations obtained in \cite{add-alb}. No such result is available yet for uniform $2$-connected general maps, although one expects that it should hold.

\subsection{Preliminaries}\label{sec:prelim}
\subsubsection{The Gromov-Hausdorff and Gromov-Hausdorff-Prokhorov topologies}

Originating from the ideas of Gromov, the following notions of metric geometry have become widely used in probability theory to state scaling limit results. We refer the interested reader to \cite{burago_course_2001} for general background on metric geometry and \cite[Section 6]{miermont_tessellations_2007} for an exposition of the main properties of the Gromov-Hausdorff and Gromov-Hausdorff-Prokhorov topologies, and especially their definition \textit{via} correspondences and couplings that we use here.

Define a \textit{correspondence} between two sets $X$ and $Y$ as a subset $C$ of $X\times Y$ such that for all $x\in X$, there exists $y\in Y$ such that $(x,y)\in C$, and \textit{vice versa}. The set of correspondences between $X$ and $Y$ is denoted as $\mathrm{Corr}(X,Y)$.
If $(X,d_X)$ and $(Y,d_Y)$ are compact metric spaces and $C\in\mathrm{Corr}(X,Y)$ is a correspondence, one may define its \emph{distortion}:
\[
\mathrm{dis}(C;d_X,d_Y)=\sup\Bigl\{|d_X(x,\widetilde x)-d_Y(y,\widetilde y)|\colon (x,y)\in C,\,(\widetilde x,\widetilde y)\in C\Bigr\}.
\]
This allows to define the \emph{Gromov-Hausdorff} distance between (isometry classes of) compact metric spaces
\[
d_{\mathrm{GH}}\bigl((X,d_X),(Y,d_Y)\bigr)=\frac{1}{2}\inf\Bigl\{\mathrm{dis}(C;d_X,d_Y)\colon C\in \mathrm{Corr}(X,Y)\Bigr\}.
\]

One can modify this notion of distance in order to get a distance between compact measured metric spaces. For measured spaces $(X,\nu_X)$ and $(Y,\nu_Y)$ such that $\nu_X$ and $\nu_Y$ are \emph{probability} measures, let us denote by $\mathrm{Coupl}(\nu_X,\nu_Y)$ the set of couplings between $\nu_X$ and $\nu_Y$, \emph{i.e.}~the set of measures $\gamma$ on $X\times Y$ with respective marginals $\nu_X$ and $\nu_Y$. Then the \textit{Gromov-Hausdorff-Prokhorov} distance is defined as
\begin{multline*}
d_{\mathrm{GHP}}\bigl((X,d_X,\nu_X),(Y,d_Y,\nu_Y)\bigr)\\
=\inf\Bigl\{\max\Bigl(\tfrac{1}{2}\,\mathrm{dis}(C;d_X,d_Y),\gamma\bigl((X\times Y)\setminus C\bigr)\Bigr)\colon C\in \mathrm{Corr}(X,Y),\,\gamma\in\mathrm{Coupl}(\nu_X,\nu_Y)\Bigr\}.
\end{multline*}
When $(X,d_X)$ and $(Y,d_Y)$ are the same metric space, one can bound this distance by the \emph{Prokhorov distance} between the measures $\nu_X$ and $\nu_Y$. This distance is defined for $\nu_1$ and $\nu_2$ two Borel measures on the same metric space $(X,d)$ by
\begin{equation}\label{def-Prokhorov}
d_{\mathrm{P}}^{(X,d)}(\nu_1,\nu_2)=\inf\Bigl\{\epsilon>0\colon\nu_1(A)\leq \nu_2(A^\epsilon)+\epsilon\text{ and }\nu_2(A)\leq \nu_1(A^\epsilon)+\epsilon,\forall A\in \mathcal B(X)\Bigr\},
\end{equation}
where $A^\epsilon$ is the set of points $x\in X$ such that $d(x,A)<\epsilon$. The bound mentioned above then corresponds to the inequality
\begin{equation}\label{eq-Strassen}
d_{\mathrm{GHP}}\bigl((X,d,\nu_1),(X,d,\nu_2)\bigr)
\leq d_{\mathrm{P}}^{(X,d)}(\nu_1,\nu_2),
\end{equation}
which is a consequence of Strassen's Theorem, see \cite[Section 11.6]{dudley_real_2002}.

Finally we will use the following fact, the proof of which is left to the reader. For $a\in [0,1)$ and Borel probability measures $\mu$, $\nu$ and $\nu'$ on some metric space $(X,d)$, it holds that
\begin{equation}\label{eq-decomp-Prokhorov}
d_{\mathrm P}^{(X,d)}\bigl(a\mu+(1-a)\nu,a\mu+(1-a)\nu'\bigr)=(1-a)d_{\mathrm P}^{(X,d)}(\nu,\nu').
\end{equation}

\subsubsection{Formulation of the GHP-scaling limit problem}

Let us begin by setting the notation for the measured metric spaces that one can canonically associate to the combinatorial objects under consideration.
%\begin{defin}
	We associate to a tree (resp.~map or quadrangulation) the following measured metric spaces:
	\begin{itemize}
		\item For $\t$ a \textit{tree}, denote by ${V}(\t)$ the set of its vertices, by $d_\t$ the distance that the graph distance induces on ${V}(\t)$, by $\nu_\t$ the uniform probability measure on ${V}(\t)$ and by $\underline\t$ the measured metric space $\underline\t=({V}(\t),d_\t,\nu_t)$. Recall that for $v\in\t$, the number of children of $v$ is denoted by $k_v(\t)$.
		\item For $\m$ a \textit{map}, recall that $V(\m)$ is its vertex set, and denote by $d_\m$ the graph distance on $V(\m)$, by $\nu_\m$ the uniform probability measure on $V(\m)$ and by $\underline\m$ the measured metric space $\underline\m=(V(\m),d_\m,\nu_m)$. 
%		Denote also by $\nu_\m$ the \textit{degree-biased} measure on $V(\m)$, \textit{i.e.} $\nu_\m(\{x\})=\mathrm{deg}(x)/(2|E(\m)|)$. Accordingly set $\underline\m=(V(\m),d_\m,\nu_\m)$.
		\item For $\q$ a \textit{quadrangulation}, denote by $V(\q)$ its vertex set, by $d_\q$ the graph distance on $V(\q)$, by $\nu_\q$ the uniform probability measure on $V(\q)$ and by $\underline\q$ the measured metric space $\underline\q=(V(\q),d_\q,\nu_q)$. 
%		Similarly as above define $\underline\q=(V(\q),d_\q,\nu_q)$ its degree-biased version.
	\end{itemize}
%\end{defin}

The problem of finding a $\mathrm{GHP}$-scaling limit consists in finding a suitable rescaling of a sequence of random compact measured metric spaces so that it admits a non-trivial limit in distribution for the $\mathrm{GHP}$-topology. Let us introduce a convenient notation for the rescaling operation on a measured metric space.
%\begin{defin}
	For $\underline X=(X,d,\nu)$ a measured metric space and $\lambda>0$, we denote by $\lambda\cdot \underline X$ the measured metric space $(X,\lambda d,\nu)$.
%\end{defin}

%\begin{defin}
%	Denote $\Omega_0$ the countable set of maps. We define on this measurable space the family of random variables $\bigl(\T,\M,\Q,(\b_v^\M)_{v\in\T},(\b_v^\Q)_{v\in\T}\bigr)$ in the following way:
%	\begin{itemize}
%		\item $\M$ is the canonical random variable on the set of maps $\M\colon\m\mapsto\m$,
%		\item $\Q:=\varphi(\M)$ is the associated quadrangulation by Tutte's bijection,
%		\item $\T:=T_{\M}=T^{(b)}_\Q$ is their common block-tree, see \cref{arbre-bloc-quad-carte},
%		\item $(\b_v^\M)_{v\in\T}$ is the family of $2$-connected blocks of $\M$,
%		\item $(\b_v^\Q)_{v\in\T}$ is the family of simple blocks of $\Q$, which by \cref{arbre-bloc-quad-carte} are also the respective images $\bigl(\varphi(\b_v^\M)\bigr)_{v\in\T}$ of the $2$-connected blocks of $\M$ by Tutte's bijection.
%	\end{itemize}
%	
%	The measurable space $\Omega_0$ can be equipped from \cref{subsec:conseq-proba} with the probability measures $(\mathbb P_u)_{u>0}$ and the respective conditioned measures $(\mathbb P_{n,u})_{n\geq 1}$.
%\end{defin}

%Then we will turn to the statement and proof of our scaling limit results, for critical and supercritical values of the parameter $u$ (section \ref{sec:surcritique}) and then for subcritical values (section \ref{sec:sous-critique}).

\subsubsection{A useful deviation estimate}

We shall now prove a deviation estimate for the diameters of the blocks of $\M$ and $\Q$. It will prove useful for all values of $u>0$. We recall the definition of \emph{stretched-exponential} quantities, as this notion provides a concise way to deal with the probabilities of exceptional events.

\begin{defin}
	A sequence $(p_n)$ of real numbers is said to be \emph{stretched-exponential} as $n\rightarrow\infty$ if there exist constants $\gamma,C,c>0$ such that
	\[
	|p_n|\leq C\exp(-cn^\gamma).
	\]
\end{defin}

As is evident from the definition, if $(p_n)_n$ and $(q_n)_n$ are stretched-exponential sequences, then so are the sequences $(p_n+q_n)_n$, $(p_n q_n)_n$, $(n^\alpha p_n)_n$ and $(n^\alpha \sup_{k\geq n^\beta}p_k)_n$ with arbitrary $\alpha,\beta>0$.

The input we shall rely on to derive our estimate is a deviation estimate for the diameter of \textbf{one} block, in both the case of 2-connected blocks of maps and simple blocks of quadrangulations. 

\begin{proposition}\label{deviation-diam-blocks}
	For any $\epsilon>0$, the probabilities
	\[
	P\({\diam({B}_k^{\mathrm{map}})\geq k^{1/4+\epsilon}}\)
	\qquad \text{ and }\qquad P\({\diam({B}_k^{\mathrm{quad}})\geq k^{1/4+\epsilon}}\)
	\]
	are stretched-exponential as $k\rightarrow\infty$.
\end{proposition}

\begin{proof}
	The estimate for uniform 2-connected maps $(B_k^{\mathrm{map}})_{k\geq 0}$ is obtained from \cite[Theorem 3.7, specialized to $x=1$]{diameter-graphs}. To obtain the estimate for uniform simple blocks of quadrangulations $({B}_k^{\mathrm{quad}})_{k\geq 0}$, one easily checks that for any path of length $l\geq 0$ in a map $\m$, there exists a path with same endpoints and length at most $2l$ in $\varphi(\m)$, its image by Tutte's bijection. Therefore for every map $\m$ one has $\diam(\varphi(\m))\leq2\,\diam(\m)$. In particular $\diam({B}_k^{\mathrm{quad}})\leq2\,\diam({B}_k^{\mathrm{map}})$, and the conclusion follows from the estimate for $(B_k^{\mathrm{map}})_{k\geq 0}$
\end{proof}

%\begin{proof}
%	Denote $\B_\star$ the block of $\M$ with the largest size. Since conditionally on $\T$, the blocks $(\B_v)_{v\in\T}$ are uniform blocks of respective size $(k_v(\T))_{v\in\T}$, we get for all $n\geq k$ the identity in law
%	\[
%	\mathrm{Law}\bigl(\B_k\:;\: \QQsymb\bigr)=\mathrm{Law}\Bigl(\B_\star\:;\: \PPun{\;\;\cdot\;\:\Bigm|\abs{\B_\star}=k ; \abs{\T}=2n+1}\Bigr).
%	\]
%	Therefore we get for all $n\geq k$,
%	\begin{align*}
%	\QQ{\diam(\mathfrak{B}_k)\geq k^{1/4+\epsilon}}
%		&= \PPun{\diam(\B_\star)\geq k^{1/4+\epsilon}\Bigm|\abs{\B_\star}=k ; \abs{\T}=2n+1}\\
%		&\leq \PPun{\diam(\M)\geq k^{1/4+\epsilon}\Bigm|\abs{\B_\star}=k ; \abs{\T}=2n+1}\\
%		&\leq \frac{\PPun{\diam(\M)\geq k^{1/4+\epsilon}\bigm|\abs\T=2n+1}}{\PPun{\abs{\B_\star=k}\bigm| \abs\T=2n+1}}\\
%		&= \frac{\PPun{\diam(\M)\geq k^{1/4+\epsilon}\bigm|\abs\M=n}}{\PPun{\max_{v\in\T}k_v(\T)=2k\bigm| \abs\T=2n+1}}
%	\end{align*}
%	Taking $n=\lambda k$ for some $\lambda\geq 1$, the numerator is $\oexp(k)$ thanks to proposition \ref{prop:deviations-diameter-map}. Since $\mathbf{E}(\mu_{u=1})<1$, there is a condensation phenomenon taking place when conditioning by $\{\abs\T=2n+1\}$ ; and if we chose $n=\lambda k$ with $\lambda:=(1-\mathbf{E}(\mu_{u=1}))^{-1}$, then the numerator decays only polynomially in $k$. {\color{red}\textsc{Référence à mettre}}
%	
%	This proves that $\QQ{\diam(\mathfrak{B}_k)\geq k^{1/4+\epsilon}}$ is indeed $\oexp(k)$.
%\end{proof}

This deviation estimate for the diameter of \textbf{one} block allows to control the deviations of the diameter of \textbf{every} block of $\M_{n,u}$ and $\Q_{n,u}$, in the sense of the following corollary.

\begin{corollary}\label{diam-blocs}
	For all $u>0$ and all $\delta>0$, the probabilities
	\begin{align}
	&P\(\exists v\in\T_{n,u},\,\diam(\b_v^{\M_{n,u}})\geq \max\(n^{1/6},k_v(\T_{n,u})^{(1+\delta)/4}\)\),\quad n\geq 1,\label{deviation-blocs-map}\\
	&P\(\exists v\in\T_{n,u},\,\diam(\b_v^{\Q_{n,u}})\geq \max\(n^{1/6},k_v(\T_{n,u})^{(1+\delta)/4}\)\),\quad n\geq1,\label{deviation-blocs-quad}
	\end{align}
	are stretched-exponential as $n\rightarrow\infty$.
\end{corollary}

\begin{proof}
	Let $\b$ be either a $2$-connected map, or a simple quadrangulation. Then $\diam(\b)$ is bounded by its number of edges, which is $|\b|$ if $\b$ is a map, and $2|\b|$ if it is a quadrangulation. In particular, recalling that the outdegrees in the block-tree are twice the sizes of the respective blocks, we get for all $u>0$ and $n\geq 1$, that
	\[
	\forall v\in\T_{n,u},\,\left[\diam(\b_v^{\M_{n,u}})\leq k_v(\T_{n,u})/2\right]\text{ and }\left[\diam(\b_v^{\Q_{n,u}})\leq \,2\cdot k_v(\T_{n,u})/2\right].
	\]
	Denote by $\mathrm{A}(\M_{n,u})$ the ``bad'' subset of $\T_{n,u}$ made of the vertices $v$ such that both $k_v(\T_{n,u})/2\geq n^{1/6}$ and $\diam(\b_v^{\M_{n,u}}) \geq k_v(\T_{n,u})^{(1+\delta)/4}$. By the above trivial bound on diameters, to show that the probabilities \cref{deviation-blocs-map} are stretched-exponential as $n\rightarrow\infty$, it suffices to see that the probability of the event $\{\mathrm{A}(\M_{n,u})\neq\emptyset\}$ is stretched-exponential as $n\rightarrow\infty$.
	
	By \cref{th-trees}, conditionally on $\T_{n,u}$, each block $\b_v^{\M_{n,u}}$ is sampled uniformly from $2$-connected maps with size $k_v(\T_{n,u})/2$ respectively. Therefore, conditionally on $\T_{n,u}$, for each vertex $v$ in $\T_{n,u}$ we have
	\begin{align*}
	P\bigl(v\in\mathrm{A}({\M_{n,u}})\,|\,\T_{n,u}\bigr)
	&=\indic{k_v(\T_{n,u})/2\geq n^{1/6}}
	%P\({\diam( B_{k_v(\T_{n,u})/2})\geq {k_v(\T_{n,u})}^{(1+\delta)/4}}\)\\
	P\({\diam( B_{k/2}^{\mathrm{map}})\geq {k}^{(1+\delta)/4}}\)\Big|_{k=k_v(\T_{n,u})}\\
	&\leq\sup_{k/2\geq n^{1/6}}P\({\diam( B_{k/2}^{\mathrm{map}})\geq {k}^{(1+\delta)/4}}\)\\
	&\leq \sup_{k\geq n^{1/6}}P\({\diam( B_{k}^{\mathrm{map}})\geq {k}^{(1+\delta)/4}}\).
	\end{align*}
	Since $\T_{n,u}$ has $2n+1$ vertices, this yields by a union bound,
	\[
	P(\mathrm{A}({\M_{n,u}})\neq\emptyset)\leq (2n+1)\sup_{k\geq n^{1/6}/2}P\({\diam( B_{k}^{\mathrm{map}})\geq {k}^{(1+\delta)/4}}\),
	\]
	which is stretched-exponential as $n\rightarrow\infty$ by \cref{deviation-diam-blocks}, as announced. A similar use of \cref{deviation-diam-blocks} proves that the probabilities \cref{deviation-blocs-quad} are stretched-exponential as $n\rightarrow\infty$.
\end{proof}

%	Let us first reason conditionnaly on $T$. We have $\PPu{\;\cdot\:\given\abs\M=n}$-almost-surely,
%	
%	\begin{align*}
%	&\PPu{\exists v \in \T,\quad k_v(\T)\geq n^{1/6},\quad \diam(\B_v)\geq k_v(\T)^{(1+\delta)/4} \given \T\,;\,\abs{\M}=n}\\
%	&\leq \EEu{\sum_{ v \in \T} \indic{ k_v(\T)\geq n^{1/6}} \indic{\diam(\B_v)\geq k_v(\T)^{(1+\delta)/4}} \given \T\,;\,\abs\M=n}\\
%	&= \sum_{ v \in \T} \indic{ k_v(\T)\geq n^{1/6}} \PPu{{\diam(\B_v)\geq k_v(\T)^{(1+\delta)/4}} \given \T\,;\,\abs\M=n}\\
%	&= \sum_{ v \in \T} \indic{ k_v(\T)\geq n^{1/6}} p\bigl(k_v(\T)\bigr)\\
%	&\leq n \sup_{k\geq n^{1/6}}p(k).
%	\end{align*}
%	Integrating over the law of $\T$ under $\PPu{\;\cdot\:\given\abs\M=n}$, we get the (unconditional) bound
%	\[
%	\PPu{\exists v \in \T,\quad k_v(\T)\geq n^{1/6},\quad \diam(\B_v)\geq k_v(\T)^{(1+\delta)/4}\given\abs\M=n}\leq n \sup_{k\geq n^{1/6}}p(k).
%	\]
%	By corollary \ref{corollary:deviation-diam-blocks} we have $p(k)=\oexp(k)$, the latter bound is thus also $\oexp(n)$.
%\end{proof}

%We need two ingredients to understand the behaviour of the sum $\sum_{i=1}^h \Dc_i$ which appears in lemma \ref{lemma:law-dist-along-spine}, we need to understand the tail behaviour of the variable $\Dc$. For our application, we only need to know that it has moments of order $1+\epsilon$.

\subsection{The supercritical and critical cases}\label{sec:surcritique}

\subsubsection{Statement of the result}
\label{subsubsec:statement-of-the-result}
For $1<\theta\leq 2$, let us denote by ${\mathcal T}^{(\theta)}$ a $\theta$-stable Lévy tree equipped with its mass measure. There are several equivalent constructions of these objects. A common way is to define them \textit{via} excursions of $\theta$-stable Lévy processes. Namely, $\mathcal{T}^{(\theta)}$ is the real tree encoded by the height process of an excursion of length one of a $\theta$-stable Lévy process, see \cite{arbrestable}. To fix a normalization for $\mathcal{T}^{(\theta)}$, we consider in the construction an excursion obtained by a cyclic shift from a $\theta$-stable Lévy Bridge with Laplace exponent $\lambda\mapsto \lambda^\theta$. Note that the measured metric space $\mathcal{T}^{(2)}$ corresponds to $\sqrt 2$ times the \emph{Brownian Continuum Random Tree}, which is encoded by an excursion of length $1$ of the standard Brownian motion.
The precise definition \textit{via} excursions is not important for our statement and one can take \cref{limite-echelle-periodique} below as an alternative definition.

\begin{theorem}
	\label{scaling-limit-super-critical}
	There exist positive constants $(\kappa^{\mathrm{map}}_u,\kappa^{\mathrm{quad}}_u)_{u\geq u_C}$ such that we have the following joint convergences in distribution, in the Gromov-Hausdorff-Prokhorov sense:
	%	We set 
	%	\begin{align*}
	%	\kappa^{\mathrm{map}}_u
	%	=E\left[\mathrm d_{\widehat{\B}_u^{\mathrm{map}}}\(\mathrm{root}\bigl(\widehat{\B}_u^{\mathrm{map}}\bigr) , U \)\right] 
	%	\qquad\text{and}\qquad
	%	\kappa^{\mathrm{quad}}_u
	%	=E\left[\mathrm d_{\widehat{\B}_u^{\mathrm{quad}}}\(\mathrm{root}\bigl(\widehat{\B}_u^{\mathrm{quad}}\bigr) , U \)\right].
	%	\end{align*}
	%	Then 
	\begin{enumerate}
		\item 
		If $u>u_C$, we have
		\[
		\frac{\sigma(u)}{\sqrt 2} (2n)^{-1/2}\cdot\(\underline{\T}_{n,u},\underline{\M}_{n,u},\underline{\Q}_{n,u}\)
		\xrightarrow[n\rightarrow\infty]{\mathrm{GHP},\,(d)}
		\({\mathcal{T}}^{(2)},\kappa^{\mathrm{map}}_u\cdot{\mathcal{T}}^{(2)},\kappa^{\mathrm{quad}}_u\cdot{\mathcal{T}}^{(2)}\),
		\]
		where we set
		\begin{equation}\label{eqn-explicite-sigma}
		\sigma(u)^2={1+\frac{4u\, \(y(u)\)^2\,B''\(y(u)\)}{uB\(y(u)\)+1-u}} =
%		{\frac{u^{2}-4 u +3 -\left(u^{\frac{3}{2}}-\sqrt{u}\right) \sqrt{u -1}}{\left(u-1-\sqrt{u(u -1)} \right) \left(5 u -9\right)}}.
		%\left({3-\frac{2u}{\sqrt{(u-1)u}}}\right)^{-1}
		\frac{3 u -3+2 \sqrt{u \left(u -1\right)}}{5 u -9}.
		\end{equation}
		\item 
		If $u=u_C=9/5$, we have
		\[
		\frac{2}{3} (2n)^{-1/3}\cdot\(\underline\T_{n,u_C},\underline\M_{n,u_C},\underline\Q_{n,u_C}\)
		\xrightarrow[n\rightarrow\infty]{\mathrm{GHP},\,(d)}
		\({\mathcal{T}}^{(3/2)},\kappa^{\mathrm{map}}_{u_C}\cdot{\mathcal{T}}^{(3/2)},\kappa^{\mathrm{quad}}_{u_C}\cdot{\mathcal{T}}^{(3/2)}\).
		\]
	\end{enumerate}
	Additionally, the constants $(\kappa^{\mathrm{map}}_u,\kappa^{\mathrm{quad}}_u)_{u\geq u_C}$ can be expressed as follows.
	\begin{equation}
	\label{defin-kappa-u}
	\kappa_u^{\mathrm{map}}=\sum_{j\geq 1}2j\mu^u(2j)\mathcal D_j^{\mathrm{map}}
	\qquad \text{ and }\qquad
	\kappa_u^{\mathrm{quad}}=\sum_{j\geq 1}2j\mu^u(2j)\mathcal D_j^{\mathrm{quad}},
	\end{equation}
	where $\mathcal D_j^{\mathrm{map}}$ (resp.~$\mathcal D_j^{\mathrm{quad}}$) is the expectation of the distance, in a uniform $2$-connected map with $j$ edges (resp.~simple quadrangulation with $j$ faces) of the distance of the root vertex to the base vertex of a uniform corner (resp.~to the closest endpoint of a uniform edge).
\end{theorem}

\begin{figure}
\begin{center}
\includegraphics[width=8cm]{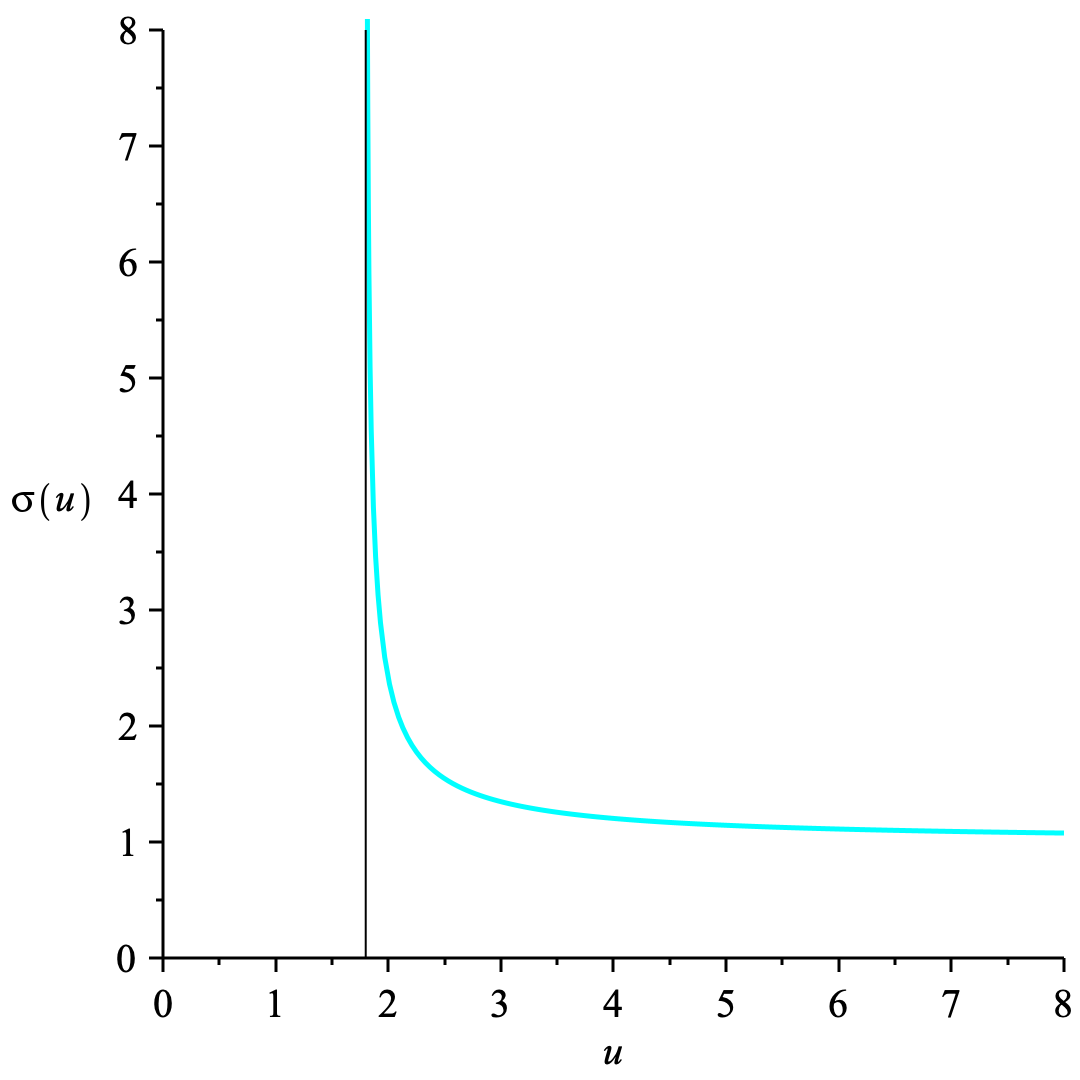}
\end{center}
\caption{Plot of $\sigma$ as a function of $u$. The vertical line corresponds to $u=u_C$.}
\label{plot-sigma-u}
\end{figure}

\begin{rem}\label{rem-sigma}
Let us explain how one gets the second equality of \cref{eqn-explicite-sigma}, which allows to draw \cref{plot-sigma-u}. From the proof of the convergence, one gets:
\[\sigma(u)^2={1+\frac{4u\, (y(u))^2\,B''(y(u))}{uB(y(u))+1-u}}.\]

Differentiating the algebraic equation \cref{B-alg} satisfied by $B$ with respect to $y$ and taking \cref{bij} as a fourth equation gives a polynomial system from which one can get $B(y(u))$ and $B''(y(u))$ as functions of $u$ only (using the resultant or a Gr\"obner basis algorithm). Then, one can conclude from the expression of $\sigma(u)$ in \cref{eqn-explicite-sigma}.
\end{rem}

\begin{rem}The quantity $\kappa_u^{\mathrm{quad}}$ could in principle be obtained \textit{via} the explicit formula obtained in \cite{Bouttier_2010} for the generating function $g_\ell$ of simple edge-rooted quadrangulations with a distinguished edge at prescribed distance $\ell$ from the root vertex.
\end{rem}

\subsubsection{Discussion and overview of the proof}

Let $u\geq u_C$.

Consider a geodesic in either $\M_{n,u}$ or $\Q_{n,u}$ between two distant blocks $\b$ and $\widetilde \b$, respectively indexed by $v$ and $\widetilde v$ in the block-tree. This geodesic must go through all the blocks whose index $w$ in the block-tree $\T_{n,u}$ is on the path from $v$ to $\widetilde v$, in the order induced by this path in the tree.

We have seen in \cref{th-trees} that under the law of $\M_{n,u}$ or $\Q_{n,u}$, the blocks are independent conditionally on the block-tree, and when $u\geq u_C$ they tend to all have non-macroscopic $o(n)$ size by \cref{grandbloc,size-bloc-critical}. One therefore expects that when $n$ is large, the distance between two distant blocks $\b$ and $\widetilde\b$ falls into a \textit{law of large numbers} behaviour and is of the same order as $d_{\T_{n,u}}(v,\widetilde v)$.

According to this heuristic, the macroscopic distances in $\M_{n,u}$ and $\Q_{n,u}$ should be concentrated around a deterministic scalar multiple of the distances in $\T_{n,u}$. But $\T_{n,u}$ is a critical Galton-Watson tree conditioned to have $2n+1$ vertices, with explicit tail asymptotic given by \cref{tree-critical} for its offspring distribution, yielding that its scaling limit is a stable tree.

To make this heuristic work, one needs to understand the typical distribution of degrees on a typical path in the tree. It turns out that on a typical path in a size-$n$ critical Galton-Watson tree, the degrees are asymptotically independent and identically distributed; and moreover they are distributed as the size-biased version of the offspring distribution.
This will be obtained by a spine decomposition for trees, adapted to our context.
%This is reminiscent of local convergence of size-$n$ critical Galton-Watson trees to Kesten's tree, see \cite{Kesten1986SubdiffusiveBO} and \cite{Abraham2013LocalLO}. This tree consists in an infinite spine of vertices with \textit{i.i.d} size-biased degree, on which Galton-Watson trees are grafted. We use here what is known as the \textit{size-biasing relation} for Galton-Watson trees, involving finite versions of Kesten's tree, which we adapt to our setting of a tree with blocks.
%One then needs to understand the typical distances on the blocks along the spine: are those distances \textit{i.i.d}, what is their expectation, how fat are their tails.

We bring the attention of the reader to the fact that a proof similar in spirit has been done for the Gromov-Hausdorff metric in the general abstract setup of enriched trees by Stufler \cite[Theorem 6.60]{Stufler-survey-2020}, and we could readily apply this result to deal with the case $u>u_C$, modulo a technical complication regarding the additivity of distances in the quadrangulation case. When $u=u_C$ however, the distances within blocks have fat tails, so we fall outside the scope of Stufler's result. To deal with this, our last technical ingredient is a suitable large deviation estimate: we show that after an adequate truncation of the variables depending on $n$, large (and moderate) deviation events still have very small probability.

We now proceed with the proof.

\subsubsection{Additivity of the distances along consecutive blocks}

In \cref{additivite-dist-carte,additivite-dist-quad}, we justify that a macroscopic distance is indeed a sum of distances on ``in-between'' blocks, in the case of blocks lying on the same branch in the block-tree.

\paragraph{The map case.} For $\b$ a $2$-connected map, and $l$ an integer in $\{1,\dotsc,2|\b|\}$, let us denote by $D(\b,l)$ the graph distance in $\b$ between its root vertex and the vertex on which lies the $l$-th corner of $\b$ in breadth-first order (or in whatever arbitrary ordering rule is chosen in the block-tree decomposition, see \cref{louigi}).

\begin{figure}
\begin{center}
\includegraphics[width=6.5cm]{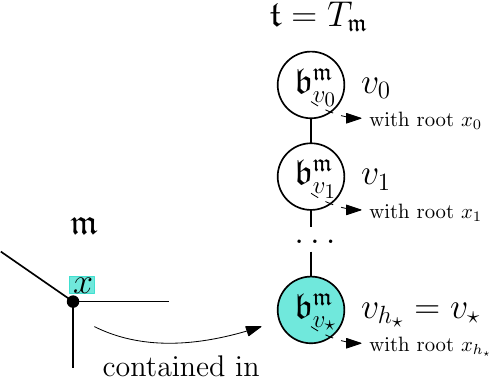}
\hfill
\end{center}
\caption{Situation of \cref{additivite-dist-carte}.}
\label{fig:additivity-distances}
\end{figure}

Fix a vertex $x$ on $\m$.
Let $v_\star$ be the vertex $v$ of the block-tree $\t$, closest to the root of $\t$, such that $x$ is a vertex of $\b_v^\m$. Denote by $h_\star:=h_\t(v_\star)$, and by $(v_i)_{0\leq i\leq h_\star}$ the ancestor line of $v_\star$ in $\t$, with $v_0$ the root and $v_{h_\star}=v_\star$. For $0\leq i\leq h_\star$, let $x_i$ be the root vertex of $\b_{v_i}^\m$. Finally, let $(l_i)_{0\leq i< h_\star}$ be the respective breadth-first index of the corner in $\b^\m_{v_{i}}$ in which the block $\b^\m_{v_{i+1}}$ is attached. The situation is illustrated on \cref{fig:additivity-distances}.

\begin{lemma}\label{additivite-dist-carte}	
	For $0\leq i\leq h_\star$, we have
	\[
	d_\m(x,x_i)=d_{\b_{v_\star}^\m}(x,x_{h_\star})+\sum_{i\leq j<h_\star}D(\b_{v_j}^\m,l_j).
	\]
\end{lemma}

\begin{proof}
	By definition, $D(\b_{v_j}^\m,l_j)=d_\m(x_j,x_{j+1})$ for $i\leq j<h_\star$. We get by the triangle inequality that the left-hand-side is at most the right-hand side.
	Therefore it suffices to show that any geodesic path in $\m$ from $x$ to $x_i$ visits each of the points $(x_j)_{i<j\leq h_\star}$, in decreasing order of $j$.
	
	Let $j$ be such that $i<j \leq h_\star$. Denote by $\t_j$ the tree of descendants of $v_j$ in $\t$ (rooted in $v_j$) and also $\m_j$ and $\widetilde \m_j$ the submaps of $\m$ made of the blocks $(b_v^\m)_{v\in\t_j}$ and $(b_v^\m)_{v\in\t\setminus\t_j}$ respectively. By the recursive description of the block-tree, the submaps $\m_j$ and $\widetilde\m_j$ share only the vertex $x_j$. But $x$ is a vertex of $\m_j$ since $v_\star$ is a descendant of $v_j$, and $x_i$ is a vertex of $\widetilde\m_j$ since $v_i$ is an ancestor of $v_j$. Hence any injective path between $x$ and $x_i$ must visit $x_j$ in decreasing order of $j$; and in particular for a geodesic path.
	
	Notice that it does not require the $(x_j)$ to be mutually distinct. This concludes the proof.
\end{proof}

\paragraph{The quadrangulation case.}

A slight complication arises for quadrangulations because the ``interface'' between two blocks is a double edge, containing two vertices instead of a single vertex in the map case. At first sight it is thus unclear through which of these vertices a geodesic should go. We show that there is a canonical choice: the vertex between those two which is closest to the root vertex. This relies crucially on the fact that quadrangulations are bipartite.

Fix a quadragulation $\q$. For $\b$ a simple block of $\q$, and $l$ an integer in $\{1,\dotsc,2|\b|\}$, let us denote by $D_\q(\b,l)$ the graph distance in $\b$ between the endpoint of the $l$-th edge of $\b$ in the ordering described thereafter, and the endpoint of the root edge of $\b$ which is closest to the root vertex of $\q$. The order on the edges of $\b$ that we use is the image of the lexicographic order on vertices of $\t$ \textit{via} the block-tree decomposition. This is consistent with the ordering of corners in the map case.

Fix a vertex $x$ on $\q$. Similarly to the map case, let $v_\star$ be the vertex $v$ of the block-tree $\t$, closest to the root of $\t$, such that $x$ is a vertex of $\b_v^\q$. Define accordingly $h_\star=h_\t(v_\star)$ and the ancestor line $(v_i)_{0\leq i\leq h_\star}$ of $v_\star$ in the block-tree $\t$, with $v_0$ the root and $v_{h_\star}=v_\star$. Let also $(x_i)_{0\leq i\leq h_\star}$ be the respective root vertex of $\b_{v_i}$. Finally, let $(l_i)_{0\leq i< h_\star}$ be the respective breadth-first index of the edge in $\b^\q_{v_{i}}$ to which the root edge of $\b^\q_{v_{i+1}}$ is attached.

\begin{lemma}\label{additivite-dist-quad}
	For all $0\leq i \leq h_\star$, there exists $\delta_{x,x_i}\in\{0,\pm 1,\pm 2\}$ such that
	\[
	d_\q(x,x_i)=\delta_{x,x_i}+d_{\b_{v_\star}^\q}(x,x_{h\star})+\sum_{i\leq j<h_\star}D_\q(\b_{v_j}^\q,l_j).
	\]
\end{lemma}

\begin{proof}
	The idea is quite similar in principle as in the preceding lemma, except that consecutive blocks share two vertices in the quadrangulation case, instead of one.
	
	For $0\leq j\leq h_\star$, denote by $y_j$ the endpoint of the root-edge of $\b_{v_j}^\q$ which is closest to the root vertex of $\q$. In particular $y_0$ is the root vertex of $\q$. Let $0\leq i\leq h_\star$. Then, by construction, $y_i$ and $x_i$ are adjacent to the root edge of $\b_{v_i}^\q$. Notice that a geodesic from $x$ to $x_i$ must visit at least one of the endpoints of the root-edge of $\b_{v_\star}$, which is at distance $0$ or $1$ of $y_{h_\star}$, and that $x_i$ and $y_i$ are at distance $0$ or $1$. Therefore, there exists some $\delta_{x,x_i}\in\{0,\pm 1,\pm 2\}$ such that
	\[
	d_\q(x,x_i)=\delta_{x,x_i}+d_{\b_{v_\star}^\q}(x,x_{h\star})+d_\q(y_{h_\star},y_i).
	\]
	We shall prove the following, which are sufficient to conclude:
	\begin{enumerate}
		\item For $0\leq i\leq h_\star$, it holds that $d_\q(y_{h_\star},y_i)=\sum_{i\leq j<h_\star}d_\q(y_{j+1},y_j)$;
		\item For $0\leq j<h_\star$, it holds that $d_\q(y_{j+1},y_j)=D_\q(\b_{v_j}^\q,l_j)$.
	\end{enumerate}
	It is even sufficient to show the following:
	\begin{equation}\label{eq-additivite-dist-quad}
		\forall 0\leq i\leq j\leq k\leq h_\star,\quad d_\q(y_i,y_k)=d_\q(y_i,y_j)+d_\q(y_j,y_k).
	\end{equation}
	Indeed, assuming \cref{eq-additivite-dist-quad} holds, by applying it iteratively, we directly get that $d_\q(y_{h_\star},y_i)=\sum_{i\leq j<h_\star}d_\q(y_{j+1},y_j)$. To verify the second set of identities, recall that $y_j$ is defined as the endpoint of the root edge of $b_{v_j}^\q$ which is closest to the root vertex $y_0$ of $\q$. Denote by $y'_j$ the other endpoint. Then, for $0\leq j<h_\star$ we have
	\[
	D_\q(\b_{v_j}^\q,l_j)=\min\(d_{\b_{v_j}^\q}(y_{j+1},y_j),d_{\b_{v_j}^\q}(y'_{j+1},y_j)\)=\min\(d_\q(y_{j+1},y_j),d_\q(y'_{j+1},y_j)\).
	\]
	The first equality comes from the definition of $D_\q(\b_{v_j}^\q,l_j)$, and the second one from the fact that within a block $\b$ of $\q$, the graph distance respective to $\q$ and the graph distance respective to $\b$ coincide.
	Then, assuming \cref{eq-additivite-dist-quad}, it holds that
	\[
	d_\q(y_{j+1},y_j)=d_\q(y_{j+1},y_0)-d_\q(y_{j},y_0)
	\leq d_\q(y'_{j+1},y_0)-d_\q(y_{j},y_0)
	\leq d_\q(y'_{j+1},y_j),
	\]
	where the first inequality comes from the definition of $y_{j+1}$, and the second inequality from triangle inequality. In particular, the above minimum is $d_\q(y_{j+1},y_j)$ and we have, as needed, $d_\q(y_{j+1},y_j)=D_\q(\b_{v_j}^\q,l_j)$.
	
	It still remains to prove \cref{eq-additivite-dist-quad}. Let us first prove the case $i=0$ and then deduce the general case. Let $0\leq j\leq k\leq h_\star$, and let $\gamma$ be a geodesic path from $y_0$ to $y_k$. If $\gamma$ visits $y_j$, we readily have
	\begin{equation}\label{eq-additivite-dist-a-racine}
	d_\q(y_0,y_k)=d_\q(y_0,y_j)+d_\q(y_j,y_k).
	\end{equation}
	Otherwise it visits $y'_j$, and denote by $\gamma_1$, $\gamma_2$ the portions of $\gamma$ form $y_0$ to $y'_j$, and from $y'_j$ to $y_k$ respectively. By definition of $y_j$, we have $d_\q(y_0,y_j)\leq d_\q(y_0,y'_j)$. But since $\q$ is a quadrangulation, it is bipartite and the inequality is strict $d_\q(y_0,y_j)< d_\q(y_0,y'_j)$. Form $\widetilde\gamma_1$ the concatenation of a geodesic path from $y_0$ to $y_j$ and of the oriented edge $(y_j,y'_j)$. Then, from the strict inequality we mentioned, $\mathrm{len}(\widetilde\gamma_1)\leq \mathrm{len}(\gamma_1)$, and in particular the concatenation of $\widetilde\gamma_1$ and $\gamma_2$ is a geodesic path from $y_0$ to $y_k$ which visits $y_j$. Therefore, the identity \cref{eq-additivite-dist-a-racine} also holds.
	
	Finally, let us deduce the case $i\neq 0$. Let $0\leq i\leq j\leq k\leq h_\star$. We have
	\begin{multline*}
	d_\q(y_i,y_k)=d_q(y_0,y_k)-d_\q(y_0,y_i)
	=\(d_q(y_0,y_k)-d_\q(y_0,y_j)\)+\(d_q(y_0,y_j)-d_\q(y_0,y_i)\)\\
	=d_\q(y_j,y_k)+d_\q(y_i,y_j).
	\end{multline*}
	This proves \cref{eq-additivite-dist-quad} and concludes the proof of \cref{additivite-dist-quad}.

\end{proof}

\subsubsection{Scaling limit and largest degree of critical Galton-Watson trees}

A slight technical complication that arises in our setting is that the block-tree has a \emph{lattice} offspring distribution with span $2$, in the sense of the following definition

\begin{defin}
	A measure $\mu$ on $\Z$ is called \emph{lattice} if its support is included in a subset $b+d\Z$ of $\Z$, with $d\geq 2$. The largest such $d$ is called its \emph{span}. If $d=1$, $\mu$ is called \emph{non-lattice}.
\end{defin}

The results that we need \cite[Theorem 3]{kortchemski_simple_2013} are stated for non-lattice offspring distributions. This turns out to be purely for convenience and we state the following more general result that is suited to our needs.

We recall that a probability distribution $\mu$ with mean $m_\mu$ is said to be in the domain of attraction of a stable law of index $\theta\in(1,2]$ if there exist positive constants $(C_n)_{n\geq 0}$ such that we have the following convergence in distribution
\begin{equation}\label{eq:cvgence-distrib}
\frac{U_1+\dots+U_n-nm_\mu }{C_n}\xrightarrow[n\rightarrow\infty]{(d)}X^{(\theta)},
\end{equation}
where $(U_1,\dots U_n)$ are i.i.d.~samples of the law $\mu$, and $X^{(\theta)}$ is a random variable with Laplace transform $E\left[\exp(-\lambda X^{(\theta)})\right]=\exp(\lambda^\theta)$.

\begin{proposition}\label{limite-echelle-periodique}
	For all $1<\theta\leq 2$, there exists a random measured metric space ${\mathcal T}^{(\theta)}=\(\mathcal T^{(\theta)},d^{(\theta)},\nu^{(\theta)}\)$ satisfying the following scaling limit result.
	
	Let $\mu$ be a probability distribution on $\Z_{\geq 0}$, with $\mu(1)\neq 1$, and which is assumed to be critical. Assume additionally that it is in the domain of attraction of a stable law of index $\theta\in(1,2]$.
	Let $d\geq 1$ be the span of the measure $\mu$. Then under those assumptions, we have
	\begin{enumerate}
		\item
		For all $m$ large enough, the $\mathrm{GW}_\mu(\mathrm d T)$-probability that $T$ has $dm$ edges is positive. This probability is equivalent to $c_\theta/{ (C_{dm}dm)}$ for some constant $c_\theta>0$.
		\item
		If we denote by $T_n$ a $\mathrm{GW}_\mu$-tree conditioned to have $n$ edges, then
		\[
		\(\frac{C_{dm}}{dm}\)\cdot \underline{T}_{dm}\xrightarrow[m\rightarrow\infty]{\quad (d)\quad}
		{\mathcal{T}}^{(\theta)},
		\]
		in the \emph{Gromov-Hausdorff-Prokhorov} sense,
		with $(C_n)_{n\geq 0}$ the sequence in \cref{eq:cvgence-distrib}.
		\item The largest degree in $T_{dm}$ is of order at most $C_{dm}$, in the sense that for any $\epsilon>0$
		\[
		P\(\exists v\in T_{dm}, k_v(T_{dm})\geq (C_{dm})^{1+\epsilon}\)\xrightarrow[m\rightarrow\infty]{} 0.
		\]
	\end{enumerate}
\end{proposition}

\begin{proof}
	The first statement can be obtained by a straightforward adaptation of the proof of \cite[Lemma 1]{kortchemski_simple_2013}, which relies on a local limit theorem and the cycle lemma. We specify below how this local limit theorem should be adapted. The cycle lemma adapts straightforwardly.
	
	For the second statement, let us justify that \cite[Theorem 3]{kortchemski_simple_2013} still applies when the \emph{non-lattice} (or \emph{aperiodic}) assumption is dropped, but with the number of vertices $n+1$ taken only along the subsequence $(dm+1)_{m\geq 0}$. This will prove functional convergence of the contour functions of the trees $(T_{dm})_m$ when properly rescaled, to the contour function of ${\mathcal{T}}^{(\theta)}$. This convergence of contour functions is sufficient to get the announced Gromov-Hausdorff-Prokhorov convergence.
	
	The local limit theorem \cite[Theorem 2, (ii)]{kortchemski_simple_2013} changes as follows
		\[
		\lim_{n\rightarrow\infty}\sup_{k\in\Z}\left|\frac{a_n}{d}P\(Y_n=k\)-p_1\(\frac{k}{a_n}\)\right|=0.
		\]
		See for instance \cite[Theorem 4.2.1]{ibragimov_independent_1971}. Notice that the only difference with the \textit{non-lattice} ($d=1$) local limit theorem is the factor $1/d$ in the last display.
	Examining the details of Kortchemski's arguments, this extra $1/d$ factor would appear only in the discrete absolute continuity relations which are used in the proof. But in each instance, it would appear in both the numerator and denominator of some fraction. Hence the fraction simplifies and this factor has no impact on the proof, which carries without change, except that the integer $n$, which in the paper is the number of \emph{vertices}, should now only be taken in $d\Z+1$.
	
	Finally, in order to get the third statement, one can take as a basis the local limit theorem above. From this, one can get the functional convergence of the Łucasiewicz path of $T_{dm}$, when it is rescaled by $dm$ in time and $C_{dm}$ in space. In particular, $(C_{dm})^{-1}$ times the largest degree in $T_{dm}$ is tight, and one obtains the claimed probabilistic bound. One could for instance use the same arguments as in the proof of \cite[Proposition 3.4]{kortchemski2021large}
\end{proof}

\Cref{scaling-lim-trees} then just identifies the explicit scaling constants in specific instances of the above-mentioned scaling limit theorem.

\begin{corollary}\label{scaling-lim-trees}
	Let $\mu$ be a critical probability distribution on $\Z_{\geq0}$ with span $d\geq 1$, and with $\mu(1)\neq 1$. Denote by $T_n$ a $\mathrm{GW}_\mu$-tree conditioned to have $n$ edges, for $n\in d\Z$ large enough.
	Then the following holds.
	\begin{enumerate}
		\item
		If $\mu$ has finite variance $\sigma^2$, then $P(|T|=dm)\sim cm^{-3/2}$ for some constant $c>0$, and
		\[
		(dm)^{-1/2}\cdot \underline{T}_{dm}\xrightarrow[m\rightarrow\infty]{\quad \mathrm{GHP},\,(d)\quad}
		\frac{\sqrt 2}{\sigma}\cdot{\mathcal{T}}^{(2)}.
		\]
		Additionally for all $\epsilon>0$ the largest degree of $T_{dm}$ is $o(m^{1/2+\epsilon})$ in probability.
		\item
		If $\mu\([x,+\infty)\)\underset{x\rightarrow\infty}{\sim} c x^{-\theta}$ for some $ c >0$ and $\theta\in(1,2)$, then $P(|T|=dm)\sim c'_\theta m^{-(1+1/\theta)}$ for some constant $c'_\theta>0$, and
		\[
		(dm)^{-(1-1/\theta)}\cdot \underline{T}_{dm}\xrightarrow[m\rightarrow\infty]{\quad \mathrm{GHP},\,(d)\quad}
		\left[\frac{\theta-1}{ c \,\Gamma(2-\theta)}\right]^{1/\theta} \cdot{\mathcal{T}}^{(\theta)}.
		\]
		Additionally for all $\epsilon>0$ the largest degree of $T_{dm}$ is $o(m^{1/\theta+\epsilon})$ in probability.
	\end{enumerate}
\end{corollary}

\begin{proof}
	Note that in the case where $\nu$ has finite exponential moments, \cite{marckert_depth_2003} treats the case of lattice distributions. That would suffice for our applications when $u> u_C$. We still need the second statement to treat the case $u=u_C$.
	Let us apply the preceding proposition and identify the right constants, in these two cases.
%	\begin{enumerate}
%		\item

		\paragraph{Statement 1.} If $\mu$ has finite variance $\sigma^2$, then by the {Central Limit Theorem}, for i.i.d. samples $(U_i)_i$ of the law $\mu$, we have the convergence in distribution
		\[
		\frac{U_1+\dots+U_n-n}{\sigma\cdot n^{1/2}}\xrightarrow[n\rightarrow\infty]{(d)}\mathcal{G},
		\]
		where $\mathcal G$ is a standard normal variable. In particular, $\mathcal G$ has the same law as $\frac{1}{\sqrt 2} \cdot X^{(2)}$.
		Therefore, the hypotheses of \cref{limite-echelle-periodique} are satisfied, with 
		\[
		C_n=\tfrac{\sigma}{\sqrt 2}\cdot n^{1/2},
		\]
		and the conclusion follows from this proposition.
%		\item

		\paragraph{Statement 2.} We consider the case where $\mu\([x,+\infty)\)\underset{x\rightarrow\infty}{\sim} c x^{-\theta}$ with $\theta\in(1,2)$ and $ c >0$. Let $U:=U_1$ and let us also introduce the notation
		\begin{align*}
		M_1(x)&=\int_x^\infty \mu(\mathrm d y)=\mu\([x,+\infty)\)\\
		M_2(x)&=\int_x^\infty M_1(y)\,\mathrm d y\\
		M_3(x)&=\int_0^x M_2(y)\mathrm d y.
		\end{align*}
		The function $M_3$ is non-decreasing and using the assumed tail asymptotic of $\mu$, one has the asymptotic $M_3(x)\sim cx^{2-\theta}/(2-\theta)(\theta-1)$. We may therefore use the Karamata Tauberian theorem \cite[Theorem 1.7.1]{bingham_regular_1989} to get
		\[
		\widehat M_3(h)\sim \frac{c\,\Gamma(3-\theta)}{(2-\theta)(\theta-1)}h^{\theta-2}
		=\frac{c\,\Gamma(2-\theta)}{\theta-1}h^{\theta-2},
		\]
		where $\widehat M_3$ is the \emph{Laplace-Stieltjes transform} of $M_3$, defined~---~\emph{e.g.} in \cite[Paragraph 1.7.0b]{bingham_regular_1989}~---~as
		\[\widehat M_3(h) = h \int_0^\infty e^{-hx} M_3(x) dx\]
		for all $h$ for which the integral converges absolutely. Then, if we integrate by parts three times, we obtain
		\begin{align*}
		E\left[\exp(-h(U-1))\right]=\int_0^\infty\mathrm e^{-h(x-1)}\,\mu(\mathrm d x)
		&= \mathrm e^{h}-h\mathrm e^h M_2(0)+h^3\mathrm e^h\int_0^\infty \mathrm e^{-hx}M_3(x)\,\mathrm d x\\
		&= \mathrm e^{h}-h\mathrm e^h M_2(0)+h^2\mathrm e^h\, \widehat{M}_3(h),
		\end{align*}
		 
		This, together with the fact that $M_2(0)=1$ since it is the expectation of $\mu$, yields the following expansion when $h\rightarrow0^+$,
		\begin{equation}\label{eq-dvpt-Laplace-stable}
		E\left[\exp(-h(U-1))\right]\underset{}{=}1+\frac{ c \,\Gamma(2-\theta)}{\theta-1}\cdot h^\theta\(1+o(1)\).
		\end{equation}
		
		Now, if we set
		\[
		C_n=\(\frac{ c \,\Gamma(2-\theta)}{\theta-1}\)^{1/\theta}n^{1/\theta},
		\]
		and plug $h=\lambda/C_n$ into \cref{eq-dvpt-Laplace-stable}, we get for all $\lambda\geq 0$,
		\begin{align*}
		E\left[\exp\(-\lambda\,\tfrac{U_1+\dots+U_n-n}{C_n}\)\right]
		= \(E\left[\exp(-\tfrac{\lambda}{C_n}(U-1))\right]\)^n 
		\xrightarrow[n\rightarrow\infty]{}\exp\(\lambda^\theta\).
		\end{align*}
		Hence there is convergence in distribution of $\tfrac{U_1+\dots+U_n-n}{C_n}$ to $X^{(\theta)}$, as required in \cref{limite-echelle-periodique}. So this proposition applies with the above-chosen sequence $(C_n)_n$, and the conclusion follows.
%	\end{enumerate}
\end{proof}

\subsubsection{Scaling limit of critical Galton-Watson trees equipped with a random measure}

We shall need a version of the GHP scaling limits in \cref{scaling-lim-trees}, when the trees $\t$ under consideration are equipped with some random measure on their vertices, instead of the uniform measure $\nu_\t$. Let us describe more specifically our setting.

Let $\mu$ be a probability measure on $\Z_{\geq 0}$ and $\eta=(\eta_k)_{k\geq 0}$ be a family of Borel probability measures on $\R_{\geq 0}$. We shall define an enriched version $\mathcal{L}(\mu,\eta)$ of the Galton-Watson law $\mathrm{GW}(\mu)$, defined on the set of pairs $(\t,f)$ such that $\t$ is a tree and $f$ is a non-negative function $f\colon V(\t)\rightarrow\R_{\geq 0}$. Namely, to sample a random pair $(\T, \mathbf f)$ with law $\mathcal{L}(\mu,\eta)$ first sample $\T$ according to $\mathrm{GW}(\mu)$, and then sample conditionally on $\T$ the variables $\mathbf f(v)$ for ${v\in\T}$, independently of each other, according to the laws $\eta_{k_v(\T)}(\mathrm d x)$ respectively.

In particular, the random non-negative function $\mathbf f$ defines a random measure on $V(\T)$ assigning weight $\mathbf f(v)$ to the vertex $v$. We shall use the same notation $\mathbf f$ for this measure, and denote by $|\mathbf f|$ its total weight.

\begin{proposition}\label{scaling-lim-trees-random-measure}
	Let $\mu$ be a critical offspring distribution with span $d\geq 1$ such that $\mu(1)\neq 1$. Let also $(\eta_k)_{k\geq 0}$ be Borel probability laws that are supported on $\R_{\geq 0}$.
	For $n\in d\Z$ large enough, denote by $(T_n,\mathbf f_n)$ a sample of the law $\mathcal{L}(\mu,\eta)$ conditioned to the event $\{|\T|=n\}$.
	If the annealed measure $\sum_{k}\mu(k)\eta_k (\mathrm d s)$ admits a positive and finite first moment, then the following holds.
	\begin{enumerate}
		\item
		If $\mu$ has finite variance $\sigma^2$, then
		\[
		(dm)^{-1/2}\cdot \biggl(V({T}_{dm}),d_{T_{dm}},\frac{\mathbf f_{dm}}{|\mathbf f_{dm}|}\biggr)\xrightarrow[m\rightarrow\infty]{\quad \mathrm{GHP},\,(d)\quad}
		\frac{\sqrt 2}{\sigma}\cdot{\mathcal{T}}^{(2)}.
		\]
		\item
		If $\mu\([x,+\infty)\)\underset{x\rightarrow\infty}{\sim} c x^{-\theta}$ for some $ c >0$ and $\theta\in(1,2)$, then
		\[
		(dm)^{-(1-1/\theta)}\cdot \biggl(V({T}_{dm}),d_{T_{dm}},\frac{\mathbf f_{dm}}{|\mathbf f_{dm}|}\biggr)\xrightarrow[m\rightarrow\infty]{\quad \mathrm{GHP},\,(d)\quad}
		\left[\frac{\theta-1}{ c \,\Gamma(2-\theta)}\right]^{1/\theta} \cdot{\mathcal{T}}^{(\theta)}.
		\]
	\end{enumerate}
\end{proposition}

%\paragraph{A functional Law of Large Numbers}
We shall first prove a rather general functional law of large numbers for the cumulative sum $s\mapsto\sum_{i\leq sn}\mathbf f_n(v_i)$, where $(v_i)_i$ are the vertices of $T_n$ listed in depth-first order.
%This will allow us to deduce \cref{scaling-lim-trees-random-measure} from \cref{scaling-lim-trees} by constructing a coupling of the measures $\nu_{T_{dm}}$ and $\mathbf f_{dm}/|\mathbf f_{dm}|$.
%We postpone the proof of \cref{scaling-lim-trees-random-measure} to after the proof of the following lemma.
\begin{lemma}\label{lemma-cvg-trees-LLN}
	Let $\mu$ be a critical offspring distribution with span $d\geq 1$, and with $\mu(1)\neq 1$. Let also $(\eta_k)_{k\geq 0}$ be Borel probability laws on $\R_{\geq 0}$.
	For $n\in d\Z$ large enough, denote by $(T_n,\mathbf f_n)$ a sample of the law $\mathcal{L}(\mu,\eta)$ conditioned to the event $\{|\T|=n\}$.
	Assume that the annealed measure $\sum_{k}\mu(k)\eta_k (\mathrm d s)$ admits a positive and finite first moment and denote by $\overline\eta> 0$ its expectation. Assume also that $\mu$ is in the domain of attraction of a stable distribution of index $\alpha$ with $1< \alpha\leq 2$.
	Then there holds the following convergence in probability
	\[
	\sup_{s\in[0,1]}\biggl|\frac{1}{n}\textstyle \sum_{i\leq sn}\mathbf f_n(v_i)-s\overline\eta \biggr|\xrightarrow[n\rightarrow\infty,\,n\in d\Z]{P}0,
	\]
	where $(v_0,\dots,v_n)$ are the vertices of $T_n$ listed in depth-first order.
\end{lemma}

\begin{proof}
	Let us denote by $(\widetilde v_i)_{0\leq i\leq n}$ a uniform cyclic shift of the sequence $(v_i)_{0\leq i\leq n}$, that is to say $\widetilde v_i=v_{i+\tau_n \,\mathrm{mod}\,(n+1)}$ for all $0\leq i\leq n$, where $\tau_n$ is a uniformly random element of $\{0,\dots,n\}$, sampled independently from other variables. Then an elementary re-arranging of sums yields that
	\[
	\sup_{s\in[0,1]}\biggl|\frac{1}{n} \sum_{i\leq sn}\mathbf f_n(v_i)-s\overline\eta \biggr|
	\leq
	2\sup_{0\leq s \leq t\leq 1}\biggl|\frac{1}{n} \sum_{sn\leq i\leq tn}\mathbf f_n(\widetilde v_i)-(t-s)\overline\eta \biggr|.
	\]
	Distinguishing upon whether $s$ and $t$ are smaller than $1/2$, and cutting the sum at $1/2$ in the case $s<1/2<t$, we can bound further
	\begin{align*}
	&\sup_{0\leq s \leq t\leq 1}\biggl|\frac{1}{n} \sum_{sn\leq i\leq tn}\mathbf f_n(\widetilde v_i)-(t-s)\overline\eta \biggr|\\
	&\leq
	2\sup_{0\leq s \leq t\leq 1/2}\biggl|\frac{1}{n} \sum_{sn\leq i\leq tn}\mathbf f_n(\widetilde v_i)-(t-s)\overline\eta \biggr|
	+
	2\sup_{1/2\leq s \leq t\leq 1}\biggl|\frac{1}{n} \sum_{sn\leq i\leq tn}\mathbf f_n(\widetilde v_i)-(t-s)\overline\eta \biggr|\\
	&=
	2\sup_{0\leq s \leq t\leq 1/2}\biggl|\frac{1}{n} \sum_{sn\leq i\leq tn}\mathbf f_n(\widetilde v_i)-(t-s)\overline\eta \biggr|
	+
	2\sup_{0\leq s \leq t\leq 1/2}\biggl|\frac{1}{n} \sum_{sn\leq i\leq tn}\mathbf f_n(\widetilde v_{i+\ceil{n/2}})-(t-s)\overline\eta \biggr|\\
	&\leq
	4\sup_{0 \leq t\leq 1/2}\biggl|\frac{1}{n} \sum_{ i\leq tn}\mathbf f_n(\widetilde v_i)-t\overline\eta \biggr|
	+
	4\sup_{0 \leq t\leq 1/2}\biggl|\frac{1}{n} \sum_{ i\leq tn}\mathbf f_n(\widetilde v_{i+\ceil{n/2}})-t\overline\eta \biggr|.
	\end{align*}
	Now notice that $(\widetilde v_{i+\ceil{n/2}})_{0\leq i\leq n}$ is itself a uniform cyclic shift of the sequence $(v_i)_{0\leq i\leq n}$, so that the second term in the last display has the same law as the first one, and we only need to bound this one. We have reduced the problem to showing that the following convergence in probability holds
	\begin{equation}\label{cvg-cyclic-shift-LLN}
	\sup_{0 \leq t\leq 1/2}\biggl|\frac{1}{n} \sum_{ i\leq tn}\mathbf f_n(\widetilde v_i)-t\overline\eta \biggr|
	\xrightarrow[n\rightarrow\infty,\,n\in d\Z]{P}0.
	\end{equation}
	
	We now appeal to the so-called cycle lemma, see \cite[Paragraph 6.1]{PitmanBook} and more precisely Lemma 6.1 for the cycle lemma and Lemma 6.3 for its application to trees. In our setting it implies that the cyclically shifted sequence of degrees $(k_{\widetilde v_i}(T_n))_{0\leq i\leq n}$ has the same law as that of an i.i.d. sequence $(\xi_i)_{0\leq i\leq n}$ of samples of the law $\mu$ conditioned to the event $\{\sum_{0,\leq i\leq n}(\xi_i-1)=-1\}$.
	Now recall that conditionally on $T_n$, each variable $\mathbf f(v_i)$ is sampled according to the law $\eta_{k_{v_i}(T_n)}$ and independently of the family $(\mathbf f(v_j))_{j\neq i}$.
	Therefore the identity in distribution obtained from the cycle lemma admits a straightforward generalization for the cyclically shifted sequence $\bigl(k_{\widetilde v_i}(T_n),\mathbf f_n(\widetilde v_i)\bigr)_{0\leq i\leq n}$.
	More precisely, let $(\xi_i,X_i)_{i\geq 0}$ be an i.i.d.~sequence such that $\xi_0$ has law $\mu$, and such that conditionally on $\xi_0$ the variable $X_0$ has law $\eta_{\xi_0}$.
	Then, there holds the following identity in distribution
	\[
	\mathrm{Law}\Bigl(\bigl(k_{\widetilde v_i}(T_n),\mathbf f_n(\widetilde v_i)\bigr)_{0\leq i\leq n};P\Bigr)
	=
	\mathrm{Law}\Bigl(\bigl(\xi_i,X_i\bigr)_{0\leq i\leq n};P\left(\,\cdot\,\Bigm|\textstyle\sum_{0\leq i\leq n}(\xi_i-1)=-1\right)\Bigr).
	\]
	Using the Markov property at time $\floor{n/2}$ for the random walk $(\sum_{0\leq i \leq k}(\xi_i-1))_{k\geq 0}$, we get for every non-negative Borel function $F\colon(\Z\times\R)^{\floor{n/2}+1}\rightarrow\R_{\geq 0}$ the following
	\begin{multline*}
		E\left[F\Bigl(\bigl(k_{\widetilde v_i}(T_n),\mathbf f_n(\widetilde v_i)\bigr)_{0\leq i\leq n/2}\Bigr)\right]
			=	E\left[F\Bigl(\bigl(\xi_i,X_i\bigr)_{0\leq i\leq n/2}\Bigr)\frac{\indic{\textstyle\sum_{0\leq i\leq n}(\xi_i-1)=-1}}{P\left(\textstyle\sum_{0\leq i\leq n}(\xi_i-1)=-1\right)}\right]\\
			=	E\left[F\Bigl(\bigl(\xi_i,X_i\bigr)_{0\leq i\leq n/2}\Bigr)\frac{q_{n-\floor{n/2}}(-1-\textstyle\sum_{0\leq i\leq n/2}(\xi_i-1))}{q_{n+1}(-1)}\right],
	\end{multline*}
	where we used the notation $q_k(j)=P\bigl(\sum_{1\leq i \leq k}(\xi_i-1)=j\bigr)$.
	Let us remark that there exists $n_0$ such that $q_n(-1)\neq 0$ for all the integers $n\geq n_0$ which belong to $d\Z$, and that
	\begin{equation}\label{absolute-continuity-finite}
	\sup_{n\geq n_0,\, n\in d\Z}\:\sup_{j\in \Z}\frac{q_{n-\floor{n/2}}(l)}{q_{n+1}(-1)}<+\infty.
	\end{equation}
	Indeed, we may use the local limit theorem \cite[Theorem 4.2.1]{ibragimov_independent_1971} which covers the case of random walks on $\Z$ whose increments have law a (possibly non-aperiodic) distribution in the domain of attraction of a stable distribution with index $\alpha\in(1,2]$, such as the random walk $\bigl(\sum_{0\leq i \leq k}(\xi_i-1)\bigr)_{k\geq 0}$. This gives us
	\[
	\lim_{k\rightarrow\infty}\sup_{j\in\Z}\left|\frac{B_k}{d}q_k(j)-g\left(\frac{j}{B_k}\right)\right|=0,
	\]
	where $g$ is the density function of some stable distribution with index $\alpha$ satisfying notably $g(0)\neq0$, and where $(B_k)_k$ is a sequence of numbers such that $(k^{-1/\alpha}B_k)_k$ is slowly varying by \cite[Paragraph 2.2]{ibragimov_independent_1971}.
	We easily deduce \cref{absolute-continuity-finite} from the last display. Therefore there exists a constant $C>0$ such that for every non-negative Borel function $F\colon(\Z\times\R)^{\floor{n/2}+1}\rightarrow\R_{\geq 0}$, we have for $n\geq n_0$,
	\[
	E\left[F\Bigl(\bigl(k_{\widetilde v_i}(T_n),\mathbf f_n(\widetilde v_i)\bigr)_{0\leq i\leq n/2}\Bigr)\right]
	\leq C\cdot E\left[F\Bigl(\bigl(\xi_i,X_i\bigr)_{0\leq i\leq n/2}\Bigr)\right].
	\]
	We deduce for every $\epsilon>0$ and every $n\geq n_0$,
	\begin{multline}\label{LLN-comparison-unconditioned}
	P\left(\sup_{0\leq t\leq 1/2}\biggl|\frac{1}{n} \sum_{i\leq tn}\mathbf f_n(\widetilde v_i)-t\overline\eta \biggr|\geq \epsilon\right)
	\leq C\cdot 
	P\left(\sup_{0\leq t\leq 1/2}\biggl|\frac{1}{n} \sum_{i\leq tn}X_i-t\overline\eta \biggr|\geq \epsilon\right).
	\end{multline}
	Notice that the variables $(X_i)_{i\geq 0}$ are i.i.d.~with mean $\overline\eta$ by definition. By the strong law of large numbers, it holds almost surely that for all $t\in[0,1]\cup\mathbb Q$,
	\[
	\frac{1}{N}\sum_{ i\leq tN} X_i\xrightarrow[N\rightarrow\infty]{}t\overline\eta .
	\]
	Since the variables $(X_i)$ are non-negative, the left-hand-side is a (random) non-decreasing function of $t$ for all $N\geq 1$. In particular, the pointwise almost sure convergence above yields by Dini's theorem almost sure convergence in the sup norm, namely
	\[
		\sup_{0\leq t\leq 1/2}\biggl|\frac{1}{n} \sum_{i\leq tn}X_i-t\overline\eta \biggr|
		\xrightarrow[n\rightarrow\infty,\,n\in d\Z]{\mathrm{a.s.}}0.
	\]
	Combining this with \cref{LLN-comparison-unconditioned}, we obtain the desired convergence in probability \cref{cvg-cyclic-shift-LLN} and this concludes the proof.
\end{proof}

\begin{proof}[{Proof of \cref{scaling-lim-trees-random-measure}}]
Let $U$ be sampled uniformly and independently of other variables and for $n\in d\Z$ large enough, let $(v_0,\dots,v_n)$ be the vertices of $T_n$ listed in depth-first order. We denote by $x_n(U)$ the vertex $v_{\floor{(n+1)U}}$. Let also $y_n(U)$ be the vertex $v_{k_n(U)}$, where $k_n(U)$ is the smallest index $k\in\{0,\dots,n\}$ such that $\sum_{i\leq k}\mathbf f_n(v_i)\geq U|\mathbf f_n|$. By construction, conditionally on $(T_n,\mathbf f_n)$, the random vertex $x_n(U)$ has law $\nu_{T_n}$ and $y_n(U)$ has law $\mathbf f_n/|\mathbf f_n|$. Now by \cref{lemma-cvg-trees-LLN}, the sequence of functions $(s\mapsto \frac{1}{|\mathbf f_n|}\sum_{i\leq sn}\mathbf f_n(v_i))_{n}$ converges in probability for the uniform norm to the identity function $s\mapsto s$ when $n$ tends to $\infty$ in $d\Z$. We deduce using the definition of $k_{n}(U)$ that
$
\left|\frac{k_n(U)}{n}-U\right|
%\xrightarrow[n\rightarrow\infty,\,n\in d\Z]{P}0.
$
converges to $0$ in probability, and in particular that the same goes for
$
\left|\frac{k_n(U)}{n+1}-\frac{\floor{(n+1)U}}{n+1}\right|
%\xrightarrow[n\rightarrow\infty,\,n\in d\Z]{P}0.
$.

Let $\theta=2$ if $\mu$ has finite variance as in case 1.~of the statement, or let $\theta$ be such that $\mu\([x,+\infty)\)\underset{x\rightarrow\infty}{\sim} c x^{-\theta}$ for some $ c >0$ and $\theta\in(1,2)$ as in case 2.~of the statement. Let us set $D_n= n^{-(\theta-1)/\theta}\,d_{T_n}$ the rescaled distance function on $V(T_n)$ and $h_n\colon s\in[0,1]\mapsto n^{-(\theta-1)/\theta}\,h_{T_n}(v_{\floor{(n+1)s}})$ be the rescaled height process of $T_n$. Using the following well-known bound on distances in a tree
\[
d_{T_n}(v_i,v_k)\leq h_{T_n}(v_i)+ h_{T_n}(v_k)-2\inf_{j\in\{i,\dots,k\}}h_{T_n}(v_j)+2,
\]
we get the bound
\begin{align*}
D_n(x_n(U),y_n(U))
&=\frac{1}{n^{(\theta-1)/\theta}} \,d_{T_n}(v_{k_n(U)},v_{\floor{(n+1)U}})\\
&\leq 2\,\omega_{h_n}\left(\left|\tfrac{k_n(U)}{n+1}-\tfrac{\floor{(n+1)U}}{n+1}\right|\right)+\frac{2}{n^{(\theta-1)/\theta}},
\end{align*}
where $\omega_{h_n}(\delta)=\sup_{|x-y|\leq\delta}|h_n(x)-h_n(y)|$ is the modulus of continuity of $h_n$ defined for all $\delta>0$.

We justified in the proof of \cref{scaling-lim-trees} that \cite[Theorem 3]{kortchemski_simple_2013} applies, even if $\mu$ is not assumed to be aperiodic in our setting. This theorem tells us in particular that the rescaled height process $h_n$ converges in distribution as $n$ tends to infinity to some limit, in the Skhorokhod topology. Since the limit is almost surely continuous, properties of the Skhorokhod topology imply that the convergence actually holds in distribution with respect to the topology of uniform convergence. By characterization of tightness for this topology, we have for all $\epsilon>0$,
\[
\lim_{n\rightarrow\infty}\limsup_{\delta\rightarrow0}\,P\left(\omega_{h_n}(\delta)\geq \epsilon\right)=0,
\]
from which we deduce that
\begin{equation}\label{conv-coupling-vertices}
D_n(x_n(U),y_n(U))\xrightarrow[n\rightarrow\infty,\,n\in d\Z]{P}0.
\end{equation}

Recall that conditionally on $(T_n,\mathbf f_n)$, the vertices $x_n(U)$ and $y_n(U)$ have law $\nu_{T_n}$ and $\mathbf f_n/|\mathbf f_n|$ respectively.
This yields using the definition \cref{def-Prokhorov} a bound for the Prokhorov distance between these two measures
\begin{align*}
d_{\mathrm P}^{V(T_n),D_n}\left(\nu_{T_n},\tfrac{\mathbf f_n}{|\mathbf f_n|}\right)
&\leq\inf\Bigl\{\epsilon>0\colon P\left(D_n(x_n(U),y_n(U))\geq\epsilon\Bigm|T_n\right)\leq\epsilon \Bigr\}.
\end{align*}
In particular, we have for $\epsilon>0$,
\begin{align*}
P\left(d_{\mathrm P}^{V(T_n),D_n}\left(\nu_{T_n},\tfrac{\mathbf f_n}{|\mathbf f_n|}\right)\geq \epsilon\right)
&\leq P\Bigl(P\left(D_n(x_n(U),y_n(U)) \geq\epsilon\bigm|T_n\right)\geq\epsilon\Bigr)\\
&\leq\epsilon^{-1}P\bigl(D_n(x_n(U),y_n(U)) \geq\epsilon\bigr),
\end{align*}
where we used Markov's inequality to get the last upper bound. By \cref{conv-coupling-vertices}, we get the convergence in probability
\[
d_{\mathrm P}^{V(T_n),D_n}\left(\nu_{T_n},\tfrac{\mathbf f_n}{|\mathbf f_n|}\right)
\xrightarrow[n\rightarrow\infty,\,n\in d\Z]{P}0.
\]
By inequality \cref{eq-Strassen}, we deduce that
\[
d_{\mathrm{GHP}}\left(\left(V(T_n),D_n,\nu_{T_n}\right),\left(V(T_n),D_n,\tfrac{\mathbf f_n}{|\mathbf f_n|}\right)\right)
\xrightarrow[n\rightarrow\infty,\,n\in d\Z]{P}0.
\]
We conclude the proof by combining the last display with \cref{scaling-lim-trees}.
\end{proof}

\subsubsection{The spine decomposition and size-biased laws}

In this section we present a \emph{size-biasing} relation for the block-tree, in the sense of \cite{Lyons1995ConceptualPO}. Actually, we extend in a straightforward way this size-biasing relation to our setting, where we have a Galton-Watson tree and some decorations, namely the blocks.
More precisely, consider the following measure on maps with a distinguished vertex of their block tree $(\m,v)$
\[
\mathbb P_u(\mathrm d\M)\sum_{v\in \T}\delta_v(\mathrm d V_\star),
\]
where $\delta_v$ is the Dirac measure $A\mapsto \delta_v(A)=\indic{v\in A}$.
Then this $\sigma$-finite measure can be decomposed as a sum of probability measures $\sum_{h\geq 1}\widehat{\mathbb P}_{u,h}(\mathrm d \M,\mathrm d V_\star)$, where under $\widehat{\mathbb P}_u^h$ the vertex $V_\star$ has height $h$ in $\T$, its ancestors' degrees having \emph{size-biased} law as defined below.
The present section makes that precise.

{\paragraph{Description of $\widehat {\mathbb P}_{u,h}$.}
	
\begin{defin}
	Let $\nu$ be a probability distribution on $\mathbb Z_{\geq 0}$ with finite expectation $m_\nu$. Then the \emph{size-biased} distribution $\widehat\nu$ is defined by
	\[
	\forall k\in\mathbb Z_{\geq 0},\quad
	\widehat\nu(k)=\frac{k\,\nu(k)}{m_\nu}.
	\]
\end{defin}}

%\begin{defin}\label{def-GW-biaise}
	When $\nu$ is a (sub-)critical offspring distribution with $\nu(0),\nu(1)\neq 0$, denote by $\(\widehat{\mathrm{GW}}_{\nu,h}\)_{h\geq 0}$ the following family of laws, on the sets of discrete trees with a distinguished vertex at height $h$ respectively. It may be described algorithmically:
	\begin{itemize}
		\item Each vertex will either be mutant or normal, and their number of offspring are sampled independently from each other;
		\item Normal vertices have only normal children, whose number is sampled according to~$\nu$;
		\item Mutant vertices of height less than $h$ have a number of children sampled according to the size-biased distribution $\widehat\nu$, all of which are normal except one, chosen uniformly, which is mutant;
		\item The only mutant vertex at height $h$ reproduces like a normal vertex and is the distinguished vertex $V_\star$.
	\end{itemize}
	This yields a pair $(T,V_\star)$, where $T$ is a discrete tree and $V_\star$ is a distinguished vertex of $T$ with height $h$.	
	We denote by $(V_i)_{0\leq i\leq h_\T(V_\star)}$ the ancestor line of $V_\star$, and $L_i$ the order of $V_{i+1}$ in the children of $V_i$ respectively. Observe that the construction gives that $(k_{V_i}(\T))_i$ are i.i.d. with law $\widehat\nu$, and conditionally on those variables, the variables $(L_i)_i$ are independent with uniform law on $\{1,\dotsc,k_{V_i}(\T)\}$ respectively.
%\end{defin}

We may now define the family of probability measures $(\widehat{\mathbb P}_{u,h})_{h\geq 0}$ as follows. Let $h\geq 0$.
\begin{itemize}
	\item Sample $(\T,V_\star)$ according to the law $\widehat{\mathrm{GW}}_{\mu^u,h}$.
	\item For each $v\in\T$, sample independently and uniformly a 2-connected map $\b_v^\M$ with $k_v(\T)/2$ edges.
	\item Build the map $\M$ whose block decomposition is $(b_v^\M)_{v\in\T}$, and $\Q=\varphi(\M)$ its image by Tutte's bijection.
\end{itemize}
We are now equipped to state the size-biasing relation.

\begin{proposition}\label{loi-biaisee}
	For $u\geq u_C$, the $\sigma$-finite measure $\mathbb P_u(\mathrm d\M)\sum_{v\in \T}\delta_v(\mathrm d V_\star)$ on maps with a distinguished vertex of their block-tree decomposes as the following sum of probability measures,
	\[
	\mathbb P_u(\mathrm d\M)\sum_{v\in \T}\delta_v(\mathrm d V_\star)=\sum_{h\geq 0}\widehat{\mathbb P}_{u,h}(\mathrm d\M,\mathrm d V_\star).
	\]
\end{proposition}

\begin{proof}
	The standard size-biasing relation for (sub-)critical Galton-Watson trees reads
	\[
	\mathrm{GW}_{\nu}(\mathrm d\t)\sum_{v\in \t}\delta_v(\mathrm d v_\star)
	=\sum_{h\geq 0}(m_\nu)^h\cdot\delta_h(h_\t(v_\star))\cdot \widehat{\mathrm{GW}}_{\nu,h}(\mathrm d \t,\mathrm d v_\star).
	\]
	When $u\geq u_C$, the offspring distribution $\mu^u$ is critical, so $m_{\mu^u}=1$. Specializing the last display to $\nu=\mu^u$ and to the value of $\t$ corresponding to the block-tree of some map $\m$, this gives for all such $(\m,v_\star)$,
	\[
	\mathrm{GW}_{\mu^u}(\t)=\sum_{h\geq0}\delta_h(h_\t(v_\star))\widehat{\mathrm{GW}}_{\mu^u,h}(\t,v_\star).
	\]
	Therefore, if we multiply both sides by $\prod_{v\in\t}\frac{1}{b_{k_v(\t)/2}}$, we get the following by \cref{th-trees}:
	\[
	\mathbb P_u(\m)=
	\sum_{h\geq 0}\delta_h(h_\t(v_\star))\cdot \widehat{\mathrm{GW}}_{\mu^u,h}(\t,v_\star)\cdot\prod_{v\in\t}\frac{1}{b_{k_v(\t)/2}}=\sum_{h\geq 0}\widehat{\mathbb P}_{u,h}(\m,v_\star).
	\]
	Since $\sum_{v\in\t}\delta_v(v_\star)=1$, the last display expresses the measure $\mathbb P_u(\mathrm d\M)\sum_{v\in\t}\delta_v(\mathrm d V_\star)$ as a sum of the probability measures $(\widehat{\mathbb P}_{u,h})_{h\geq 0}$.
\end{proof}

\paragraph{Probabilistic properties of \texorpdfstring{$\widehat{\mathbb P}_{u,h}$}{P_{u,h}}.}

Since we need metric information on blocks whose size follows the size-biased law $\widehat \mu^u$, let us introduce adequate notation.
Let $u\geq u_C$. Denote by $\xic_u$ a sample of the distribution $\widehat{\mu}^u$ on some probability space $(\Omega,P)$. Then jointly define the random variables $\widehat{\B}_u^{\mathrm{map}}$ and $\widehat{\B}_u^{\mathrm{quad}}$ as sampled uniformly among respective blocks with size $\xic_u/2$, in such a way that they are linked by Tutte's bijection, i.e.~their joint law satisfies
	\[
	\(\widehat{\B}_u^{\mathrm{map}},\widehat{\B}_u^{\mathrm{quad}}\)
	\overset{(d)}{=}
	\(B^{\mathrm{map}}_{\xic_u/2},B^{\mathrm{quad}}_{\xic_u/2}\).
	\]
	Furthermore, conditionally on $\widehat \xi_u$, sample independently $U$ a uniform label in $\{1,\dotsc,\widehat \xi_u\}$.
	This yields the following $4$-tuple 
	\[\bigl(\xic_u\:,\:
	\widehat{\B}_u^{\mathrm{map}}\:,\:
	\widehat{\B}_u^{\mathrm{quad}}\:,\:
	U\bigr).
	\]

\begin{lemma}\label{lemma:law-dist-along-spine}
	For all $h\geq 1$, we have the identity in law
	\[
	\mathrm{Law}\left(\Bigl(
	k_{V_i}(\T)\:,\:
	\b^\mathrm{\M}_{V_i}\:,\:
	\b^{\mathrm{\Q}}_{V_i}\:,\:
	L_{i}
	\Bigr)_{0\leq i<h}\: ;\:\widehat{\mathbb P}_{u,h}\right)
	=
	\left[
	\mathrm{Law}\Bigl(
	\bigl(\xic_u\:,\:
	\widehat{\B}_u^{\mathrm{map}}\:,\:
	\widehat{\B}_u^{\mathrm{quad}}\:,\:
	U\bigr)
	\: ;\:
	P
	\Bigr)\right]^{\otimes h}
	\]
	where $\mathrm{Law}(X;Q)$ is the law of $X$ under $Q$.
\end{lemma}

\begin{proof}
	Recall that under $\widehat{\mathbb P}_{u,h}$, the pair $(\T,V_\star)$ has law $\widehat{\mathrm{GW}}_{\mu^u,h}$. By definition of the law $\widehat{\mathrm{GW}}_{\mu^u,h}$, the ancestor line of the distinguished vertex $V_\star$ in $\T$ is made of mutant vertices. This means that the family $(k_{v_i}(\T))_{0,\leq i<h}$ is i.i.d. sampled from the size-biased distribution $\widehat\mu^u$, which is the law of $\xic_u$, and that independently of each other, each $V_{i+1}$ has uniform rank $L_i$ among the $k_{V_i}(\T)$ children of $V_i$. Hence we have the identity in law
	\[
	\mathrm{Law}\left(\Bigl(
	k_{V_i}(\T)\:,\:
	L_{i}
	\Bigr)_{0\leq i<h}\: ;\:\widehat{\mathbb P}_{u,h}\right)
	=
	\left[
	\mathrm{Law}\(\bigl(
	\xic_u\:,\:
	U
	\bigr);P\)\right]^{\otimes h}.
	\]
	
	Now under $\widehat{\mathbb P}_{u,h}$ the conditional law of the blocks $(\b^{\M}_v)_{v\in\T}$ with respect to $\T$ is that of independent blocks, sampled uniformly from blocks with size $(k_{V}(\T)/2)_{v\in\T}$ respectively.	
	In particular, the blocks $(\b^\M_{V_i})_{0\leq i<h}$ are sampled independently, uniformly from blocks with size $(k_{V_i}(\T)/2)_{0\leq i<h}$ respectively. Therefore the preceding identity in law extends to the following one
	\[
	\mathrm{Law}\left(\Bigl(
	k_{V_i}(\T)\:,\:
	\b^\mathrm{\M}_{V_i}\:,\:
	L_{i}
	\Bigr)_{0\leq i<h}\: ;\:\widehat{\mathbb P}_{u,h}\right)
	=
	\left[
	\mathrm{Law}\(\bigl(
	\xic_u\:,\:
	\widehat{\B}_u^{\mathrm{map}}\:,\:
	U
	\bigr);P\)\right]^{\otimes h}.
	\]
	Finally, recall from \cref{arbre-bloc-quad-carte} that $\b_{V_i}^\mathrm{\Q}$ is the image of $\b_{V_i}^\mathrm{\M}$ by Tutte's bijection. Since by definition $\widehat\B_u^{\mathrm{quad}}$ is also the image of $\widehat\B_u^{\mathrm{map}}$ by this bijection, the identity in law extends to the one in the proposition.
\end{proof}

We get in particular from \cref{lemma:law-dist-along-spine} that the variables $\bigl(D(\b_{V_i}^\M,L_i)\bigr)_{0\leq i<h}$ are i.i.d.~under $\widehat{\mathbb{P}}_{u,h}$. It is a bit less clear that the variables $\bigl(D_\Q(\b_{V_i}^\Q,L_i)\bigr)_{0\leq i<h}$ from \cref{additivite-dist-quad} are also i.i.d., since they seem to simultaneously depend on global metric properties of $\Q$.

\begin{lemma}\label{prop-indep-dist-epine-quad}
	Denote by $D(\b,l)$ the distance in a simple quadrangulation $\b$ between its root vertex and the closest endpoint of the $l$-th edge in the order induced by the block-tree decomposition, the same order as the one introduced before \cref{additivite-dist-quad}. Then for all $h\geq 1$, there is the identity in law
	\[
	\mathrm{Law}\left(
	(D_\Q(\b^{\mathrm{\Q}}_{V_i},L_i))_{0\leq i<h}\: ;\:\widehat{\mathbb P}_{u,h}\right)
	=
	\left[
	\mathrm{Law}\Bigl(
	D(\widehat{\B}_u^{\mathrm{quad}},U)
	\: ;\:
	P
	\Bigr)\right]^{\otimes h}.
	\]
\end{lemma}

\begin{proof}
	Recall from the notation introduced for \cref{additivite-dist-quad} that for $\b$ a simple block of a quadrangulation $\q$, and $l$ an integer in $\{1,\dotsc,2|\b|\}$, $D_\q(\b,l)$ is the graph distance in $\b$ between the endpoints of the $l$-th edge of $\b$ in breadth-first order, and the endpoint of the root edge of $\b$ which is closest to the root vertex of $\q$.
	
	Denote by $\b\mapsto F(\b)$ the mapping which reverses the rooted oriented edge of a simple quadrangulation. Introduce also for $\b$ a simple quadrangulation, $f_\b$ the permutation of $\{1,\dotsc,2|\b|\}$ which maps the breadth-first order on $\b$ to the breadth-first-order on $F(\b)$. Finally, define the event $\mathcal E_i$ that the root vertex of $b_{V_i}^\Q$ is closer to the root vertex of $\Q$ than the other endpoint of the root edge of $b_{V_i}^\Q$. Then by definition, for all $0\leq i< h$ we have that
	\[
	D_\Q(\b^{\mathrm{\Q}}_{V_i},L_i)
	= \indicBis{\mathcal E_i}\cdot D\Bigl(\b^{\mathrm{\Q}}_{V_i},L_i\Bigr)
		+\(1-\indicBis{\mathcal E_i}\)\cdot D\Bigl(F(\b^{\mathrm{\Q}}_{V_i}),f_{\b^{\mathrm{\Q}}_{V_i}}(L_i)\Bigr).
	\]
	Let $\mathcal F_i$ denote the sigma-algebra of the variables $(
	k_{V_j}(\T)\:,\:
	\b^\mathrm{\M}_{V_j}\:,\:
	\b^{\mathrm{\Q}}_{V_j}\:,\:
	L_{j}
	)_{0\leq j<i}$. Then by \cref{lemma:law-dist-along-spine}, we have that the tuple $(
	k_{V_i}(\T)\:,\:
	\b^\mathrm{\M}_{V_i}\:,\:
	\b^{\mathrm{\Q}}_{V_i}\:,\:
	L_{i}
	)$ is independent of $\mathcal F_i$, and has the same law as
	$
	\bigl(\xic_u\:,\:
	\widehat{\B}_u^{\mathrm{map}}\:,\:
	\widehat{\B}_u^{\mathrm{quad}}\:,\:
	U)$.
	Now the crucial point is that the event $\mathcal E_i$ is $\mathcal F_i$-measurable, since it can be decided whether or not it holds by looking only at the first $i$ blocks on the spine. In particular it is independent of $(
	k_{V_j}(\T)\:,\:
	\b^\mathrm{\M}_{V_j}\:,\:
	\b^{\mathrm{\Q}}_{V_j}\:,\:
	L_{j}
	)_{ j\geq i}$.
	This implies the following
	\begin{align*}
	&\mathrm{Law}\left(
	\(D_\Q(\b^{\mathrm{\Q}}_{V_i},L_i)\)_{0\leq i<h}\: ;\:\widehat{\mathbb P}_{u,h}\right)\\
	&=\bigotimes_{0\leq i<h}
		\biggl[
		\widehat{\mathbb P}_{u,h}(\mathcal E_i)\cdot
			\mathrm{Law}\Bigl(
			D(\widehat{\B}_u^{\mathrm{quad}},U)
			\Bigr)
		+(1-\widehat{\mathbb P}_{u,h}(\mathcal E_i))\cdot
			\mathrm{Law}\Bigl(
			D\bigl(F(\widehat{\B}_u^{\mathrm{quad}}),f_{\widehat{\B}_u^{\mathrm{quad}}}(U)\bigr)
			\Bigr)
		\biggr].
	\end{align*}
	The proposition is therefore proved if we justify the identity in law
	\begin{equation}\label{eq-identite-loi-dist-blocs-quad}
	D(\widehat{\B}_u^{\mathrm{quad}},U)
	\overset{(d)}{=}
	D\bigl(F(\widehat{\B}_u^{\mathrm{quad}}),f_{\widehat{\B}_u^{\mathrm{quad}}}(U)\bigr).
	\end{equation}
	To check this, first notice that $F$ is a bijection since it is involutive, so that in particular the uniform law on simple quadrangulations with $k$ edges is invariant under $F$. By definition, for $\b$ a simple quadrangulation, $f_\b$ is also a bijection so that the uniform measure on $\{1,\dotsc,2|\b|\}$ is invariant under it. Denoting $U_k$ a uniform random variable on $\{1,\dotsc,2k\}$, this gives for each $k\geq 1$ the identity in law
	\[
	D(B_k^{\mathrm{quad}},U_k)
	\overset{(d)}{=}
	D\bigl(F(B_k^{\mathrm{quad}}),f_{B_k^{\mathrm{quad}}}(U_k)\bigr).
	\]
	Since the pair $(\widehat{\B}_u^{\mathrm{quad}},U)$ is the $\xic_u/2$-mixture of the laws $(B_k,U_k)_{k\geq 1}$, the identity in law \cref{eq-identite-loi-dist-blocs-quad} also holds and this conludes the proof.
\end{proof}

\paragraph{Moments of typical distances in a size-biased block.}

We may now examine how fat are the tails of this i.i.d. family of distances along the spine, which we wish to sum.

\begin{proposition}\label{lemma:tail_D-hat}
	Let $D$ be either the variable $D(\widehat\B^{\mathrm{map}}_u,U)$ or $D(\widehat\B^{\mathrm{quad}}_u,U)$. Then for $u> u_C$, there exists $\epsilon>0$ such that $E[\exp(tD)]<\infty$ for all real $t<\epsilon$. And for $u=u_C$, we have $E\left[D^\beta\right]<\infty$ for all $0<\beta<2$.
\end{proposition}

\begin{proof}
	The variable $D$ is defined as a distance in $\widehat\B_u$, where $\widehat\B_u$ is either $\widehat\B^{\mathrm{map}}_u$ or $\widehat\B^{\mathrm{quad}}_u$. Hence it suffices to prove that the above moments are finite when we replace $D$ by $\diam(\widehat \B_u)$.
	
	Let $u>u_C$. Then $\diam(\widehat \B_u)\leq \xic_u$, and the latter variable has finite exponential moments since $P(\xic_u\geq x)=\sum_{2j\geq x}2j\mu^u(\{2j\})$, where $\mu^u$ has a tail decaying exponentially fast by \cref{tree-critical}.
	
	Now take $u=u_C=9/5$ and let $\delta\in (0,2)$. Also let $\epsilon>0$ to be chosen later depending on $\delta$. Using the notation $B_k$ for $B_k^\mathrm{map}$ or $B_k^\mathrm{quad}$, set
	\[
	p_\epsilon(k)=P\({\diam(B_k)\geq k^{1/4+\epsilon}}\).
	\]
	By \cref{deviation-diam-blocks}, we have that $p_\epsilon(k)$ decays stretched exponentially as $k\rightarrow\infty$. Therefore we get a constant $C>0$ such that $k^2 p_\epsilon(k)\leq C$ for all $k$.
	Recall that we have $\diam (\widehat\B_u)\leq \xic_u$. Distinguishing upon whether $\diam (\widehat\B_u)\leq (\xic_u)^{1/4+\epsilon}$ or $\diam (\widehat\B_u)>(\xic_u)^{1/4+\epsilon}$ and taking a conditional expectation with respect to $\xic_u$, we get
	\begin{align*}
	E\left[\bigl( \diam(\widehat \B_u)\bigr)^{2-\delta}\right]
	&\leq E\left[{\left({(\xic_u)}^{1/4+\epsilon}\right)^{2-\delta}\indic{\diam (\widehat\B_u)\leq (\xic_u)^{1/4+\epsilon}}}\right]
	+E\left[{(\xic_u)^{2-\delta} \:p_\epsilon\Bigl(\xic_u\Bigr)}\right]\\
	&\leq E\left[{(\xic_u)^{(1/4+\epsilon)(2-\delta)}}\right] +C\\
	&=\sum_{2j\geq 0} (2j)^{(1/4+\epsilon)(2-\delta)}\cdot 2j\mu^{u_C}(\{2j\})+C.
	\end{align*}
	If $\epsilon$ is small enough so that $(1/4+\epsilon)(2-\delta)<1/2$, then the last sum is finite since by \cref{tree-critical} we have $\mu^{u_C}(\{2j\})=O(j^{-5/2})$. Therefore $E\left[{\bigl(\diam(\widehat \B_u)\bigr)^{2-\delta}}\right]<\infty$.
\end{proof}

%\begin{rem}
	Let us make a brief commentary, and justify that when $u=u_C$, \cref{lemma:tail_D-hat} is optimal, in the sense that $D(\widehat\B^{\mathrm{quad}}_{u_C},U)$ does not have moments of order $\beta$ for $\beta\geq 2$. Firstly, one easily checks that functionals on pointed measured metric spaces of the form
	\[
	(X,x_0,d_X,\nu_X)\mapsto\int_X \nu_X(\mathrm dx)\,\bigl(d_X(x_0,x)\bigr)^\beta
	\]
	are continuous with respect to the Gromov-Hausdorff-Prokhorov topology. Addario-Berry and Albenque~\cite{AddarioBerryAlbenque2021} prove the $\mathrm{GHP}$ convergence of size $k$ uniform simple quadrangulations, rescaled by $\mathrm{cst}\cdot k^{-1/4}$, to the measured Brownian sphere $(\mathcal S,D^*,\lambda)$. This holds when putting either the uniform measure on vertices of $B_k^{\mathrm{quad}}$ or the size-biased one by \cite{quadrang-coeur-cv}. In particular, by the abovementioned continuity, we have the convergence in distribution
	\[
	\,E\left[\left(\mathrm{cst}\cdot k^{-1/4}D(B_k^{\mathrm{quad}},U_k)\right)^\beta\Bigm|B_k^{\mathrm{quad}}\right]
	\xrightarrow[k\rightarrow\infty]{(d)}
	\int_{\mathcal S}\lambda(\mathrm dx)\bigl(D^*(x_0,x)\bigr)^\beta,
	\]
	where $U_k$ is uniform on $\{1,\dotsc,2k\}$ and $x_0$ is the distinguished point on the Brownian sphere. Since the variable $\int_{\mathcal S}\lambda(\mathrm dx)\bigl(D^*(x_0,x)\bigr)^\beta$ is almost surely positive, the left-hand-side forms a tight sequence of $(0,\infty)$-valued random variables. Therefore it is bounded away from $0$ with \textit{uniform} positive probability. This implies a lower bound $E\left[D(B_k^{\mathrm{quad}},U_k)^\beta\right]\geq c (k^{1/4})^\beta$, for some $c=c(\beta)>0$. In particular,
	\begin{align*}
	E\left[D(\widehat\B^{\mathrm{quad}}_{u_C},U)^\beta\right]
	=\sum_{2j\geq 0}E\left[D(B_j^{\mathrm{quad}},U_j)^\beta\right]\cdot\widehat\mu^{u_C}(\{2j\})
	&\geq \sum_{2j\geq 0} cj^{\beta/4}\cdot 2j\mu^{u_C}(\{2j\})\\
	&=\sum_{2j\geq 0}\Theta(j^{\beta/4+1-5/2}).
	\end{align*}
	The latter sum is infinite when $\beta\geq 2$, which proves that $D(\widehat\B^{\mathrm{quad}}_{u_C},U)$ does not have moments of order $\beta$ for $\beta\geq 2$.
	The same argument would hold for $D(\widehat\B^{\mathrm{map}}_{u_C},U)$, but we lack at the moment the GHP convergence of size-$k$ uniform 2-connected maps.
%\end{rem}

\subsubsection{Moderate deviations estimate}

%The arguments exposed in the proof of \cite[Theorem 6.60]{Stufler-survey-2020}, which could be used to deal with the case $u>u_C$, make crucial use of the existence of exponential moments for the offspring distribution of the tree of blocks. We propose in this paper a proof that only requires the existence of a polynomial moment of order $\beta>1$.

When increments of a random walk possess only a polynomial moment of order $\beta>1$, as is the case of $D(\widehat\B^{\mathrm{map}}_{u},U)$ and $D(\widehat\B^{\mathrm{quad}}_{u},U)$ when $u=u_C$, moderate and large deviation events can possibly have probabilities which decay slowly, that is polynomially with $n$. In the case of heavy-tailed increments, this indeed happens since those moderate and large deviation events can be realised by taking one large increment. This \textit{one-big-jump} behaviour is actually precisely how these large deviations events are realised. This phenomenon, which we have already encountered in \cref{sec:phase-diagram} for $u<u_C$, is known as condensation. For a more precise statement, see \cite{survey-trees,ArmendarizLoulakis2009,2Louigi}.

One could hope that if we prevent the variables from condensating, we could still get stretched-exponentially small probabilities for large deviation events. We make this precise in the following proposition, by stating that this is the case when we suitably truncate the increments. We were not able to find an instance of such an estimate in the literature, although it has certainly been encountered in some form. We thus include a short proof, which as usual relies on a Chernoff bound.

\begin{proposition}\label{prop:moderate-dev}
	Let $X$ be a real random variable with i.i.d.~copies $(X_i)_{i\geq 1}$. Assume that there exists $\beta\in(1,2]$ such that $E\left[{|X|^\beta}\right]<\infty$ and that we have $E\left[{X}\right]= 0$.
	
	Then, for all $\delta>0$, $\gamma\in(0,1/\beta+\delta)$, and $\nu\in\bigl(0,\delta\wedge(1/\beta+\delta-\gamma)\bigr)$, there exists a constant $C>0$ such that for all $n\geq 1$,
	\[
	P\({\max_{1\leq k\leq n}\sum_{i=1}^{k}X_i\indic{X_i\leq n^\gamma}>n^{1/\beta+\delta}}\)\leq C\exp(-n^\nu).
	\]
\end{proposition}

\begin{rem}
	A straightforward adaptation of the proof shows that the conclusion still holds if the only assumptions on the variables $(X_i)_i$ are $E\bigl[X_i\bigm|\,X_1,\dotsc,X_{i-1}\bigr]\leq 0$ and $\sup_{i\geq 1}E\bigl[|X_i|^\beta\,\bigm|\,X_1,\dotsc,X_{i-1}\bigr]<\infty$.
\end{rem}

\begin{proof}[Proof of \cref{prop:moderate-dev}]
	Fix an arbitrary $\theta$ such that $\max(\gamma,1/\beta)< \theta<1/\beta+\delta$. By Chernoff's bound, we get for all $1\leq k\leq n$,
	\begin{align*}
	P\({\sum_{i=1}^{k}X_i\indic{X_i\leq n^\gamma}>n^{1/\beta+\delta}}\)
%	&\leq \exp(-\lambda n^{1/\beta+\delta})\Bigl(E\left[{\exp\bigl(\lambda X\indic{X\leq n^\gamma}\bigr)}\right]\Bigr)^k\\
	&\leq \exp(-n^{1/\beta+\delta-\theta})\Bigl(E\left[{\exp\bigl(n^{-\theta} X\indic{X\leq n^\gamma}\bigr)}\right]\Bigr)^k\\
	&\leq
	\exp(-n^{1/\beta+\delta-\theta}) \Bigl(1\vee E\left[{\exp\bigl(n^{-\theta} X\indic{X\leq n^\gamma}\bigr)}\right]\Bigr)^n.
	\end{align*}
	Therefore we obtain by a union bound the estimate
	\[
	P\({\max_{1\leq k\leq n}\sum_{i=1}^{k}X_i\indic{X_i\leq n^\gamma}>n^{1/\beta+\delta}}\)\leq
	n\cdot \exp(-n^{1/\beta+\delta-\theta}) \Bigl(1\vee E\left[{\exp\bigl(n^{-\theta} X\indic{X\leq n^\gamma}\bigr)}\right]\Bigr)^n.
	\]
	Since $\theta$ is arbitrary in the interval $\bigl(\max(\gamma,1/\beta)\:,\:1/\beta+\delta\bigr)$, the exponent $\nu:=1/\beta+\delta-\theta$ is arbitrary in the interval $\bigl(0,\delta\wedge(1/\beta+\delta-\gamma)\bigr)$. As a consequence, to prove the proposition it is sufficient to show that
%	\[
%	\Bigl(E\left[{\exp\bigl(n^{-\theta} X\indic{X\leq n^\gamma}\bigr)}\right]\Bigr)^n=O(1).
%	\]
%	To this end, it is sufficient to show that
	\begin{equation}\label{proof:mod-deviations-heavy:ineq-1}
	E\left[{\exp\Bigl(n^{-\theta} X\indic{X\leq n^\gamma}\Bigr)}\right]\leq 1+O(n^{-1}).
	\end{equation}
	
	Notice that since $\beta\in(1,2]$, for all $M>0$ the following inequality holds for $t$ near $0$ or~$-\infty$:
	\[
	\exp(t)\leq 1+t+M|t|^\beta.
	\]
	Therefore, if one takes $M$ large enough it holds for all $t\in(-\infty,1]$. Fix such a constant $M$.
	
	Given $\lambda,s\geq0$, distinguishing upon whether $\lambda x\in(-\infty,1]$ or not and using that $x\indic{x\leq s}\leq x$ and $x\indic{\lambda x\leq 1}\leq x$ (even when $x<0$), we get for all $x\in\R$,
	\begin{align}
	\label{proof-boundingexp}
	\exp\Bigl(\lambda x\indic{x\leq s}\Bigr)
	&\leq 
		\Bigl(1+\lambda x\indic{x\leq s}+M\lambda^\beta|x|^\beta\indic{x\leq s}\Bigr)\cdot\indic{\lambda x\leq 1}+\exp({\lambda x \indic{x\leq s}})\cdot\indic{\lambda x>1}\nonumber\\
	&\leq 
		1+\lambda x+M\lambda^\beta|x|^\beta +\mathrm \exp(\lambda s)\cdot\indic{\lambda x>1}.
	\end{align}
	Applying this inequality with $x=X$, $\lambda=n^{-\theta}$, $s=n^\gamma$ and taking expectations we obtain
	\[
	E\left[{\exp\Bigl(n^{-\theta} X\indic{X\leq n^\gamma}\Bigr)}\right]
		\leq 1+n^{-\theta}E\left[ X\right]+ Mn^{-\beta\theta}{\textstyle E\left[{|X|^\beta}\right]}+\exp\bigl(n^{\gamma-\theta}\bigr)\cdot\textstyle P\({X\geq n^\theta}\).
	\]
	Recall that $E[ X]= 0$ by hypothesis, that $\gamma-\theta<0$ by choice of $\theta$, and use Markov's inequality. This yields
	\[
	E\left[{\exp\Bigl(n^{-\theta} X\indic{X\leq n^\gamma}\Bigr)}\right]
		\leq 1+0+ Mn^{-\beta\theta}{\textstyle E\left[{|X|^\beta}\right]}+\mathrm \exp({1})\cdot n^{-\beta\theta}{\textstyle E\left[{|X|^\beta}\right]}.
	\]
	Since by hypothesis $E\left[{|X|^\beta}\right]<\infty$, we have
	\[
	E\left[{\exp\Bigl(n^{-\theta} X\indic{X\leq n^\gamma}\Bigr)}\right]
		\leq 1+ O(n^{-\beta\theta})\leq 1+O(n^{-1})
	\]
	where the last inequality comes from the choice of $\theta$, which is greater than $1/\beta$. Therefore \eqref{proof:mod-deviations-heavy:ineq-1} is satisfied and the proposition is proved.
\end{proof}

\subsubsection{A lemma to compare \texorpdfstring{$\m$}{m}, \texorpdfstring{$\q$}{q} and \texorpdfstring{$\t$}{t}}

Let us state a lemma that elaborates on the additivity of distances on consecutive blocks, so that we can bound the $\mathrm{GHP}$-distance between a map (resp.~a quadrangulation) and its block-tree scaled by some constant.
Let $\kappa_1$ and $\kappa_2$ be positive constants.
Let $\m$ be a map, $\q$ its associated quadrangulation by Tutte's bijection, and $\t$ their block-tree.

For $x$ a vertex of either $\m$ or $\q$, denote as in \cref{additivite-dist-carte,additivite-dist-quad} by $v_\star$ the vertex $v$ of $\t$ closest to the root of $\t$ such that $x$ is a vertex of $\b_v^\m$ (resp.~$\b_v^\q$). Set similarly $h_\star=h_\t(v_\star)$ the height of $v_\star$ in $\t$, and $(v_i)_{0\leq i\leq h_\star}$ the ancestor line of $v_\star$ in $\t$, with $v_0$ the root of $\t$ and $v_{h_\star}=v_\star$. Also denote by $x_i$ the root vertex of $\b_{v_i}^\m$ (resp.~$\b_{v_i}^\q$). Finally, let $(l_i)_{0\leq i< h_\star}$ be the respective breadth-first index of the corner in $\b^\m_{v_{i}}$ (resp. the edge in $\b^\q_{v_{i}}$) to which the root corner of $\b^\m_{v_{i+1}}$ (resp. the root edge of $\b^\q_{v_{i+1}}$) is attached.
Finally, denote by $\Delta(\m)$ (resp.~$\Delta(\q)$) the largest diameter of a block of $\m$ (resp.~$\q$). We set the quantities
\begin{align*}
R(\m,v_\star,\kappa_1)
&=\max_{0\leq i<h_\star} \left|\sum_{j=i}^{h_\star-1}\(D(\b^\m_{v_j},l_j)-\kappa_1\)\right|,\\
\text{and}\quad R(\q,v_\star,\kappa_2)
&=
\max_{0\leq i<h_\star} \left|\sum_{j=i}^{h_\star-1}\(D_{\q}(\b^\q_{v_j},l_j)-\kappa_2\)\right|.
\end{align*}
Notice that the preceding quantities depend on $x$ only through $v_\star$ and therefore make sense as functions of only $(\m,v_\star,\kappa_1)$ and $(\q,v_\star,\kappa_2)$ respectively.

\begin{lemma}\label{map-vs-tree}
	Let ${\mathbf f}^\m$ and ${\mathbf f}^\q$ be the functions on $V(\t)$ defined by ${\mathbf f}^\m(v)=|V(\b_v^\m)|-1$ and ${\mathbf f}^\q(v)=|V(\b_v^\q)|-2$ respectively. With the above notation, we have for all $\epsilon>0$,
	\[
	\mathrm d_\mathrm{GHP}\(\epsilon\cdot\underline\m\,,\,{\epsilon}\kappa_1\cdot\bigl(\t,d_\t,\tfrac{{\mathbf f}^\m}{|{\mathbf f}^\m|}\bigr)\)
	\leq \kappa_1\epsilon+ \frac{3\epsilon}{2}\Delta(\m)+\epsilon\max_{v_\star\in V(\t)} R(\m,v_\star,\kappa_1)+\frac2{|\mathbf f^\m|},
	\]
	and
	\[
	\mathrm d_\mathrm{GHP}\(\epsilon\cdot\underline\q\,,\,\epsilon\kappa_2\cdot\bigl(\t,d_\t,\tfrac{{\mathbf f}^\q}{|{\mathbf f}^\q|}\bigr)\)
	\leq (\kappa_2+3)\epsilon +\frac{3\epsilon}{2}\Delta(\q)+\epsilon\max_{v_\star\in V(\t)} R(\q,v_\star,\kappa_2)+\frac4{|\mathbf f^\q|}.
	\]
\end{lemma}

\begin{proof}
Let us first treat the inequality involving $\q$, which is a bit more involved. Consider the correspondence $C$ between $V(\q)$ and ${V}(\t)$ defined as follows. A vertex $x$ of $\q$ is set in correspondence with a vertex $v$ of $\t$ {if and only if} $v$ is the vertex $v_\star$ defined as above from $x$; and let $\rho(x) := v_\star$. Put differently, a vertex $v$ of $\t$ is put in correspondence precisely with the ${\mathbf f}^\q(v)=|V(\b_v^\q)|-2$ vertices of the block $\b_v^\q$ which are not incident to the root edge of $\b_v^\q$ (except when $v$ is the root vertex in which case $v$ is in correspondence with all the vertices of $\b_v^\q$).

Let $\gamma$ be the uniform measure on the previously defined set $C=\{(x,\rho(x))\colon x\in V(\q)\}\subset V(\q)\times {V}(\t)$. Let the function $\pi\colon C\rightarrow V(\t)$ be the restriction of the projection $\pi_\t : (x,y) \in V(\q)\times {V}(\t)\mapsto y \in {V}(\t)$. The preimages of $\pi$ have cardinal $|\pi^{-1}(v)|=|\rho^{-1}(v)|={\mathbf f}^\q(v)+2\delta_{\mathrm{root}}(v)$, where $\delta_{\mathrm{root}}(v)$ is an indicator that $x$ is incident to the root-edge of $\q$.
Tautologically, the measure $\gamma$ defines a coupling between its images by the projections $\pi_\q : (x,y) \in V(\q)\times {V}(\t)\mapsto x \in {V}(\q)$ and $\pi_\t$. That is to say, $\gamma$ is a coupling of the measures $\nu_\q$ and $\frac{{\mathbf f}^\q+2\delta_{\mathrm{root}}}{|{\mathbf f}^\q|+2}$. It is also supported by $C$, \textit{i.e.}~$\gamma\bigl((V(\q)\times V(\t))\setminus C\bigr)=0$.
By the triangle inequality and the preceding observations, we have
\begin{align*}
&d_\mathrm{GHP}\(\epsilon\cdot\underline\q\,,\,\epsilon\kappa_2\cdot\bigl(\t,d_\t,\tfrac{{\mathbf f}^\q}{|{\mathbf f}^\q|}\bigr)\)\\
&\leq{} 
	d_\mathrm{GHP}\(\epsilon\cdot\underline\q\,,\,\epsilon\kappa_2\cdot\bigl(\t,d_\t,\tfrac{{\mathbf f}^\q+2\delta_{\mathrm{root}}}{|{\mathbf f}^\q|+2}\bigr)\)
	+d_\mathrm{GHP}\(\epsilon\kappa_2\cdot\bigl(\t,d_\t,\tfrac{{\mathbf f}^\q+2\delta_{\mathrm{root}}}{|{\mathbf f}^\q|+2}\bigr)\,,\,\epsilon\kappa_2\cdot\bigl(\t,d_\t,\tfrac{{\mathbf f}^\q}{|{\mathbf f}^\q|}\bigr)\)\\
&\leq{}\frac{\epsilon}{2}\,\mathrm{dis}(C; d_\q,\kappa_2 d_\t)+d_{\mathrm P}^{(V(\t),\epsilon \kappa_2 d_\t)}(\tfrac{{\mathbf f}^\q+2\delta_{\mathrm{root}}}{|{\mathbf f}^\q|+2},\tfrac{{\mathbf f}^\q}{|{\mathbf f}^\q|}).
\end{align*}
The last inequality uses \cref{eq-Strassen} to bound the second GHP distance by a Prokhorov distance.
Now, the Prokhorov distance between two measures is bounded by their total variation distance, and for measures $\mu$ and $\nu$ we have elementarily $d_{\mathrm{TV}}(\frac{\mu+\nu}{|\mu|+|\nu|},\frac{\mu}{|\mu|})\leq \frac{2|\nu|}{|\mu|}$. Therefore, we have
\[
d_{\mathrm P}^{(V(\t),\epsilon\kappa_2 d_\t)}\left(\frac{{\mathbf f}^\q+2\delta_{\mathrm{root}}}{|{\mathbf f}^\q|+2},\frac{{\mathbf f}^\q}{|{\mathbf f}^\q|}\right)
\leq
d_{\mathrm{TV}}\left(\frac{{\mathbf f}^\q+2\delta_{\mathrm{root}}}{|{\mathbf f}^\q|+2},\frac{{\mathbf f}^\q}{|{\mathbf f}^\q|}\right)
\leq \frac{4}{|{\mathbf f}^\q|}.
\]

It remains to bound the distortion $\mathrm{dis}(C;d_\q,\kappa_2 d_\t)$.
This amounts to bounding $|d_\q(x,\widetilde x)-\kappa_2 d_\t(\rho(x),\rho(\widetilde x))|$ uniformly for all pairs of vertices $(x,\widetilde x)$ in $V(\q)\times V(\q)$. Let $x$ and $\widetilde x$ be vertices of $\q$. As before, we define the vertices $v_\star=\rho(x)$ and $\widetilde v_\star=\rho(\widetilde x)$ in $V(\t)$, their respective heights $h_\star$ and $\widetilde h_\star$, their ancestor lines $(v_i)_{0\leq i\leq h_\star}$ and $(\widetilde v_i)_{0\leq i\leq h_\star}$, the labels $(l_i)_{0\leq i<h_\star}$ and $(\widetilde l_i)_{0\leq i<h_\star}$, and the vertices $(x_i)_{0\leq i\leq h_\star}$ and $(\widetilde x_i)_{0\leq i\leq h_\star}$.

Let $i$ be such that $v_i=\widetilde v_i$ is the last common ancestor of $v_\star$ and $\widetilde v_\star$ in $\t$. First notice that there exists $\delta_0\in\{0,\pm 1,\pm 2\}$ such that
\begin{equation}\label{decomp-dist-selon-dernier-ancetre}
d_\q(x,\widetilde x)=\delta_0+d_\q(x,x_{i+1})+d_{\b_{v_i}^\q}(x_{i+1},\widetilde x_{i+1})+d_\q(\widetilde x_{i+1},\widetilde x).
\end{equation}
Indeed, similarly as in the proof of \cref{additivite-dist-quad}, a geodesic from $x$ to $\widetilde x$ must visit, once and in that order, 
\begin{itemize}
	\item the vertex $x$,
	\item either $x_{i+1}$, or $x'_{i+1}$ the other endpoint of the root-edge of $\b_{v_{i+1}}^\q$,
	\item either $\widetilde x_{i+1}$, or $\widetilde x'_{i+1}$ the other endpoint of the root-edge of $\b_{\widetilde v_{i+1}}^\q$,
	\item the vertex $\widetilde x$.
\end{itemize}
Since $x_{i+1}$ and $x'_{i+1}$, as well as $\widetilde x_{i+1}$ and $\widetilde x'_{i+1}$, are at distance $1$ respectively, and since a geodesic between points in $\b_{v_i}$ must stay in $\b_{v_i}$, we get that \cref{decomp-dist-selon-dernier-ancetre} holds, for some $\delta_0\in\{0,\pm1,\pm2\}$.
Then, \cref{additivite-dist-quad} allows to decompose the distances $d_\q(x,x_{i+1})$ and $d_\q(\widetilde x_{i+1},\widetilde x)$, with some $\delta,\widetilde \delta$ in $\{0,1\}$. Combining this with \cref{decomp-dist-selon-dernier-ancetre}, this gives
\begin{align*}
	d_\q(x,\widetilde x)
	={}&\delta_0+\left[\delta+d_{\b_{v_\star}^\q}(x,x_{h_\star})+\sum_{i+1\leq j<h_\star-1}D_{\q}(v_j,l_j)\right]
	+d_{b_{v_i}^\q}(x_{i+1},\widetilde x_{i+1})\\
	&+\left[\widetilde\delta+d_{\b_{\widetilde v_\star}^\q}(\widetilde x,\widetilde x_{\widetilde h_\star})+\sum_{i+1\leq j<\widetilde h_\star-1}D_{\q}(\widetilde v_j,\widetilde l_j)\right]\\
	={}&\delta_0+\delta +\widetilde\delta 
	+d_{\b_{v_\star}^\q}(x,x_{h_\star})
	+d_{\b_{v_i}^\q}(x_{i+1},\widetilde x_{i+1})
	+d_{\b_{\widetilde v_\star}^\q}(\widetilde x,\widetilde x_{\widetilde h_\star})\\
	&+\sum_{i+1\leq j<h_\star-1}\left(D_{\q}(v_j,l_j)-\kappa_2\right)
	+\sum_{i+1\leq j<\widetilde h_\star-1}\bigl(D_{\q}(\widetilde v_j,\widetilde l_j)-\kappa_2\bigl)\\
	&+\kappa_2 (h_\star-i-1)+\kappa_2(\widetilde h_\star-i-1).
\end{align*}
The sum of the first six terms has absolute value at most $6+3\Delta(\q)$, the two sums have absolute value at most $R(\q,x,\kappa_2)$ and $R(\q,\widetilde x,\kappa_2)$ respectively, and the two remaining terms sum to $\kappa_2 (h_\star-i-1)+\kappa_2(\widetilde h_\star-i-1)=\kappa_2d_\t(v,\widetilde v)-2\kappa_2$. Therefore by the triangle inequality,
\[
|d_\q(x,\widetilde x)-\kappa_2 d_\t(v,\widetilde v)|
\leq 2\kappa_2+6+3\Delta(\q)+R(\q,x,\kappa_2)+R(\q,\widetilde x,\kappa_2).
\]
Since this holds for every $(x,v)=(x,\rho(x))\in C$ and $(\widetilde x,\widetilde v)=(\widetilde x,\rho(\widetilde x))\in C$, the max of the right-hand side over $x,\widetilde x\in V(\q)$ is actually a bound on the distortion $\mathrm{dis}(C;d_\q,\kappa_2 d_\t)$, which is precisely what we needed to conclude.

For the inequality involving $\m$, the reasoning is quite similar. Take $C$ the correspondence such that $x\in V(\m)$ is in correspondence with $v\in {V}(\t)$ {if and only if} $v$ is the vertex $\rho(x)=v_\star$. Equivalently, a vertex $v$ of $\t$ is put in correspondence with the ${\mathbf f}^\m(v)=|V(\b_v^\m)|-1$ non-root vertices of the block $\b_v^\m$, except when $v$ is the root vertex of $\t$ in which case $v$ is in correspondence with all the vertices of $\b_v^\m$.
Then, similarly, the uniform measure $\gamma$ on $C$ defines a coupling between the measures $\nu_\m$ and $\frac{{\mathbf f}^\m+\delta_{\mathrm{root}}}{|{\mathbf f}^\m|+1}$, where $\delta_{\mathrm{root}}(x)$ is the indicator that $x$ is the root vertex of $\m$. As in the quadrangulation case, we have
\begin{align*}
	&d_\mathrm{GHP}\(\epsilon\cdot\underline\m\,,\,\epsilon\kappa_1\cdot\bigl(\t,d_\t,\tfrac{{\mathbf f}^\m}{|{\mathbf f}^\m|}\bigr)\)
	\leq{}\frac{\epsilon}{2}\,\mathrm{dis}(C; d_\m,\kappa_1 d_\t)+d_{\mathrm P}^{(V(\t),\epsilon \kappa_1 d_\t)}(\tfrac{{\mathbf f}^\m+\delta_{\mathrm{root}}}{|{\mathbf f}^\m|+1},\tfrac{{\mathbf f}^\m}{|{\mathbf f}^\m|}),
\end{align*}
with the similar bound
\[
d_{\mathrm P}^{(V(\t),\epsilon\kappa_1 d_\t)}\left(\frac{{\mathbf f}^\m+\delta_{\mathrm{root}}}{|{\mathbf f}^\m|+1},\frac{{\mathbf f}^\m}{|{\mathbf f}^\m|}\right)
\leq \frac{2}{|{\mathbf f}^\m|}.
\]
The distortion of $C$ is bounded with a very similar argument as above involving \cref{additivite-dist-carte} instead of \cref{additivite-dist-quad}, except that we do not need to introduce $\delta_0,\delta,\widetilde \delta$. We leave the details to the reader. This gives for all $(x,v)\in C$ and $(\widetilde x, \widetilde v)\in C$, the bound
\[
|d_\m(x,\widetilde x)-\kappa_1 d_\t(v,\widetilde v)|
\leq 2\kappa_1+3\Delta(\m)+R(\m,x,\kappa_1)+R(\m,\widetilde x,\kappa_1),
\]
which proves the inequality involving $\m$ in the statement.
\end{proof}

\subsubsection{Proof of Theorem \ref{scaling-limit-super-critical}}

%We can now turn to the main proof of this section.
%
%\begin{proof}[Proof of \cref{scaling-limit-super-critical}]
%	In the case where $u>u_C$, the offspring distribution $\mu^u$ has exponential moments by \cref{tree-critical}, so that one could readily apply \cite[Theorem 6.60]{Stufler-survey-2020} to obtain the scaling limit in the map case, by interpreting $\M$ as an $\mathcal R$-enriched tree, the structure $\mathcal R$ being 2-connected maps. In this proof however we only rely on the presence of polynomial moments. This is necessary to treat the case $u=u_C$.
	Let $u\geq u_C$.
	Let us first prove the claimed scaling limit for the block-tree $\T_{n,u}$. By \cref{th-trees}, $\T_{n,u}$ has law $\mathrm{GW}(\mu^u,2n)$, where the distribution $\mu^u$ has span $2$.
%	\begin{enumerate}
%		\item

	\paragraph{Scaling limit of $\T_{n,u}$ for $u>u_C$.}
			If $u>u_C$, then by the third statement of \cref{tree-critical}, $\mu^u$ is critical and admits a variance $\sigma(u)^2<\infty$. \cref{scaling-lim-trees} thus gives the announced scaling limit for $\T_{n,u}$,
			\[
			(2n)^{-1/2}\cdot \underline{\T}_{n,u}\xrightarrow[n\rightarrow\infty]{\quad GHP,(d)\quad}
			\frac{\sqrt 2}{\sigma(u)}\cdot{\mathcal{T}}^{(2)}.
			\]
			The expression for $\sigma(u)$ in terms of the generating function $B$ which is given in the statement comes from a straightforward computation from the generating function of $\mu^u$, which by \cref{mu} is
			\[
			\sum_{k\geq 0}x^k \mu^u(k)=\frac{uB(x^2y(u))+1-u}{uB(y(u))+1-u}.
			\]
			This expression admits the explicit form in terms of $u$ which is given in the statement and explained in \cref{rem-sigma}.
		
		\paragraph{Scaling limit of $\T_{n,u}$ for $u=u_C$.}
			If $u=u_C$, then by the second statement of \cref{th-trees}, $\mu^{u_C}$ is critical and satisfies $\mu^{u_C}(\{2j\})\sim\frac{1}{4\sqrt{3\pi}}j^{-5/2} $. Therefore we get the equivalent
			\[
			\mu^{u_C}([x,\infty))
			=\sum_{2j\geq x}\mu^{u_C}(\{2j\})
			\sim\int_{x/2}^{\infty}\frac{1}{4\sqrt{3\pi}}s^{-5/2}\,\mathrm d s
			=\frac{1}{3}\sqrt{\frac{2}{3\pi}}x^{-3/2}.
			\]
			Therefore, using \cref{scaling-lim-trees} with $\theta=3/2$, we get
			\[
			(2n)^{1-2/3}\cdot\underline{\T}_{n,u_C} \xrightarrow[n\rightarrow\infty]{\quad GHP,(d)\quad}
			\left[\frac{\tfrac{3}{2}-1}{\tfrac{1}{3}\sqrt{\tfrac{2}{3\pi}}\Gamma(2-\tfrac{3}{2})}\right]^{2/3}\cdot{\mathcal{T}}^{(3/2)}.
			\]
			Using that $\Gamma(1/2)=\sqrt{\pi}$, the constant on the right-hand side simplifies and this translates as announced to
			\[
			\frac{2}{3}(2n)^{-1/3}\cdot\underline\T_{n,u_C}
			\xrightarrow[n\rightarrow\infty]{\quad GHP,(d)\quad}
			{\mathcal{T}}^{(3/2)}.
			\]
%	\end{enumerate}
	
	\paragraph{Restatement of the problem.}
	We let $\alpha=2$ when $u>u_C$, and $\alpha=3/2$ when $u=u_C$. We have identified the $\mathrm{GHP}$-limit of $n^{-(\alpha-1)/\alpha}\cdot\underline\T_{n,u}$. By \cref{scaling-lim-trees-random-measure}, the measured metric spaces 
	\begin{align*}
	n^{-(\alpha-1)/\alpha}\cdot\left(V(\T_{n,u}),d_{\T_{n,u}},\tfrac{\mathbf f^{\M_{n,u}}}{|\mathbf f^{\M_{n,u}}|}\right)\quad\text{ and }\quad
	n^{-(\alpha-1)/\alpha}\cdot\left(V(\T_{n,u}),d_{\T_{n,u}},\tfrac{\mathbf f^{\Q_{n,u}}}{|\mathbf f^{\Q_{n,u}}|}\right),
	\end{align*}
	also converge to the same limit. It remains to compare them in the $\mathrm{GHP}$ sense to the measured metric spaces $n^{-(\alpha-1)/\alpha}\cdot\underline\M_{n,u}$ and $n^{-(\alpha-1)/\alpha}\cdot \underline\Q_{n,u}$ respectively. Let $\kappa_1=\kappa_u^{\mathrm{map}}$ and $\kappa_2=\kappa_u^{\mathrm{quad}}$.
%	 What we  to show that both quantities
%	\begin{equation}\label{eq:distances-à-majorer}
%	\mathrm d_\mathrm{GHP}\(n^{-\frac{\alpha-1}\alpha}\cdot\underline\M_{n,u}\,,\kappa_1\cdot\mathbf X_{n,u}\)
%	\quad\text{ and }\quad
%	\mathrm d_\mathrm{GHP}\bigl(n^{-\frac{\alpha-1}{\alpha}}\cdot\underline\Q_{n,u}\,,\kappa_2\cdot\mathbf Y_{n,u}\bigr)
%	\end{equation}
%	converge to $0$ in probability.
	For the ease of reading, we introduce for $\eta>0$ the following ``bad'' events,
	\begin{align*}
		B^\M_{n,\eta}&=\Bigl\{d_\mathrm{GHP}\(n^{-\frac{\alpha-1}\alpha}\cdot\underline\M\,,\,n^{-\frac{\alpha-1}\alpha}\kappa_1\cdot\left(V(\T),d_{\T},\tfrac{\mathbf f^{\M}}{|\mathbf f^{\M}|}\right)\)\geq 2\eta\Bigr\},\\
		B^\Q_{n,\eta}&=\Bigl\{d_\mathrm{GHP}\(n^{-\frac{\alpha-1}\alpha}\cdot\underline\Q\,,\,n^{-\frac{\alpha-1}\alpha}\kappa_2\cdot\left(V(\T),d_{\T},\tfrac{\mathbf f^{\Q}}{|\mathbf f^{\Q}|}\right)\)\geq 2\eta\Bigr\},
	\end{align*}
	as well as auxiliary events for $\eta,\delta>0$,
	\begin{align*}
		A_{1;n,\eta}^\M&=\left\{\exists v\in V(\T),\,R(\M,v,\kappa_1)\geq \eta n^{\frac{\alpha-1}{\alpha}} \right\},\\
%		&\text{ and }\quad
		A_{1;n,\eta}^\Q&=\left\{\exists v\in V(\T),\,R(\Q,v,\kappa_2)\geq \eta n^{\frac{\alpha-1}{\alpha}} \right\},\\
%		A_{2}&=\{H(\T)\leq \eta^{-1}n^{\frac{\alpha-1}{\alpha}}\},\\
		A_{2;n,\delta}^\M&=\left\{\Delta(\M)\leq n^{(1+\delta)^2(\alpha-1)/2\alpha}\right\},\\
%		\quad\text{ and }\quad
		A_{2;n,\delta}^\Q&=\left\{\Delta(\Q)\leq n^{(1+\delta)^2(\alpha-1)/2\alpha}\right\},\\
		A_{3;n,\eta}^\M&=\left\{\kappa_1n^{-\frac{\alpha-1}{\alpha}}+\tfrac{3n^{-\frac{\alpha-1}{\alpha}}}{2}\Delta(\M)+\frac{2}{|\mathbf f^\M|}\geq \eta\right\},\\
%		&\text{ and }\quad
		A_{3;n,\eta}^\Q&=\left\{(\kappa_2+3)n^{-\frac{\alpha-1}{\alpha}}+\frac{3n^{-\frac{\alpha-1}{\alpha}}}{2}\Delta(\Q)+\frac{4}{|\mathbf f^\Q|}\geq \eta\right\}.
	\end{align*}
	With this notation, what we have to show is
	\[
	\lim_{\eta\rightarrow0}\limsup_{n\rightarrow\infty}\,\mathbb P_{n,u}(B^\M_{n,\eta})=0
	\quad \text{ and }\quad
	\lim_{\eta\rightarrow0}\limsup_{n\rightarrow\infty}\,\mathbb P_{n,u}(B^\Q_{n,\eta})=0.
	\]
	
	\paragraph{Using \cref{map-vs-tree}.} Thanks to the $\mathrm{GHP}$ upper bounds in \cref{map-vs-tree}, we have
	\begin{equation}\label{preuve-surcrit:eq:majoration-mauvaise-proba}
	\mathbb P_{n,u}(B^\M_{n,\eta})\leq \mathbb P_{n,u}(A_{1;n,\eta}^\M)+\mathbb P_{n,u}(A_{3;n,\eta}^\M)
	\quad\text{ and }\quad
	\mathbb P_{n,u}(B^\Q_{n,\eta})\leq \mathbb P_{n,u}(A_{1;n,\eta}^\Q)+\mathbb P_{n,u}(A_{3;n,\eta}^\Q).
	\end{equation}
		
	\paragraph{Bounding the diameters of the blocks.}
	By \cref{diam-blocs}, for $\delta>0$, the maximum diameter of blocks of either $\M_{n,u}$ or $\Q_{n,u}$
	is bounded with probability $1-o(1)$ by $\max(n^{1/6},W(\T_{n,u})^{(1+\delta)/4})$, where $W(\t)$ denotes the largest degree of $\t$.
	By \cref{scaling-lim-trees}, $W(\T_{n,u})$ is $o\(n^{(1+\delta)/\alpha}\)$ in probability. Since $(1+\delta)^2/4\alpha\geq 1/6$, what precedes gives that for all $\delta>0$,
	\[
	\max\bigl(\Delta(\M_{n,u}),\Delta({\Q_{n,u}})\bigr)=o\(n^{(1+\delta)^2/4\alpha}\)\qquad\text{in probability.}
	\]
	Notice that for $\delta$ small enough, $(1+\delta)^2/4\alpha<(1+\delta)^2(\alpha-1)/2\alpha$ since $\alpha\geq 3/2$. This implies that for all $\delta>0$ sufficiently small, we have
	\[
	\limsup_{n\rightarrow\infty}\,
	\mathbb P_{n,u}\Bigl((A_{2;n,\delta}^\M)^c\Bigr)=0
	\quad\text{ and }\quad
	\limsup_{n\rightarrow\infty}\,
	\mathbb P_{n,u}\Bigl((A_{2;n,\delta}^\Q)^c\Bigr)=0.
	\]
	By \cref{lemma-cvg-trees-LLN}, the quantities $|\mathbf f^\M|$ and $|\mathbf f^\Q|$ are $\Theta(n)$ in probability under $\mathbb P_{n,u}$. Therefore, the preceding bound on diameters also implies the following
	\[
	\lim_{\eta\rightarrow0}\limsup_{n\rightarrow\infty}\,\mathbb P_{n,u}(A_{3;n,\eta}^\M)=0
	\quad \text{ and }\quad
	\lim_{\eta\rightarrow0}\limsup_{n\rightarrow\infty}\,\mathbb P_{n,u}(A_{3;n,\eta}^\Q)=0.
	\]
	Thanks to \cref{preuve-surcrit:eq:majoration-mauvaise-proba}, it suffices to show that for sufficiently small $\delta>0$,
	\[
	\lim_{\eta\rightarrow0}\limsup_{n\rightarrow\infty}\,\mathbb P_{n,u}(A_{1;n,\eta}^\M\cap A_{2;n,\delta}^\M)=0
	\quad \text{ and }\quad
	\lim_{\eta\rightarrow0}\limsup_{n\rightarrow\infty}\,\mathbb P_{n,u}(A_{1;n,\eta}^\Q\cap A_{2;n,\delta}^\Q)=0.
	\]
	
	\paragraph{Bounding the height of $\T_{n,u}$.}
	We have identified above the scaling limit of $\T_{n,u}$ and the appropriate normalization of distances. In particular, $n^{(\alpha-1)/\alpha}\cdot\T_{n,u}$ is tight in the $\mathrm{GHP}$-topology. An immediate consequence is that $n^{-(\alpha-1)/\alpha} H(\T_{n,u})$ is tight, where $H(\T_{n,u})$ is the height of $\T_{n,u}$. In particular, our problem reduces once more to showing that for sufficiently small $\delta>0$,
	\begin{multline}\label{preuve-surcrit:eq:apres-coupe-hauteur}
		\lim_{\eta\rightarrow0}\limsup_{n\rightarrow\infty}\,\mathbb P_{n,u}\Bigl(A_{1;n,\eta}^\M\cap A_{2;n,\delta}^\M\cap \left\{H(\T)\leq \eta^{-1}n^{\frac{(\alpha-1)}{\alpha}}\right\}\Bigr)=0\\
		\text{ and }\quad
		\lim_{\eta\rightarrow0}\limsup_{n\rightarrow\infty}\,\mathbb P_{n,u}\Bigl(A_{1;n,\eta}^\Q\cap A_{2;n,\delta}^\Q\cap \left\{H(\T)\leq \eta^{-1}n^{\frac{(\alpha-1)}{\alpha}}\right\}\Bigr)=0.
	\end{multline}
	
	\paragraph{Using the spine decomposition.}
	Fix $\delta>0$, as small as necessary.
	Let us only treat the term involving $\Q$ in \cref{preuve-surcrit:eq:apres-coupe-hauteur}, as the expression for $R(\q,x,\kappa_2)$ we used to define the event $A_{1;n,\eta}^\Q$ carries more dependence than that of $R(\m,x,\kappa_1)$. Indeed the summands $D_\q(b_{v_i}^\q,l_i)$ involve in their definition a global metric property of $\q$. The case of the term involving $\M$ is similar and easier to deal with.
	
	Recall that by definition, the law $\mathbb P_{n,u}$ is the law $\mathbb P_{u}$, conditioned on the event $\{|\T|=n\}$. Since $\mathbb P_u(|\T|=n)$ decays polynomially by \cref{scaling-lim-trees}, we may get rid of the conditioning if the unconditional versions of the probabilities we wish to bound decay sufficiently fast. Namely, it suffices to prove that for all $\eta>0$, the following (unconditional) probability is stretched-exponential in $n$
	\begin{equation}\label{preuve-surcritique-proba-a-maj}
%	\mathbb P_u\left(\{\exists v_\star\in\T,\,R(\Q,e_\star,\kappa_2)\geq \eta n^{\frac{\alpha-1}{\alpha}} \}\cap \{H(\T)\leq \eta^{-1}n^{\frac{\alpha-1}{\alpha}}\}\cap\{\Delta(\Q)\leq n^{(1+\delta)^2(\alpha-1)/2\alpha}\} \right).
	\mathbb P_{u}\Bigl(A_{1;n,\eta}^\Q\cap A_{2;n,\delta}^\Q\cap \{H(\T)\leq \eta^{-1}n^{\frac{(\alpha-1)}{\alpha}}\}\Bigr).
	\end{equation}
	By a union bound and then by \cref{loi-biaisee}, one can bound this by
	\begin{align*}
	&\mathbb E_u\left[\sum_{v\in V(\T)} \indic{R(\Q,v,\kappa_2)\geq \eta n^{\frac{\alpha-1}{\alpha}}}\indic{ H(\T)\leq \eta^{-1}n^{\frac{\alpha-1}{\alpha}}}
%	\indic{\Delta(\Q)\leq n^{(1+\delta)^2(\alpha-1)/2\alpha} }
	\indic{A_{2;n,\delta}^\Q}
	\right]\\
	&= 
	\sum_{h\geq 1}\widehat{\mathbb P}_{u,h}\left(\{R(\Q,V_\star,\kappa_2)\geq \eta n^{\frac{\alpha-1}{\alpha}} \}\cap \{H(\T)\leq \eta^{-1}n^{\frac{\alpha-1}{\alpha}}\}\cap A_{2;n,\delta}^\Q \right)\\
	&=
	\sum_{h= 1}^{\floor{\eta^{-1}n^{\frac{\alpha-1}{\alpha}}}}
	\widehat{\mathbb P}_{u,h}\left(\{R(\Q,V_\star,\kappa_2)\geq \eta n^{\frac{\alpha-1}{\alpha}} \}\cap A_{2;n,\delta}^\Q \right)\\
	&=
	\sum_{h= 1}^{\floor{\eta^{-1}n^{\frac{\alpha-1}{\alpha}}}}
	\widehat{\mathbb P}_{u,h}\left(\left\{\max_{0\leq i<h } \left|\sum_{j=i}^{h -1}\(D_{\Q}(\b^\Q_{v_j},L_j)-\kappa_2\)\right|\geq \eta n^{\frac{\alpha-1}{\alpha}} \right\}\cap A_{2;n,\delta}^\Q \right)\\
	&\leq
	\sum_{h= 1}^{\floor{\eta^{-1}n^{\frac{\alpha-1}{\alpha}}}}
	\Bigg[
	\widehat{\mathbb P}_{u,h}\Biggl(\max_{0\leq i<h } \sum_{j=i}^{h-1}\psi_{n,\delta}\(D_{\Q}(\b^\Q_{V_j},L_j)-\kappa_2\)\geq \eta n^{\frac{\alpha-1}{\alpha}}\Biggr)\\
	&\hspace{15em}+
	\widehat{\mathbb P}_{u,h}\Biggl(\max_{0\leq i<h } \sum_{j=i}^{h-1}\psi_{n,\delta}\(\kappa_2-D_{\Q}(\b^\Q_{V_j},L_j)\)\geq \eta n^{\frac{\alpha-1}{\alpha}}\Biggr)
	\Bigg],
	\end{align*}
	where 
	\[
	\psi_{n,\delta}(x)=x\indic{x\leq\max(\kappa_2, n^{(1+\delta)^2(\alpha-1)/2\alpha})}.
	\]
	The last inequality may require some explanations. First we apply a union bound with respect to the sign of the expression under the absolute value. Then we use the control that $A_{2;n,\delta}^\Q$ offers on $\Delta(\Q)$ the maximum diameter of blocks of $\Q$, and the positivity of the distances $D_{\Q}(\b^\Q_{v_j})$, to insert an indicator function. Hence the appearance of $\psi_{n,\delta}$.
	
	\paragraph{Reducing to a large deviations event with truncated variables.}
	We let $\bigl(\xic_{u,j}\:,\:
	\widehat{\B}^{\mathrm{quad}}_{u,j}\:,\:
	U_j\bigr)_{j\geq 0}$ be an i.i.d.~sequence of copies of the triple $\bigl(\xic_u\:,\:
	\widehat{\B}_u^{\mathrm{quad}}\:,\:
	U\bigr)$. We also let $X_j=D\bigl(\widehat{\B}_{u,j}^{\mathrm{quad}},U_j\bigr)-\kappa_2$. Then by \cref{prop-indep-dist-epine-quad}, the arguments of the function $\psi_{n,\delta}$ that appear in the last upper bound we obtained, are actually i.i.d. and have joint law under $\widehat{\mathbb P}_{u,h}$ the law of $(X_j)_{0\leq j<h}$. Therefore this last upper bound is equal to
	\begin{align*}
	&\sum_{h= 1}^{\floor{\eta^{-1}n^{\frac{\alpha-1}{\alpha}}}}
	\left[P\left(\max_{0\leq i<h } \sum_{j=i}^{h -1}\psi_{n,\delta}\(X_j\)\geq \eta n^{\frac{\alpha-1}{\alpha}}\right)
	+
%	\sum_{h= 1}^{\eta^{-1}n^{\frac{\alpha-1}{\alpha}}}
	P\left(\max_{0\leq i<h } \sum_{j=i}^{h -1}\psi_{n,\delta}\(-X_j\)\geq \eta n^{\frac{\alpha-1}{\alpha}}\right)
	\right].
	\end{align*}
	Since the sequence $(X_j)_{0\leq j<h}$ is i.i.d., we re-order the terms of the two sums which appear inside the probabilities in the last display, so that they run on indices $j=1,\dotsc,i$. Hence, if we set $h_n=n^{(\alpha-1)/\alpha}$, then we can bound the last display by
	\begin{align*}
	&\eta^{-1}h_n
	\left[
	P\left(\max_{0\leq i<\eta^{-1}h_n } \sum_{j=1}^{i}\psi_{n,\delta}\(X_j\)\geq \eta h_n\right)
	+
	P\left(\max_{0\leq i<\eta^{-1}h_n } \sum_{j=1}^{i}\psi_{n,\delta}\(-X_j\)\geq \eta h_n\right)
	\right].
%	&=h(n,\eta)
%	P\left(\max_{0\leq i<h(n,\eta) } \sum_{j=1}^{i}\psi_{n,\delta}\(X_j\)\geq \eta^2 h(n,\eta)\right)
%	+
%	h(n,\eta)
%	P\left(\max_{0\leq i<h(n,\eta) } \sum_{j=1}^{i}\psi_{n,\delta}\(-X_j\)\geq \eta^2 h(n,\eta)\right)
	\end{align*}
	
	\paragraph{Using the moderate deviations estimate.}
	Let $\gamma=\gamma(\delta)=(1+\delta)^2/2$. Then $(h_n)^\gamma\geq \kappa_2$ for $n$ large, so
	\[
	\psi_{n,\delta}(x)=x\indic{x\leq \max(\kappa_2,(h_n)^\gamma)}=x\indic{x\leq (h_n)^\gamma}.
	\]
	Let us check that the choice of $\kappa_2 = \kappa_u^{\mathrm{quad}}$, the latter quantity being defined in \cref{defin-kappa-u}, makes the variables $(X_j)$ centered. Conditionally on the event $\{|\widehat{\B}_{u}^{\mathrm{quad}}| = k\}$, the variable $\widehat{\B}_{u}^{\mathrm{quad}}$ is a uniform simple quadrangulation with $k$ edges, and $U$ is uniform in $\{1,\dots,2k\}$. Therefore it holds that
	\begin{align*}
	\E{D\bigl(\widehat{\B}_{u,j}^{\mathrm{quad}},U_j\bigr)} &= \sum_{k\geq1} P\left({|\widehat{\B}_{u}^{\mathrm{quad}}| = k}\right) \E{D\bigl(\widehat{\B}_{u}^{\mathrm{quad}},U\bigr) \mid |\widehat{\B}_{u}^{\mathrm{quad}}| = k}\\
	&= \sum_{k\geq1} \hat{\mu}^u(2k) \mathcal{D}_k^{\mathrm{quad}} = \kappa_2,
	\end{align*}
	where we used successively the definition of $(\b,l)\mapsto D(\b,l)$ in \cref{prop-indep-dist-epine-quad}, the definition of $\mathcal{D}_k^{\mathrm{quad}}$ after \cref{defin-kappa-u}, and the definition of $\kappa_2=\kappa_u^{\mathrm{quad}}$ in \cref{defin-kappa-u}.
	Therefore, the i.i.d.~variables $(X_j)$ are indeed centered.
	 Now by \cref{lemma:tail_D-hat}, they possess moments of order $\beta$ for all $1\leq \beta<2$.
	Since for $\delta$ sufficiently small we have $\gamma<1$, \cref{prop:moderate-dev} yields that
	\[
	P\left(\max_{0\leq i<\eta^{-1}h_n } \sum_{j=1}^{i}X_j \indic{X_j\leq (h_n)^\gamma}\geq \eta h_n\right)
	\]
	is stretched-exponential as $n\rightarrow\infty$,
	and the same holds when replacing $(X_j)$ by $(-X_j)$.
	
	This proves that for each $\eta>0$, the probability \cref{preuve-surcritique-proba-a-maj} is indeed stretched-exponential in $n$, and concludes the proof.
\qed

\subsection{Scaling limit of the quadrangulations in the subcritical case}\label{sec:sous-critique}

Let us finally identify the scaling limit of the quadrangulation $\underline{\Q}_{n,u}$ when $u<u_C$.

\subsubsection{Statement of the result}

Denote by $ {\mathcal S}=(\mathcal S,D^*,\lambda)$ the \textit{Brownian sphere}, also known as the \textit{Brownian Map}. One may take \cref{scaling-lim-quad} below as a definition.

\begin{theorem}
\label{scaling-limit-subcritical}
	Assume $u<u_C=9/5$. We have the following convergence in distribution for the \emph{Gromov-Hausdorff-Prokhorov} topology
	\[
	\(\frac{9(3+u)}{8(9-5u)}\)^{1/4}{n^{-1/4}}\cdot\underline\Q_{n,u} \xrightarrow[n\rightarrow\infty]{\quad(d), \mathrm{GHP}\quad} {\mathcal S}.
	\]
\end{theorem}

In the case $u=1$, one recovers the Brownian sphere as the scaling limit of uniform quadrangulations with $n$ faces, which has been proven in \cite{LeGall2013} and \cite{Miermont2013}. It is also the scaling limit of uniform \textit{simple} quadrangulations with $n$ faces, which was proven in \cite{add-alb}. The latter corresponds informally to the case $u\rightarrow0$.

We emphasize that those results, and especially the one of \cite{add-alb}, serve as an input in our proof and we do not provide a new proof of them. Accordingly, let us precisely state the latter result, so that we can use it in the subsequent proof.

\begin{proposition}[\cite{add-alb}]\label{scaling-lim-quad}\label{add-alb}
	Uniform simple quadrangulations with $k$ faces admit the Brownian sphere as scaling limit, with the following normalization
	\[
	\(\frac{3}{8k}\)^{1/4}\cdot B_k^{\mathrm{quad}}
	\xrightarrow[k\rightarrow\infty]{\quad(d), \mathrm{GHP}\quad} {\mathcal S}.
	\]
\end{proposition}

%\begin{proof}
	This is precisely the result \cite[Theorem 1.1]{add-alb}, restricted to the case of simple quadrangulations. Notice that in their result, the scaling limit is stated in terms of $M_n$, a uniform simple quadrangulation with $n$ vertices, not faces. This is not a problem since by Euler's formula, a quadrangulation has $n$ vertices if and only if it has $n-2$ faces. Therefore $B_k^{\mathrm{quad}}$ has the same law as $M_{k+2}$.
%\end{proof}

%\begin{rem}
	Note that \cref{scaling-limit-subcritical} only deals with the quadrangulation $\underline{\Q}_{n,u}$, but not the map $\underline{\M}_{n,u}$. Let us detail what would be needed to obtain a similar statement for $\underline\M_{n,u}$.
	\begin{itemize}
		\item To obtain a Gromov-Hausdorff scaling limit, the missing ingredient is the equivalent for $2$-connected maps of the result of \cite{add-alb}, that is to say GH(P) convergence of uniform 2-connected maps with $n$ edges, rescaled by a constant times $n^{-1/4}$, to the Brownian sphere.
		\item In order to strengthen this to GHP convergence when the map is equipped with the uniform measure on vertices, one would need the above mentioned convergence of 2-connected maps, but in the GHP sense. It would also require a way to compare, in the Prokhorov sense, the degree-biased measure on vertices of $\underline\M_{n,u}$, and the uniform measure. For quadrangulations on the other hand, this comparison can be done using \cite[Lemma 5.1]{quadrang-coeur-cv}.
	\end{itemize}
%\end{rem}

The paper \cite{quadrang-coeur-cv} makes precise the relationship between the convergence of \textbf{uniform} quadrangulations with $n$ faces \cite{LeGall2013,Miermont2013}, and the convergence of \textbf{simple uniform} quadrangulations with $n$ faces \cite{add-alb}. It is shown that a quadrangulation sampled uniformly among those which have size $n$ and whose biggest block has size $k(n)\sim cn$ with an adequate $c>0$, converges jointly with said biggest block to the Brownian sphere, in the GHP sense.

The proof of Gromov-Hausdorff convergence for these quadrangulations amounts to showing that pendant submaps that are grafted on the macroscopic block have negligible diameter, that is $o\(n^{1/4}\)$, which is done by \cite[Proposition 1.12]{quadrang-coeur-cv}. The strategy of proof is not directly applicable here, since it uses an \emph{a priori} diameter bound on the pendant submaps, which we do not have for general $u$.
As explained in what follows, it is sufficient to have an \emph{a priori} diameter bound on single blocks themselves, which is why we need \cref{deviation-diam-blocks}. To strengthen GH convergence to GHP convergence however, we use the same arguments as those exposed in \cite{quadrang-coeur-cv} modulo some technical details.

\subsubsection{Sketch of the proof}

On the combinatorics side, \cref{souscritique} characterizes the phase $u<u_C$ by a condensation phenomenon: when $n$ is large, there is precisely one block of linear size, while others have size $O(n^{2/3})$. This theorem is stated for a map with law $\mathbb P_{n,u}$, that is the law of $\M_{n,u}$, but by \cref{subsec:conseq-proba}, Tutte's bijection commutes with the block decomposition, so that the same happens for $\Q_{n,u}$.

On the metric side, there is not much more going on. The block-tree is subcritical in this phase by \cref{tree-critical} and therefore has small height. Combining this with the $O(n^{2/3})$ bound on the size of non-macroscopic blocks, and the deviation estimate of \cref{deviation-diam-blocks} on diameters of blocks, we get that $\Q_{n,u}$ is approximately equal to its largest block, in the Gromov-Hausdorff sense in the scale $n^{1/4}$. This argument is rather general and should be easy to adapt to other models of graphs or maps with a block-tree decomposition under a condensation regime.

In order to strengthen this convergence to one in the Gromov-Hausdorff-Prokhorov sense, we use the rather general result \cite[Corollary 7.2]{quadrang-coeur-cv}, by comparing the mass measure on vertices with a projection on the macroscopic block, which is modulo some technical details an exchangeable vector on the edges where the pendant submaps are attached. This corollary tells that this random measure is well-approximated by its expectation, which is uniform on the edges of the macroscopic block, or equivalently that it is degree-biased on its vertices. The last part of the argument is specific to quadrangulations, for which we can compare the degree-biased and the uniform measure on vertices by \cite[Lemma 6.1]{quadrang-coeur-cv}.

\subsubsection{Comparison of a quadrangulation and its biggest block}

Let us introduce some notation.
Let $v^\circ$ be the vertex of $\t$ with largest outdegree, choosing one arbitrarily if there are several, and let $\q^\circ=\b_{v^\circ}^\q$.
Denote by $\t[v]$ the subtree of descendants of a node $v$ in $\t$, rooted at $v$. For an edge $e$ of $\q^\circ=\b_{v^\circ}^\q$, the block-tree decomposition associates to it a vertex $v$, so that we can denote by $\q[e]$ the quadrangulation whose block-tree decomposition is $(\b^\q_w)_{w\in\t[v]}$. Recall that by convention, if $v$ is a leaf then $\q[e]$ is the edge map, with $2$ vertices and $1$ edge, the edge $e$.
Write also $\q^+$ for the quadrangulation whose block decomposition is $(\b^\q_v)_{v\in\t[v^\circ]}$. In particular, $\q^\circ$ is the simple core of $\q^+$, and its other blocks are the blocks of the pendant subquadrangulations $(\q[e])_{e\in E(\q^\circ)}$.
Finally, let $\pi^{\q^+}_{\q^\circ}$ be the probability measure on vertices of $\q^\circ$ obtained by projection of the contribution to $\nu_\q$ of each pendant map $(\q[e])_{e\in E(\q^\circ)}$ to the biggest block $\q^\circ$. More formally, for each edge $e$ of $\q^\circ$, let $\{e^+,e^-\}$ be its extremities. Then,
\[
\pi^{\q^+}_{\q^\circ}
%\mathcal Z^{-1}\sum_{e\in E(\q^\circ)}\Bigl|\widetilde\#_{\q[e]}\Bigr|\bigl(\tfrac{1}{2}\delta_{e^-}+\tfrac{1}{2}\delta_{e^-}\bigr)
=\frac 1{|V(\q^+)|-|V(\q^\circ)|}\sum_{e\in E(\q^\circ)}\left(\bigl|V(\q[e])\bigr|-2\right)\bigl(\tfrac{1}{2}\delta_{e^-}+\tfrac{1}{2}\delta_{e^+}\bigr).
\]
Observe that  $V(\q^\circ)$ shares exactly two elements with each $(V(\q[e]))_{e\in E(\q^\circ)}$, when those vertex-sets are naturally embedded in $V(\q)$. Hence the last display indeed defines a probability measure.

\begin{lemma}\label{lemma-comp-GHP-core}
	For any $\varepsilon > 0$, it holds that
	\[
	d_{\mathrm{GHP}}(\epsilon\cdot\underline\q,\epsilon\cdot\underline\q^\circ)
	\leq 2r_{\mathrm{GH}}+r_{\mathrm {P}}+
	\bigl(1-\tfrac{|V(\q^\circ)|}{|V(\q^+)|}\bigr)
	d_{\mathrm P}^{(V(\q^\circ),\epsilon d_\q)}\Bigl({\pi^{\q^+}_{\q^\circ}},\nu_{\q^\circ}\Bigr),
	\]
	where
	\begin{align*}
	r_{\mathrm{GH}}
		=2\epsilon H(\t)\, \max_{v\neq v^\circ}\diam(\b_v^\q)
		\quad\text{ and }\quad	
	r_{\mathrm{P}}
		=\frac{2\bigl|V(\q)\setminus V(\q^+)\bigr|}{|V(\q)|}.
%	r_{\mathrm P,2}
%%		&:=\frac{4}{|V(\q^\circ)|},\\	
%%	r_{\mathrm P,3}
%		&:=
%	\max(\epsilon,1/|V(\q^\circ)|),
	\end{align*}
\end{lemma}

\begin{proof}
	There are successive comparisons to be made for the GHP distance.
	
	\paragraph{Metric comparison.} The term $r_{\mathrm{GH}}$ bounds how distant the spaces $\epsilon\cdot\q$, $\epsilon\cdot\q^+$ and $\epsilon\cdot\q^\circ$ are, from a metric point of view, \textit{i.e.}~in the GH sense. Recall that we can see $\q$ and $\q^+$ as their biggest block $\q^\circ$, together with some maps attached to it. Therefore one needs to bound the maximal diameter of the attached maps. We use a brutal bound on the diameter of the non-macroscopic blocks by their maximal diameter, together with a bound on the number of consecutive blocks in the attached maps. This number is bounded by $\mathrm{diam}(\t)\leq 2H(\t)$. Therefore the maximal diameter of attached maps in $\epsilon\cdot\q$ or $\epsilon\cdot\q^+$ is bounded by
	\[
	r_{\mathrm{GH}}:=2\epsilon H(\t)\, \max_{v\neq v^\circ}\diam(\b_v^\q).
	\]
	In particular, take the correspondence $B_1$ on $V(\q)\times V(\q^+)$ such that $x\in V(\q)$ is in correspondence with only itself if it belongs to $V(\q^+)$, or otherwise with both endpoints of the root-edge of $\q^+$ if it belongs to $V(\q)\setminus V(\q^+)$. The uniform measure on $B_1$ is a coupling between $\nu_\q$ and some measure $\mu^+$ on $V(\q^+)$. One therefore gets, using the triangle inequality and \cref{eq-Strassen},
	\begin{align}\label{eq:prokh-quad-1}
	d_{\mathrm{GHP}}(\epsilon\cdot\underline\q,\epsilon\cdot\underline\q^+)\leq r_{\mathrm{GH}}+d_{\mathrm P}^{(V(\q^+),\epsilon d_\q)}(\mu^+,\nu_{\q^+}).
	\end{align}
	Similarly, take the correspondence $B_2$ on $V(\q^+)\times V(\q^\circ)$ such that $x\in V(\q^+)$ is in correspondence with only itself if it belongs to $V(\q^\circ)$, or otherwise with both endpoints $\{e^+,e^-\}$ of the root-edge of $\q[e]$ if $x$ belongs to $V(\q[e])\setminus \{e^+,e^-\}$ for some edge $e\in E(\q^\circ)$. Then the uniform measure on $B_2$ is a coupling between $\nu_{\q^+}$ and some measure $\mu^\circ$ on $V(\q^+)$. We get as above
	\begin{align}\label{eq:prokh-quad-2}
	d_{\mathrm{GHP}}(\epsilon\cdot\underline\q^+,\epsilon\cdot\underline\q^\circ)\leq r_{\mathrm{GH}}+d_{\mathrm P}^{(V(\q^\circ),\epsilon d_\q)}(\mu^\circ,\nu_{\q^\circ}).
	\end{align}

	\paragraph{Comparing the uniform measures on vertices of $\q$ and $\q^+$.}
	Observe that $\nu_{\q^+}$ is the counting measure on $V(\q^+)$ renormalized to a probability distribution, while $\mu^+$ is the renormalized version of the same counting measure but with additional mass
	\[
	m:={|V(\q)\setminus V(\q^+)|-2},
	\]
	the latter being split equally on the endpoints of the root-edge of $\q^+$. Elementarily, this yields a total variation bound, as follows
	\[
	d_{\mathrm{TV}}(\mu^+,\nu_{\q^+})
		\leq \frac{2 m}{V(\q)}\leq \frac{2|V(\q)\setminus V(\q^+)|}{V(\q)}=:r_{\mathrm P}.
	\]
	Since the Prokhorov distance is bounded by the total variation distance, we have
	\begin{align}\label{eq:prokh-quad-3}
	d_{\mathrm P}^{(V(\q^+),\epsilon d_\q)}(\mu^+,\nu_{\q^+})\leq r_{\mathrm P}.
	\end{align}
	
	\paragraph{Comparing the uniform measures on vertices of $\q^+$ and $\q^\circ$.}
	From the definition of $\pi^{\q^+}_{\q^\circ}$ and from the following partitioning, under the natural embedding of the vertex-sets in $V(\q)$,
	\[
	V(\q^+)=V(\q^\circ)\bigsqcup_{e\in E(\q^\circ)}V(\q[e])\setminus\{e^+,e^-\},
	\]
	observe that the measure $\mu^\circ$ obtained from the correspondence $B_2$ above decomposes as follows
	\[
	\mu^\circ=\tfrac{|V(\q^\circ)|}{|V(\q^+)|}\nu_{\q^\circ}+
	\tfrac{|V(\q^+)|-|V(\q^\circ)|}{|V(\q^+)|}\pi^{\q^+}_{\q^\circ}.
	\]
	In particular, we obtain from \cref{eq-decomp-Prokhorov} that
	\begin{align}\label{eq:prokh-quad-4}
		d_{\mathrm P}^{(V(\q^\circ),\epsilon d_\q)}(\mu^\circ,\nu_{\q^\circ})
			=\Bigl(1-\tfrac{|V(\q^\circ)|}{|V(\q^+)|}\Bigr)\,
			d_{\mathrm P}^{(V(\q^\circ),\epsilon d_\q)}\Bigl({\pi^{\q^+}_{\q^\circ}},\nu_{\q^\circ}\Bigr).
	\end{align}
	\paragraph{Concluding the proof.}
	By the triangle inequality, we have
	\[
		d_{\mathrm{GHP}}(\epsilon\cdot\underline\q,\epsilon\cdot\underline\q^\circ)
		\leq 	d_{\mathrm{GHP}}(\epsilon\cdot\underline\q,\epsilon\cdot\underline\q^+)
		+ 	d_{\mathrm{GHP}}(\epsilon\cdot\underline\q^+,\epsilon\cdot\underline\q^\circ).
	\]
	Using \cref{eq:prokh-quad-1} and \cref{eq:prokh-quad-3} to bound the first term, and \cref{eq:prokh-quad-2} and \cref{eq:prokh-quad-4} to bound the second one, we get the claimed inequality.
	
\end{proof}

\subsubsection{Exchangeable decorations}

We aim to use Addario-Berry \& Wen's argument for \cite[Lemma 6.2]{quadrang-coeur-cv} which tells that for exchangeable attachments of mass on edges of $Q_n$, a quadrangulation with $n$ faces sampled uniformly, the resulting measure on $Q_n$ is asymptotically close to the uniform measure on vertices, in the sense of the Prokhorov distance on $n^{-1/4}\cdot Q_n$. They use the following ingredients:
\begin{enumerate}
	\item A concentration inequality \cite[Lemma 5.2]{quadrang-coeur-cv} which compares the measure with exchangeables attachments of mass on edges, to the degree-biased measure on vertices.
	\item A Prokhorov comparison \cite[Lemma 5.1]{quadrang-coeur-cv} between the degree-biased and uniform measure on vertices of a {quadrangulation}.
	\item GHP convergence of $n^{-1/4}\cdot \underline Q_n$ to the Brownian sphere.
	\item Properties of the Brownian sphere such as compacity and re-rooting invariance.
\end{enumerate}
The first ingredient is rather general and actually stated for any graph in \cite[Lemma 5.3]{quadrang-coeur-cv}. We will ever-so-slightly adapt its proof since there is a double edge in their setting which we do not have, and the mass is not projected on vertices in the exact same way.
The second ingredient is specific to quadrangulations and one may need different arguments to compare the degree-biased and uniform measures for other classes of maps.

Let us state which result we extract for our purpose from Addario-Berry \& Wen's paper. For $\mathbf n=(\mathbf n(e))_{e\in E(G)}$ a family of nonnegative numbers indexed by edges of a graph $G$, we denote its $p$-norm for $p\geq 1$ by
\[|\mathbf n|_p:=\left(\sum_{e\in E(G)}(\mathbf n(e))^p\right)^{1/p}.\]
Then define the following measure on $V(G)$:
\[
\mu^{\mathbf n}_{G}:=\frac{1}{|\mathbf n|_1}\sum_{e\in E(G)}\mathbf n(e)\bigl(\tfrac{1}{2}\delta_{e^+}+\tfrac{1}{2}\delta_{e^-}\bigr),
\]
with $\{e^+,e^-\}$ the set of endpoints of the edge $e$. Notice that this definition is slightly different from that of $\nu^{\mathbf n}_G$ in \cite[Section 5]{quadrang-coeur-cv}, because the mass of an edge is projected uniformly and independently on either of its enpoints in their case, while we deterministically split this mass on both endpoints. This does not change much except that we find it easier to work with. One of their results translates as the following.
\begin{proposition}[{\cite[Corollary 6.2]{quadrang-coeur-cv}}]\label{deco-echangeables}
	Let $Q_k=B_k^{\mathrm{quad}}$, which is a simple quadrangulation with $k$ faces, sampled uniformly. Consider for each $k\geq 1$, a random family $\mathbf n_k=(\mathbf n_k(e))_{e\in E(Q_k)}$ of nonnegative numbers, such that conditionally on $Q_k$ it is an exchangeable family. Assume that $|\mathbf n_k|_2/|\mathbf n_k|_1\rightarrow 0$ in probability as $k\rightarrow\infty$.
	Then there holds the convergence in probability
	\[
	d_{\mathrm P}^{(V(Q_k),\epsilon_k d_{Q_k})}\Bigl(\mu^{\mathbf n_k}_{B_k},\nu_{Q_k}\Bigr)\xrightarrow[k\rightarrow\infty]{\mathbb P}0,
	\]
	where $\nu_{Q_k}$ is the uniform measure on vertices of $Q_k$ and $\epsilon_k=k^{-1/4}$.
\end{proposition}

%\begin{proof}
	This is the statement of \cite[Corollary 6.2]{quadrang-coeur-cv}, adapted to our setting. The proof goes \textit{mutatis mutandi}, except for an adjustement in the concentration inequality \cite[Lemma 5.3]{quadrang-coeur-cv}, which we adapt below in \cref{lemme-concentration-deco-echangeable}.
%\end{proof}

\begin{lemma}[{\cite[Lemma 5.3]{quadrang-coeur-cv}}]\label{lemme-concentration-deco-echangeable}
	Let $G$ be a graph and $\mathbf n=(\mathbf n(e))_{e\in G}$ a random and exchangeable family of nonnegative numbers with $|\mathbf n|_2>0$ almost surely. Then for any $V\subset V(G)$, and any $t>0$,
	\[
	\mathbb P\left(\Bigl|\mu_G^{\mathbf n}(V)-\nu_G(V)\Bigr|>\frac{2t}{|\mathbf n|_1}\biggm||\mathbf n|_2\right)\leq 2\exp\left(-\frac{2t^2}{|\mathbf n|^2_2}\right).
	\]
\end{lemma}

%\begin{proof}

	The proof goes the same way as that of \cite[Lemma 5.3]{quadrang-coeur-cv}, except that we do not have a double edge here, and the mass on edges is projected deterministically on vertices in our case, instead of randomly. The reader may notice that there is an extra term inside the probability in their lemma. This term accounts for the double edge, which we do not have here. The same line of arguments still works though. Indeed, we have
	\[
	\mu_G^{\mathbf n}(V)=\sum_{e\in E(G[V])}\frac{\mathbf n(e)}{|\mathbf n|_1}+\frac{1}{2}\sum_{e\in \partial_e V}\frac{\mathbf n(e)}{|\mathbf n|_1},
	\]
	with $G[V]$ the induced-graph on $V$ by $G$, and $\partial_e V$ the subset of the edges of $V$ who have only one endpoint which belongs to $V$. By exchangeability, we have the expectation
	\[
	\mathbb E\left[\sum_{e\in E(G[V])}{\mathbf n(e)}+\frac{1}{2}\sum_{e\in \partial_e V}{\mathbf n(e)}\Bigm||\mathbf n|_1\right]=|\mathbf n|_1\frac{|E(G[V])|}{|E(G)|}+|\mathbf n|_1\frac{\frac12|\partial_e V|}{|E(G)|}=|\mathbf n|_1 \nu_G(V).
	\]
	The last equality holds because the degree biased-measure counts twice each edge of $G[V]$, since this edge appears in the degree of both its endpoints, while the edges of $\partial_e V$ are only counted once, in the degree on the only one of its endpoints which is in $V$.
	
	Then one concludes as in the proof of \cite[Lemma 5.3]{quadrang-coeur-cv}, by a Hoeffding-type bound for exchangeable vectors.
%\end{proof}

\subsubsection{Proof of Theorem \ref{scaling-limit-subcritical}}

%\begin{proof}[Proof of \cref{scaling-limit-subcritical}]
	\paragraph{Scaling limit of the biggest block.}
	By \cref{prop-t-m-n} and \cref{bij-quad-maps}, the biggest block of $\Q_{n,u}$, whose size we denote $C(n,u)$, is a simple quadrangulation sampled uniformly with size $C(n,u)$. Also by \cref{souscritique}, this size is asymptotically in probability,
	\[
	C(n,u)=(1-E(u))n+O_{\mathbb P}(n^{2/3})=\frac{9-5u}{3(3+u)}n+O_{\mathbb P}(n^{2/3}).
	\]
	By conditioning on $C(n,u)$ and using \cref{scaling-lim-quad}, we therefore get the following GHP scaling limit for the biggest block
	\[
	\(\frac{3}{8C(n,u)}\)^{1/4}\cdot \underline\Q_{n,u}^\circ
	\xrightarrow[n\rightarrow\infty]{\quad(d), \mathrm{GHP}\quad} {\mathcal S},
	\]
	which by the preceding equivalent in probability for $C(n,u)$ reduces to
	\[
	\(\frac{9(3+u)}{8(9-5u)}\)^{1/4}n^{-1/4}\cdot \underline\Q_{n,u}^\circ
	\xrightarrow[n\rightarrow\infty]{\quad(d), \mathrm{GHP}\quad} {\mathcal S}.
	\]
	
	\paragraph{GHP comparison of $\Q_{n,u}$ with its biggest block.}
	By the preceding scaling limit, and the use of \cref{lemma-comp-GHP-core} with $\q=\Q_{n,u}$ and $\epsilon=n^{-1/4}$, the proof of the theorem reduces to showing the convergence to $0$ in probability of the following quantities
	\begin{align*}
		r_{\mathrm{GH}}
			&:=\frac{2}{n^{1/4}} H(\T_{n,u})\, \max_{v\neq v^\circ}\diam(\b_v^{\Q_{n,u}})	\\
		r_{\mathrm{P}}
			&:=\frac{2\bigl|V(\Q_{n,u})\setminus V(\Q_{n,u}^+)\bigr|}{|V(\Q_{n,u})|}	\\
		d_P
			&:=d_{\mathrm P}^{(V(\Q_{n,u}^\circ),\epsilon_n d_{\Q_{n,u}})}\Bigl({\pi^{\Q_{n,u}^+}_{\Q_{n,u}^\circ}},
%			\bigl(1-\tfrac{|V(\Q_{n,u}^\circ)|}{|V(\Q_{n,u})|}\bigr)
			\nu_{\Q_{n,u}^\circ}\Bigr),
	\end{align*}
	where $\epsilon_n=n^{-1/4}$.
	
	\paragraph{Bounding $r_{\mathrm{GH}}$.}
	By \cref{souscritique}, the second-biggest block of $\Q_{n,u}$ has size $O(n^{2/3})$ in probability. Combining this with \cref{diam-blocs}, one gets for all $\delta>0$ the bound in probability
	\[
	\max_{v\neq v^\circ}\diam(\b_v^{\Q_{n,u}})=o\(n^{(1+\delta)/6}\).
	\]
	Also, by \cref{tree-critical}, $\T_{n,u}$ is a \textit{non-generic subcritical} Galton-Watson tree conditioned to have $2n+1$ vertices, in the terminology of \cite{kortchemskiLimitTheoremsConditioned2015}. We may therefore use \cite[Theorem 4]{kortchemskiLimitTheoremsConditioned2015} to get for all $\delta>0$ the bound in probability
	\[
	H(\T_{n,u})=o\(n^\delta\).
	\]
	Combining the two preceding estimates, we get in probability
	\[
	r_{\mathrm{GH}}=o\(n^{-\tfrac14+\delta+\tfrac{(1+\delta)}6}\)\xrightarrow[n\rightarrow\infty]{}0,
	\]
	provided that we chose $\delta>0$ small enough so that $\delta+{(1+\delta)}/6<1/4$.
	
	\paragraph{Bounding $r_{\mathrm P}$.}
	First, notice that since $\Q_{n,u}$ is a quadrangulation we have
	\[
	|V(\Q_{n,u})|=|E(\Q_{n,u})|=2n,
	\]
	and by the block-tree decomposition which puts in correspondence edges of $\Q_{n,u}$ and edges of $\T_{n,u}$, we also have
	\[
	\bigl|V(\Q_{n,u})\setminus V(\Q_{n,u}^+)\bigr|=\bigl|E(\Q_{n,u})\bigr|- \bigl|E(\Q_{n,u}^+)\bigr|= \bigl|E(\T_{n,u})\bigr|- \bigl|E(\T_{n,u}^+)\bigr|=\bigl|E(\T_{n,u}\setminus\T_{n,u}^+)\bigr|.
	\]
	Therefore we have to bound the size of the subtree $\T_{n,u}\setminus\T_{n,u}^+$. A moment of thought shows that it is bounded by
	\[
	U_{\rightarrow}(\T_{n,u})+U_{\leftarrow}(\T_{n,u}),
	\]
	where $U_{\rightarrow}(\t)$ is the index in lexicographical order of the vertex with largest degree of the tree $\t$, and $U_{\leftarrow}(\t)$ is the index in reverse lexicographical order of that same vertex. Now, \cite[Theorem 2]{kortchemskiLimitTheoremsConditioned2015} shows that $(U_{\rightarrow}(\T_{n,u}))_{n\geq 1}$ is a tight sequence. Since $U_{\leftarrow}(\T_{n,u})$ has the same law as $U_{\rightarrow}(\T_{n,u})$, the respective sequence is also tight. All in all, we get in probability
	\[
	r_{\mathbf P}=O(1/n)\xrightarrow[n\rightarrow\infty]{}0.
	\]
	
	\paragraph{Bounding $d_{\mathrm P}$.}
	Notice that
	\begin{align*}
		d_P
		=d_{\mathrm P}^{(V(\Q_{n,u}^\circ),\epsilon_n d_{\Q_{n,u}})}\Bigl({\pi^{\Q_{n,u}^+}_{\Q_{n,u}^\circ}},
%		\bigl(1-\tfrac{|V(\Q_{n,u}^\circ)|}{|V(\Q_{n,u})|}\bigr)
		\nu_{\Q_{n,u}^\circ}\Bigr)
		=
% \bigl(1-\tfrac{|V(\Q_{n,u}^\circ)|}{|V(\Q_{n,u})|}\bigr)
		d_{\mathrm P}^{(V(\Q_{n,u}^\circ),\epsilon_n d_{\Q_{n,u}})}\Bigl(\mu^{\mathbf n}_{\Q_{n,u}},\nu_{\Q_{n,u}^\circ}\Bigr),
	\end{align*}
	where $\mathbf n=\mathbf n_{n,u}$ is the family of nonnegative numbers defined by
	\[
	\forall e\in E(\Q_{n,u}^\circ),\quad\mathbf n(e)=|V(\Q_{n,u}[e])|-2.
	\]
	Let us argue that conditionally on $\Q_{n,u}^\circ$, this family $\mathbf n$ is exchangeable. Recall that $\Q_{n,u}$ has the law of $\Q$ under $\mathbb P_u$, conditioned to the event $\{|\Q|=n\}$. By the symmetries of the Galton-Watson law and \cref{th-trees}, the family
	\[
	\left(|V(\Q[e])|-2\right)_{e\in E(\Q^\circ)}
	=\left(|E(\T[v_e])|-2\right)_{e\in E(\Q^\circ)}
	\]
	is i.i.d. conditionally on $\Q^\circ$, where $v_e$ is the child of $v^\circ$ that the block-tree decomposition associates to $e$. In particular, this family is exchangeable. Since the event $\{|\Q|=n\}$ is invariant by each permutation of the subtrees attached to the node $v^\circ$ with their respective blocks, the above family stays exchangeable when conditioning by this event. Therefore $\mathbf n$ is indeed exchangeable.
	
	Now, \cite[Corollary 1]{kortchemskiLimitTheoremsConditioned2015} tells that the subtrees $(\T[v_e])_{e\in \Q_{n,u}^\circ}$ have size $O(n^{2/3})$ in probability, uniformly in the edge $e$. We thus get that
		\[
		|\mathbf n|_2=O \(\sqrt{n^{5/3}}\).
		\]
	On the other hand, we have in probability
	\[
	|\mathbf n|_1=|V(\Q_{n,u})\setminus V(\Q_{n,u}^+)|\sim cn,
	\]
	for some constant $c>0$.
	Hence, in probability
	\[
	\frac{|\mathbf n|_2}{|\mathbf n|_1}=O\(n^{-1/6}\)\xrightarrow[n\rightarrow\infty]{}0.
	\]
	All the hypotheses of \cref{deco-echangeables} have been checked, so that we may apply it, after conditioning by the size of $\Q_{n,u}^\circ$, since conditionally on its size $k$ it is a uniform simple quadrangulation of size $k$. We obtain in probability
	\[
	d_{\mathrm P}^{(V(\Q_{n,u}^\circ),\epsilon_n d_{\Q_{n,u}})}\Bigl(\mu^{\mathbf n}_{\Q_{n,u}},\nu_{\Q_{n,u}^\circ}\Bigr)
	\xrightarrow[n\rightarrow\infty]{}0.
	\]
	Hence, $d_{\mathrm P}$ also tends to $0$ in probability and this concludes the proof.
	\qed

% !TEX root = FleuratSalvy.tex

\section{Concluding remarks and perspectives}
\label{conclusion}

%%%%%%%%%%%%%%%%%%% Rappel des résultats %%%%%%%%%%%%%%%%%%%
We have exhibited a phase transition phenomenon for two closely related models of random maps with a weight $u>0$ per block. The phase transition occurs at $u=9/5$, and we have established the existence of three regimes, regarding the size of the largest block, and regarding the scaling limit (and the order of magnitude of distances).

%%%%%%%%%%%%%%%%%%% Extension à d'autres modèles %%%%%%%%%%%%%%%%%%%
\paragraph{Extension to other models.} Our method can be generalised to other models which can be decomposed into appropriate blocks with an underlying tree structure, for example the models described in \cite[Table 3]{airy}, which is partially reproduced in \cref{repro-airy-table-3}.
 %Recall the definition of \emph{bipartite} from \cref{subseq:quad}.
 A \emph{triangulation} is a map where all faces have degree $3$. It is \emph{irreducible} if every $3$-cycle defines a face. In this section, we use the same notation for the various models as in the rest of the article.

Models described in \cite[Table 3]{airy} where maps are decomposed into blocks weighted with a weight $u>0$ undergo a phase transition at the critical value $u_C$ written down in \cref{tab:values-models-airy}. More precisely, \crefrange{tree-critical}{size-bloc-critical} hold for these models with the constants of \cref{tab:values-models-airy}. Notice that for the decomposition of general maps into $2$-connected maps (\emph{i.e.} the schema linking $\mathcal{M}_1$ and $\mathcal{M}_4$)~---~which is the case studied in this paper~---~we get results consistent with \cref{tree-critical}. Moreover, the values of $u_C$ and $E(u)$ are consistent since it always holds that $E(u_C) = 1$. Furthermore, for $u=1$, we retrieve the results of \cite[Table 4]{airy}: indeed, our $1-E(1)$ equals their $\alpha_0$\footnote{It is not obvious at first glance that this should be the case for the case of simple triangulations decomposed into irreducible cores, because each node of the Galton-Watson tree corresponds to a sequence of blocks. However, an extreme condensation phenomenon occurs and the mass is concentrated in only one element of the sequence, so the behaviour remains similar.}.

Models from \cite[Table 3]{airy} are amenable to computations similar to this article's in order to get the values above. We show in \cref{tab:values-models-airy} the most obvious results and models requiring more care will be described in a separate note. In the cases of \cref{repro-airy-table-3}, there is $d\in\Z_{>0}$ such that $H(z) = z(1+M)^d$, and the corresponding law $\mu^u$ (except for triangulations) comes naturally as:
\[\mu^u(dm) =\frac{\mathbb{1}_{m\ne0} u b_m y(u)^m + \mathbb{1}_{m=0}}{u B(y(u)) + 1 - u},\qquad \mu^u(m) = 0\quad\text{when}\quad d \nmid m.\]
The cases dealing with triangulations require more care as the series are counted by vertices but the substitution is done on edges in one case, and on internal faces in the other; but keeping this in mind, the same methods can be applied.

For all models, we expect to get similar regimes as in \cref{results} (assuming the convergence of the family of blocks is known, as well as diameter estimates). However, conditioning is more difficult for some models, as the size of the map is not always immediately deduced from the size of the Galton Watson tree (\emph{e.g.} for simple triangulations ($\mathcal{T}_2$) decomposed into irreducible triangulations ($\mathcal{T}_3$), the size is the number of leaves of the Galton-Watson tree).

\begin{table}
\begin{center}
\begin{tabular}{llc}
maps, $M(z)$ & cores, $C(z)$ & submaps, $H(z)$\\
\hline
loopless, $M_2(z)$ & simple, $M_3(z)$ & $z(1+M)$ \\
all, $M_1(z)$ & $2$-connected, $M_4(z)$ & $z(1+M)^2$\\
$2$-connected $M_4(z)-z$ & $2$-connected simple, $M_5(z)$ & $z(1+M)$\\
\hline
bipartite, $B_1(z)$ & bipartite simple, $B_2(z)$ & $z(1+M)$\\
bipartite, $B_1(z)$ & bipartite $2$-connected, $B_4(z)$ & $z(1+M)^2$\\
bipartite $2$-connected, $B_4(z)$ & bipartite $2$-connected simple $B_5(z)$ & $z(1+M)$\\
\hline
loopless triangulations, $T_1(z)$ & simple triangulations, $z+zT_2(z)$ & $z(1+M)^3$\\
simple triangulations, $T_2(z)$ & irreducible triangulations, $T_3(z)$ & $z(1+M)^2$\\
\end{tabular}
\caption{Partial reproduction of \cite[Table 3]{airy}, which describes composition schemas of the form $\mathcal{M} = \mathcal{C}\circ\mathcal{H}$ except the last one where $\mathcal{M} = (1+\mathcal{M})\times \mathcal{C}\circ\mathcal{H}$. The parameter $z$ counts vertices (up to a fixed shift) in the case of triangulations, edges otherwise. Some terms have been changed to correspond to the conventions used in this article.}
\label{repro-airy-table-3}
\end{center}
\end{table}

\begin{table}
\begin{center}
\def\arraystretch{1.5}
\begin{tabular}{cc|ccc}
Maps & Cores & $u_C$ & $E(u)$ & $1-E(1)$\\
\hline
$\mathcal{M}_2$ & $\mathcal{M}_3$ & $\frac{81}{17}$ & $\frac{32 u}{3 (5 u+27)}$ & $\frac{2}{3}$\\
$\mathcal{M}_1$ & $\mathcal{M}_4$ & $\frac{9}{5}$ & $\frac{8u}{3(u + 3)}$ & $\frac{1}{3}$\\
$\mathcal{M}_4 - \mathcal{Z}$ & $\mathcal{M}_5$ & $\frac{135}{7}$ & $\frac{32u}{5(5u + 27)}$ & $\frac{4}{5}$\\
\hline
$\mathcal{B}_1$ & $\mathcal{B}_2$ & $\frac{36}{11}$ & $\frac{20u}{9(u + 4)}$ & $\frac{5}{9}$ \\
$\mathcal{B}_1$ & $\mathcal{B}_4$ & $\frac{52}{27}$ & $\frac{40u}{13(u + 4)}$ & $\frac{5}{13}$ \\
$\mathcal{B}_4$ & $\mathcal{B}_5$ & $\frac{68}{3}$ & $\frac{20u}{17(u + 4)}$ & $\frac{13}{17}$\\
\hline
$\mathcal{T}_1$ & $\mathcal{Z}+\mathcal{Z}\times\mathcal{T}_2$ & $\frac{16}{7}$ & $\frac{9u}{2(u + 8)}$ & $\frac{1}{2}$\\
$\mathcal{T}_2$ & $\mathcal{T}_3$ & $\frac{64}{37}$ & $\frac{27u}{2(32 - 5u)}$ & $\frac{1}{2}$\\
\end{tabular}
\caption{Values of $u_C$, $E(u)$ when $u\leq u_C$ and $1-E(1)$ for all the decomposition schemes of \cref{repro-airy-table-3}.}
\label{tab:values-models-airy}
\end{center}
\end{table}

%%%%%%%%%%%%%%%%%%% Perspectives %%%%%%%%%%%%%%%%%%%
\paragraph{Perspectives.} We plan to study similar models in the context of decorated planar maps (\emph{e.g.} tree-rooted maps or Schnyder woods), where the generating series exhibit different singular behaviours. In future work, we also want to investigate more closely the rate of the phase transition at the critical value $u=9/5$, in analogy to the study of the largest component for the Erd\"os-Rényi random graph \cite{erdos-renyi-transition}.

Finally, as mentioned in the introduction, the model of maps with a weight $u$ per 2-connected block has been studied as encoding certain discrete spaces of dimension larger than 2, with motivations from theoretical physics \cite{bonzomLagrange,lionni-phd}. The metric properties are however modified via the correspondence, and it would be interesting to determine if the scaling limits remain the same.

% !TEX root = FleuratSalvy.tex

\begin{figure}
\begin{center}
\includegraphics[width=16.5cm, center]{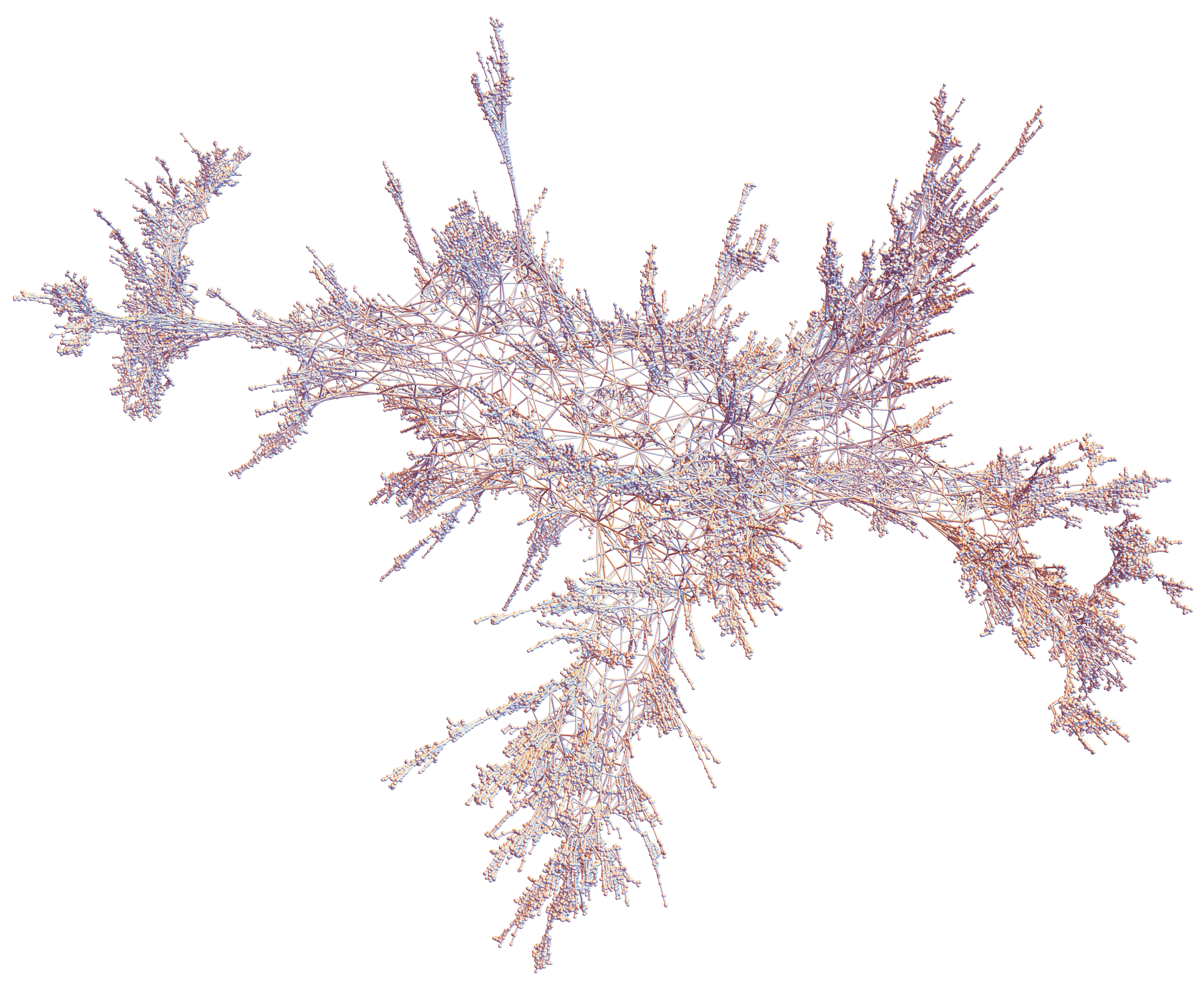}
\end{center}
\caption{Map drawn according to the subcritical model $\mathbb{P}_{n,1}$ of size around $55\ 000$.}
\label{fig:simul-sub-1}
\end{figure}

\begin{figure}
\begin{center}
\includegraphics[width=16.5cm, center]{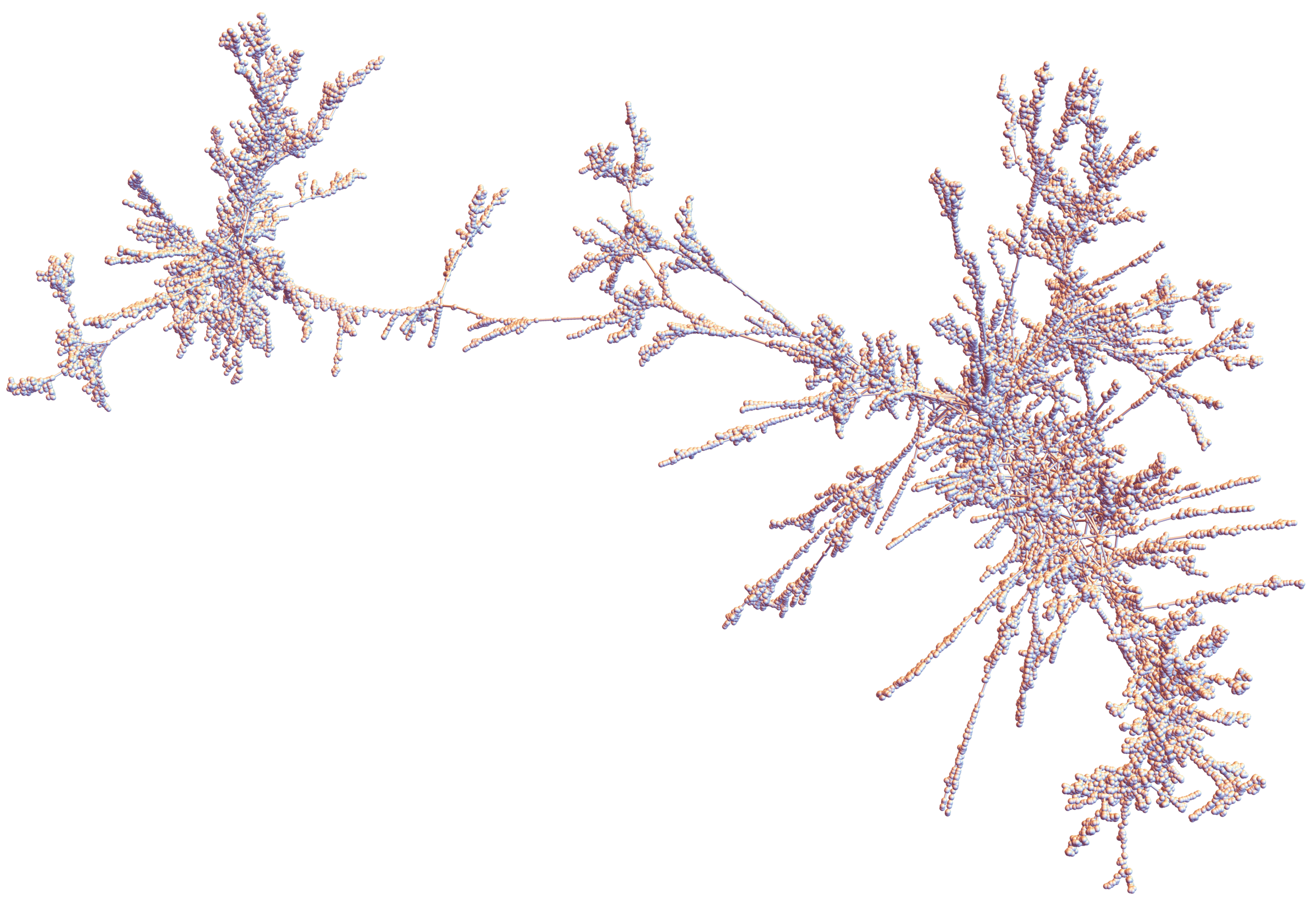}
\end{center}
\caption{Map drawn according to the subcritical model $\mathbb{P}_{n,8/5}$ of size around $55\ 000$.}
\label{fig:simul-sub-8/5}
\end{figure}

\begin{figure}
\begin{center}
\includegraphics[width=16.5cm, center]{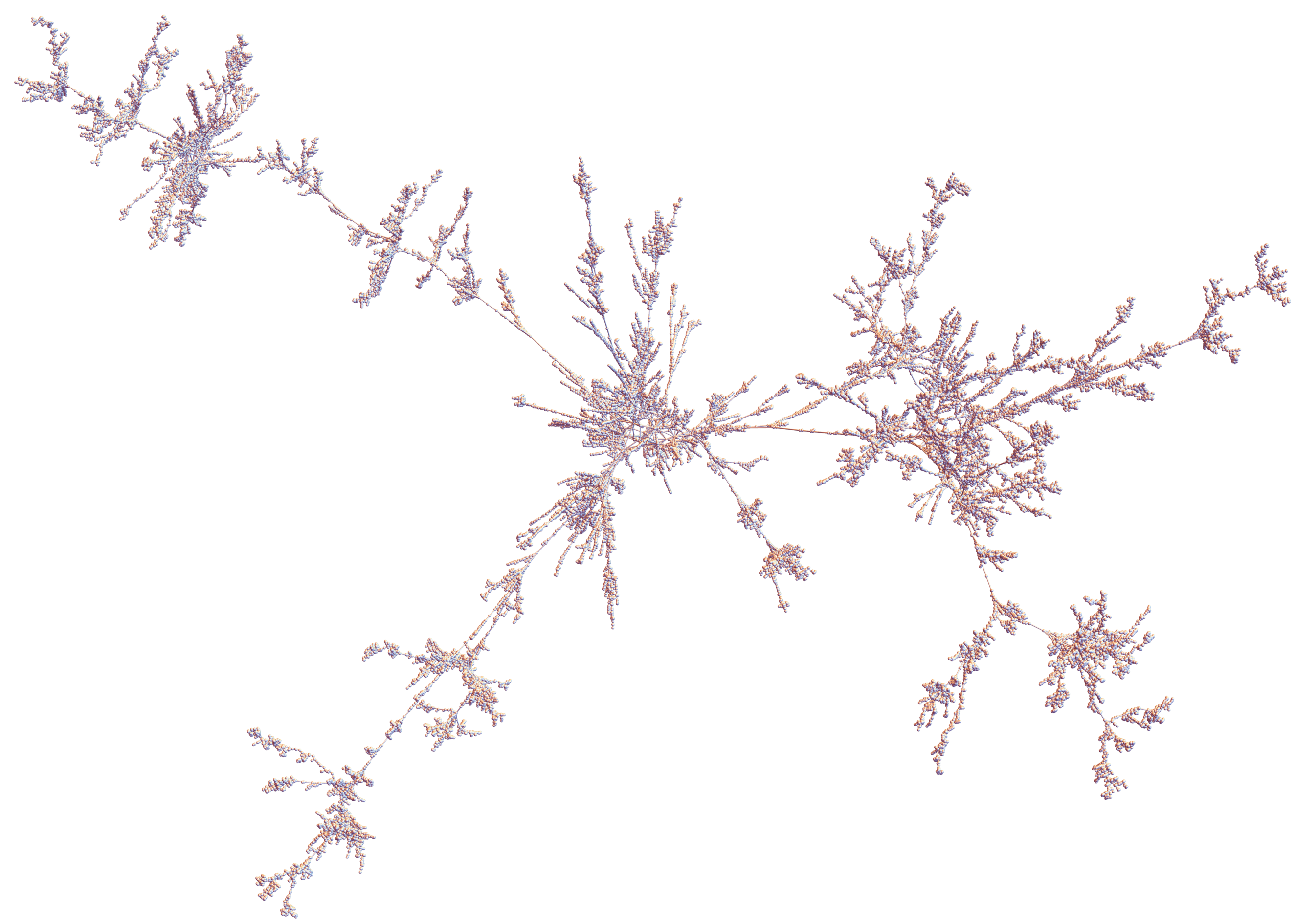}
\end{center}
\caption{Map drawn according to the critical model $\mathbb{P}_{n,9/5}$ of size around $80\ 000$.}
\label{fig:simul-crit-9/5}
\end{figure}

\begin{figure}
\begin{center}
\includegraphics[width=16.5cm, center]{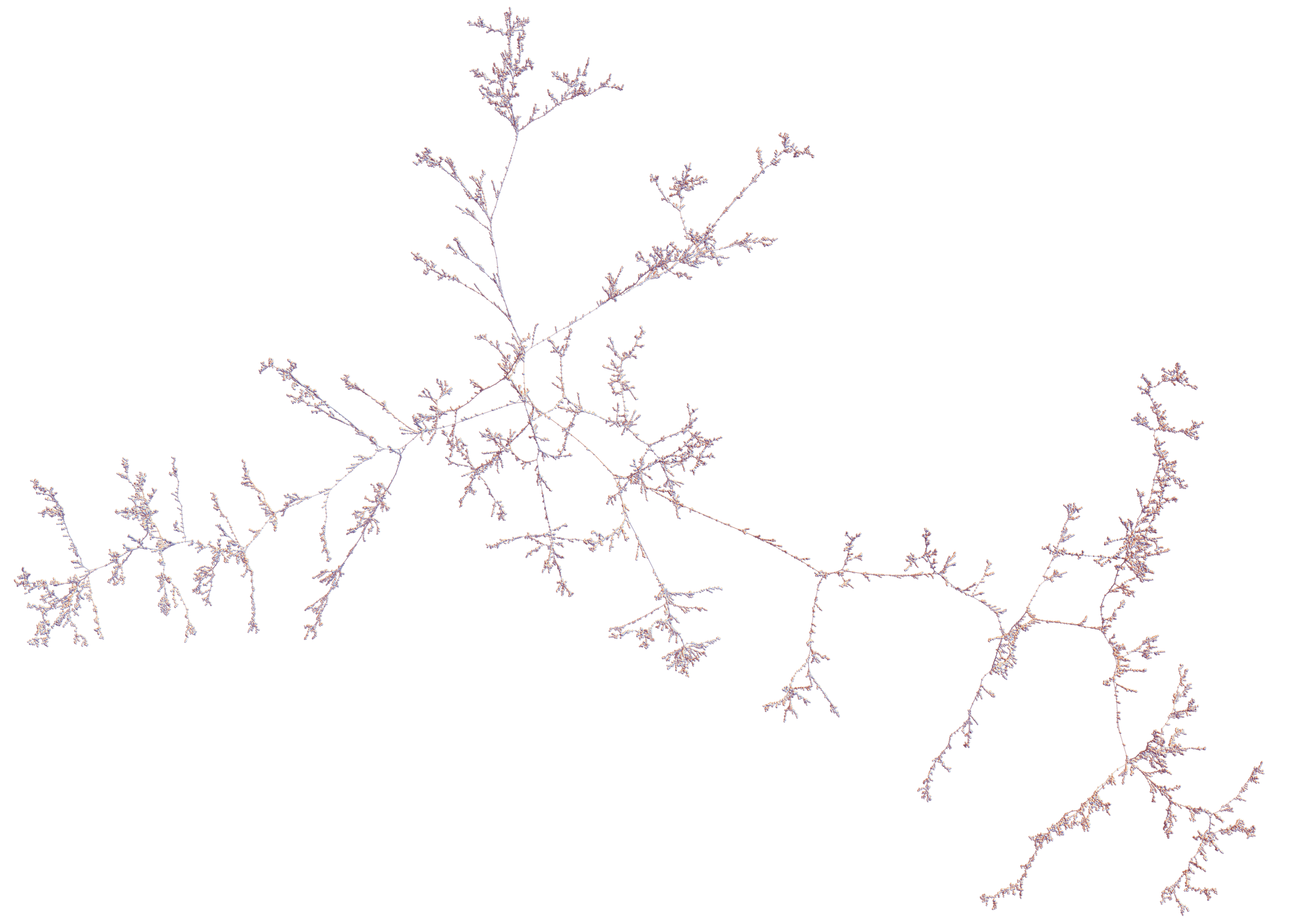}
\end{center}
\caption{Map drawn according to the supercritical model $\mathbb{P}_{n,5/2}$ of size around $75\ 000$.}
\label{fig:simul-super-5/2}
\end{figure}

\begin{figure}
\begin{center}
\includegraphics[width=16.5cm, center]{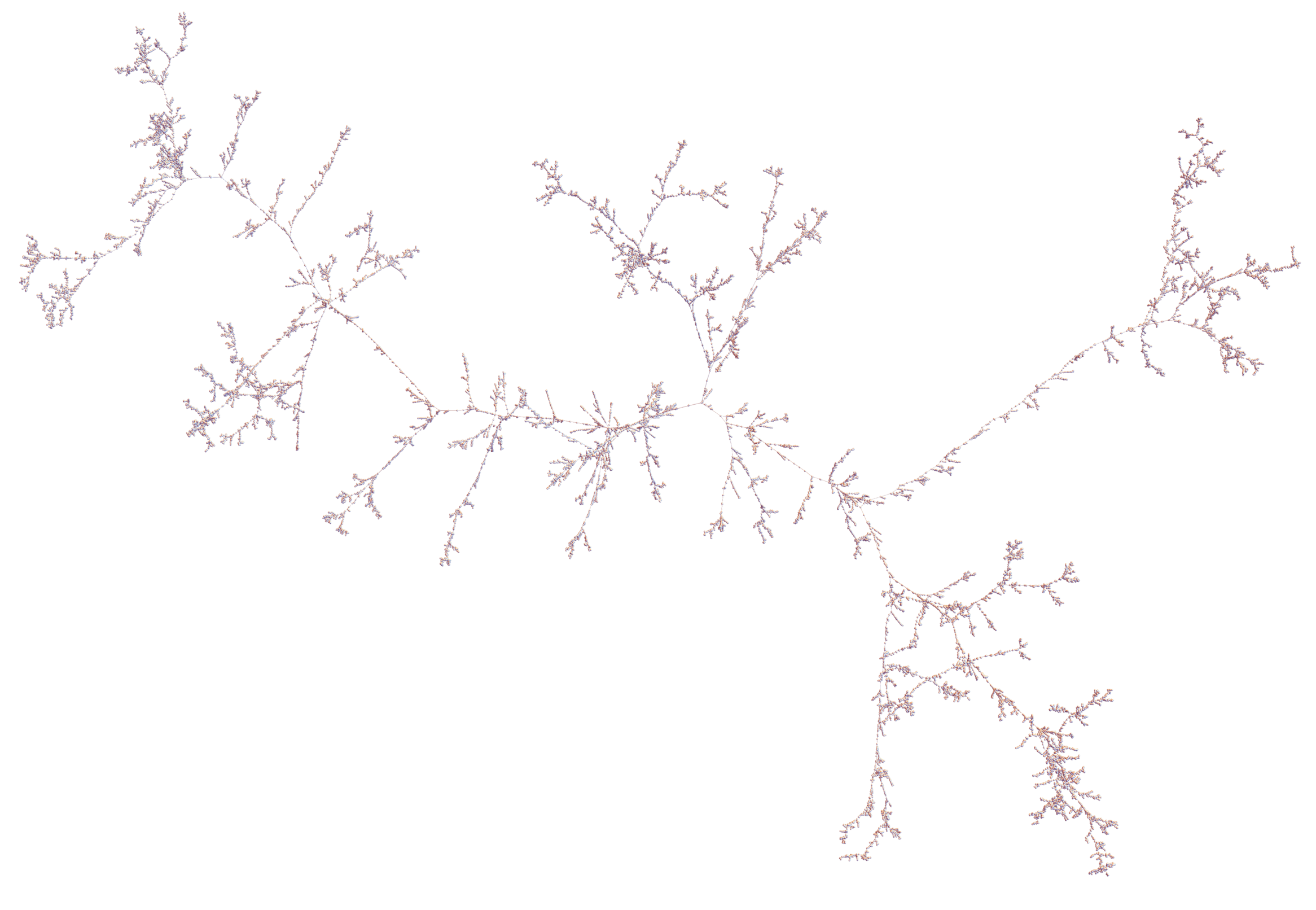}
\end{center}
\caption{Map drawn according to the supercritical model $\mathbb{P}_{n,5}$ of size around $50\ 000$.}
\label{fig:simul-super-5}
\end{figure}

\pagebreak
\bibliographystyle{alpha}
\bibliography{references}

\end{document}